\documentclass{amsart}
\usepackage[margin=1.3in]{geometry}
\usepackage{color}
\usepackage{float}
\usepackage{amsmath}
\usepackage{amssymb}
\usepackage{amsthm}
\usepackage{mathrsfs}
\usepackage{enumerate}
 
\newcommand{\fp}{{\mathfrak p}}
\newcommand{\bfm}{{\mathbf{m}}}

\newcommand{\bfM}{{\mathbf{M}}}
\newcommand{\bfN}{{\mathbf{N}}}
\newcommand{\bfT}{{\mathbf{T}}}

\definecolor{pink}{rgb}{1,.2,.6}
\definecolor{orange}{rgb}{0.7,0.3,0}
\definecolor{blue}{rgb}{.2,.6,.75}
\definecolor{green}{rgb}{.4,.7,.4}
\definecolor{purple}{RGB}{127,0,255}
\definecolor{darkgreen}{rgb}{0.09, 0.45, 0.27}

\newcommand{\rC}[1]{{\color{red}{}}}

\newcommand{\xtra}[1]{}

\newcommand{\exendnote}[1]{}

\numberwithin{equation}{section}

\newtheorem{thm}{Theorem}[section]
\newtheorem*{thm*}{Theorem}
\newtheorem{prop}[thm]{Proposition}
\newtheorem{lemma}[thm]{Lemma}

\newtheorem{cor}[thm]{Corollary}

\theoremstyle{definition}

\newtheorem*{dfn*}{Definition}

\theoremstyle{remark}
\newtheorem{remark}[thm]{Remark}

\newcommand{\A}{\mathbb{A}}
\newcommand{\C}{\mathbb{C}}
\newcommand{\F}{\mathbb{F}}
\newcommand{\Q}{\mathbb{Q}}
\newcommand{\R}{\mathbb{R}}

\newcommand{\Z}{\mathbb{Z}}

\newcommand{\Ocal}{\mathcal{O}}

\newcommand{\Dscr}{\mathscr{D}}

\newcommand{\Pscr}{\mathscr{P}}

\newcommand{\GL}{\mathrm{GL}}

\newcommand{\codim}{\mathrm{codim}}

\renewcommand{\b}[1]{\mathbf{#1}}

\newcommand{\con}{\equiv}

\newcommand{\ndiv}{\nmid}
\newcommand{\modd}[1]{\; ( \text{mod} \; #1)}
\newcommand{\bstack}[2]{\substack{#1 \\ #2}}

\newcommand{\maps}{\rightarrow}
\newcommand{\intersect}{\cap}

\newcommand{\Union}{\bigcup}

\newcommand{\supp}{{\rm supp \;}}

\newcommand{\al}{\alpha}
\newcommand{\be}{\beta}

\newcommand{\ga}{\gamma}
\newcommand{\del}{\delta}
\newcommand{\Del}{\Delta}

\newcommand{\sig}{\sigma}
\newcommand{\lam}{\lambda}
\newcommand{\Lam}{\Lambda}

\newcommand{\Pcal}{\mathcal{P}}
\newcommand{\Bcal}{\mathcal{B}}

\newcommand{\Gcal}{\mathcal{G}}

\newcommand{\Mcal}{\mathcal{M}}

\newcommand{\Hbf}{\mathbf{H}}

\newcommand{\Sbf}{\mathbf{S}}

\newcommand{\Ebf}{\mathbf{E}}

\newcommand{\hbf}{{\bf h}}

\newcommand{\mbf}{{\bf m}}

\newcommand{\xbf}{{\bf x}}

\newcommand{\beq}{\begin{equation}}
\newcommand{\eeq}{\end{equation}}

\makeatletter
\def\@tocline#1#2#3#4#5#6#7{\relax
  \ifnum #1>\c@tocdepth 
  \else
    \par \addpenalty\@secpenalty\addvspace{#2}%
    \begingroup \hyphenpenalty\@M
    \@ifempty{#4}{%
      \@tempdima\csname r@tocindent\number#1\endcsname\relax
    }{%
      \@tempdima#4\relax
    }%
    \parindent\z@ \leftskip#3\relax \advance\leftskip\@tempdima\relax
    \rightskip\@pnumwidth plus4em \parfillskip-\@pnumwidth
    #5\leavevmode\hskip-\@tempdima
      \ifcase #1
       \or\or \hskip 1em \or \hskip 2em \else \hskip 3em \fi%
      #6\nobreak\relax
    \hfill\hbox to\@pnumwidth{\@tocpagenum{#7}}\par
    \nobreak
    \endgroup
  \fi}
\makeatother

\numberwithin{equation}{section}

\begin{document}
 
\title[Generalizations of the Schr\"odinger maximal operator]{Generalizations of the Schr\"odinger maximal operator: \\building arithmetic counterexamples}

\author[Chu]{Rena Chu}
\address{Duke University, 120 Science Drive, Durham NC 27708}
\email{rena.chu@duke.edu}

\author[Pierce]{Lillian B. Pierce}
\address{Duke University, 120 Science Drive, Durham NC 27708}
\email{pierce@math.duke.edu}

\maketitle

\begin{abstract}
    Let $T_t^{P_2}f(x)$ denote the solution to the linear Schr\"odinger equation at time $t$, with initial value function $f$, where $P_2 (\xi) = |\xi|^2$.   In 1980, Carleson asked for the minimal regularity of $f$ that is required for the pointwise a.e. convergence of   $T_t^{P_2} f(x)$ to $f(x)$ as $t \rightarrow 0.$ This was recently resolved   by   work of   Bourgain, and Du and Zhang. This paper considers more general dispersive equations, 
    and constructs counterexamples to pointwise a.e. convergence   for a new class of real polynomial symbols $P$ of arbitrary degree, motivated by  a broad question: what occurs for symbols lying in a generic class?   We construct the counterexamples using number-theoretic methods, in particular the Weil bound for   exponential sums, and the theory of Dwork-regular forms. This is the first case in which counterexamples are constructed for indecomposable forms, moving beyond special regimes where $P$ has some diagonal structure. 
\end{abstract}

\section{Introduction}

Given a polynomial $P(\xi)\in \R[\xi_1,...,\xi_n]$ of degree $k \geq 2$, the operator
    \begin{align}\label{T_int_dfn}
        T_t^{P}f(x):=\frac{1}{(2\pi)^n}\int_{\R^n} \hat{f}(\xi) e^{i(\xi \cdot x + P(\xi)t)} d\xi,
    \end{align}
initially defined for $f$ of Schwartz class on $\R^n$, gives a solution to the linear PDE
    \begin{align}
        \begin{cases}
     \partial_t u - i \Pcal(D) u =0, \quad (x,t) \in \R^n \times \R, \label{PDE}\\
    u(x,0) = f(x), \quad x \in \R^n.
    \end{cases}
    \end{align}
Here $D = \frac{1}{i}(\frac{\partial}{\partial x_1}, \ldots, \frac{\partial}{\partial x_n})$ and $\Pcal(D)$ is defined according to its real symbol by 
\[
\Pcal(D) f(x) = \frac{1}{(2\pi)^n} \int_{\R^n} e^{i \xi\cdot x}P(\xi)\hat{f}(\xi)d\xi.\]

    When $P(\xi)= |\xi|^2$, in which case (\ref{PDE}) is the linear Schr\"odinger equation,
 Carleson famously asked \cite[Eqn (14) p. 24]{Car80}: what is the smallest $s>0$ such that   
    \begin{align}\label{intro_convergence}
        \lim_{t\rightarrow 0} T_t^{P}f(x)=f(x), \qquad \text{a.e.  $x\in \R^n$, for all $f\in H^s(\R^n)$.}
    \end{align}
This question was resolved for dimension $n=1$ quite swiftly by \cite{Car80,DahKen82}, which established that (\ref{intro_convergence}) holds if and only if $s\geq 1/4.$
In higher dimensions, there is a long history of work on   necessary and sufficient conditions for the Schr\"odinger pointwise convergence problem, including \cite{Cow83,Car85,Sjo87,Veg88b,Bou95,MVV96,TaoVar00,Lee06,LucRog17,DGL17,DGLZ18,LucRog19}.   For several decades it was  expected that $s=1/4$ might be the critical threshold in all dimensions, until Bourgain pushed the necessary condition on $s$ above $1/4$ in \cite{Bou13}.  It was very recently resolved (up to the endpoint) by
Bourgain \cite{Bou16}, who showed that $s\geq 1/4 + \del(n)$ with $\del(n) =  (n-1)/(4(n+1))$ is necessary, while Du and Zhang \cite{DuZha19} showed that $s> 1/4 + \del(n)$ is sufficient. 

Bourgain's counterexample construction was interesting: it cleverly employed Gauss sums to force $\sup_{0<t<1}|T_t^Pf(x)|$ to be large (from which a violation to (\ref{intro_convergence}) can be deduced) for test functions $f$ defined using exponential sums. 
Recently, in \cite{ACP23} we expanded this idea into a more flexible method for producing counterexamples to pointwise convergence results of the form (\ref{intro_convergence}) for the initial value problem (\ref{PDE}), using the Weil bound for complete exponential sums. In that initial paper, we demonstrated the new method for symbols of the form $P(X_1,\ldots,X_n) =X_1^k + \cdots + X_n^k$ for any degree $k \geq 3$, and we proved that $s \geq 1/4 + \del(n,k)$ is necessary for (\ref{intro_convergence})  to hold, for $\del(n,k) = (n-1)/(4((k-1)n+1)).$ Subsequently \cite{EPV22a} adapted the method of \cite{ACP23} to achieve a result of the same strength, for any polynomial whose leading form (homogeneous part of highest degree) takes the special shape \beq\label{intro_special}
P_k(X_1,\ldots,X_n) = X_1^k + Q_k(X_2,\ldots,X_n),
\eeq
where $Q_k \in \Q[X_2,\ldots,X_n]$ is a nonsingular form of degree $k$ that is independent of $X_1$.  For degree 2 forms, the special shape (\ref{intro_special}) does not entail a loss of generality, since any quadratic form can be diagonalized over $\R$, and as we will explain, the underlying problem allows for such changes of coordinates. However, for $k \geq 3,$
  forms $P_k$ of the shape (\ref{intro_special}) are  quite sparse among degree $k$ forms in $\Q[X_1,\ldots,X_n]$ (in a sense we quantify in \S \ref{sec_details_examples}), and it is well-known that in arithmetic problems, a form with some diagonal structure is generally  easier to handle. We are motivated by the question: what is the minimal regularity required for (\ref{intro_convergence}) when the real polynomial symbol $P$ has   leading form belonging to a generic class of degree $k$ forms in $\Q[X_1,\ldots,X_n]$?

For any fixed real symbol $P$, the key to proving or disproving pointwise convergence as in (\ref{intro_convergence}) is the associated maximal operator 
\beq\label{T_max_dfn}
f \mapsto \sup_{0<t<1} |T_t^Pf|.
\eeq
For a given $s$, to prove that pointwise convergence (\ref{intro_convergence}) holds for all $f \in H^s(\R^n)$, it suffices to prove (for example) that the maximal operator maps $H^s(\R^n)$ to $L^2_{\mathrm{loc}}(\R^n)$. In the other direction, to prove that convergence (\ref{intro_convergence}) fails for some functions in $H^s(\R^n)$, it suffices to prove that the maximal operator is unbounded from $H^s(\R^n)$ to $L^1_{\mathrm{loc}}(\R^n)$; see for example \cite[Appendix A]{Pie20} for a summary of these standard arguments. Thus we state our main result in terms of showing the maximal operator (\ref{T_max_dfn}) is unbounded from $H^s(\R^n)$ to $L^1_{\mathrm{loc}}(\R^n)$ for $s$ in a certain range.

For any fixed $n \geq 2$ and degree $k \geq 2,$ nonsingular forms are generic among degree $k$ forms in $\Q[X_1,\ldots,X_n]$. 
Given a value $s>0$, the truth (or falsity) of a bound  of the form
    \beq\label{T_bound_inv}
        \|\sup_{0<t<1}|T_t^{P}f|\|_{L^1_{\mathrm{loc}}(\R^n)} \leq C_s \|f\|_{H^s(\R^n)} \qquad \text{for all $f \in H^s(\R^n)$}
    \eeq
  is invariant under $\mathrm{GL}_n(\R)$-action on the polynomial $P$ (see \S \ref{sec_change_var}). 
Thus if one wishes to understand this putative bound for an arbitrary polynomial $P$ with nonsingular leading form $P_k \in \Q[X_1,\ldots,X_n]$, it is no loss of generality to first apply a $\mathrm{GL}_n(\Q)$ change of variable to put $P_k$ in a convenient form. We heavily  exploit the following property: for \emph{every} nonsingular form in $L[X_1,\ldots,X_n]$ for an infinite field $L$, there is a $\mathrm{GL}_n(L)$ change of variables under which the form becomes \emph{Dwork-regular} (see the definition in (\ref{DR_dfn})). Thus in the study of generic forms, it is no loss of generality to focus on Dwork-regular forms, and we do so here. 

The fact that diagonalization, so convenient for quadratic forms, is out of reach for most higher-degree forms, is a dominant theme in the study of symmetric tensors (which, roughly speaking, generalize the symmetric matrix associated to a quadratic form). This has led to the development of many notions of rank for degree $k$ forms, including the Schmidt rank (or $h$-index),  Waring rank (symmetric tensor rank), slicing rank,  relative rank, the property of decomposability, and more. Each such notion of rank is motivated by specific applications in algebraic invariant theory, number theory, algebraic geometry, computational complexity, etc.
Similarly, our present work leads to  a new notion of rank, which we now define. 

\begin{dfn*}[Intertwining rank]
A variable $X_i$   intertwines with $X_j$ (with $i \neq j$) in a form $P_k$ of degree $k \geq 2$ if $(\partial^2/\partial X_i\partial X_j)P_k \not\con 0$. By convention, $X_i$ intertwines with itself. The intertwining rank $r(X_i)$ of   $X_i$ in $P_k$ is the number of variables  with which $X_i$ intertwines. The \emph{intertwining rank} of the form  $P_k$ is $\min_{1 \leq i \leq n} r(X_i)$. 
\end{dfn*}
For example, $X_1^3 + X_2^3 + X_3^3+X_4^3$  has intertwining rank 1, while $X_1^3 + X_1X_2^2 + X_2X_3X_4$ has intertwining rank $2.$ Our main result is:

\begin{thm}\label{thm:main2} Fix $n\geq 2$ and $k\geq 2$.
Let $P\in \R[X_1,...,X_n]$ be a polynomial whose leading form $P_k \in \Q[X_1,\ldots,X_n]$ is Dwork-regular in $X_1,\ldots,X_n$ over $\Q$ and has intertwining rank $r$. Suppose there is a constant $C_s$ such that for all $f\in H^s(\R^n)$,
    \beq\label{thm_T_bound}
        \|\sup_{0<t<1}|T_t^{P}f|\|_{L^1(B_n(0,1))} \leq C_s \|f\|_{H^s(\R^n)}.
    \eeq
Then $s\geq \frac{1}{4} + \del(n,k,r)$ with   \[\del(n,k,r)=\frac{n-r}{4((k-1)(n-(r-1))+1)}.\]
\end{thm}

We now briefly situate Theorem \ref{thm:main2} with respect to   previous literature, and then we   explain the context of Dwork-regular forms, describe more precisely the notion of ``generic'' forms, and illustrate that a  strength of the theorem is its application to indecomposable forms.

\subsection{Relation to previous literature on  convergence problems}  As an immediate consequence of Theorem \ref{thm:main2}, pointwise convergence as in (\ref{intro_convergence}) fails for some $f \in H^s(\R^n)$ for the initial value problem (\ref{PDE}) defined by $P$, for each $s< 1/4 + \del(n,k,r)$  (following the standard arguments recorded in \cite[Appendix A]{Pie20}).
Our present work also adapts (in a trivial way) to dimension $n=r=1$ (see Remark \ref{remark_n_r_1}), but we omit the details, since in 1 dimension,  (\ref{intro_convergence}) holds for all polynomials $P$  of degree $k \geq 2$ if $s \geq 1/4$ and fails if $s<1/4$, by \cite[Cor. 2.6]{KPV91}, \cite{DahKen82,KenRui83}.

The threshold $1/4$ is a common sticking point of many methods in the literature relating to the convergence problem (\ref{intro_convergence}); see e.g. the survey in \cite[\S 1.2]{ACP23}. The main content of Theorem \ref{thm:main2} is for $n \geq 3$, $k \geq 3$, and $2 \leq r<n$.
 In all dimensions,   counterexamples to   (\ref{intro_convergence}) and (\ref{thm_T_bound}) for all $s<1/4$, for any real polynomial symbol (with leading form of any intertwining rank $r \geq 1$), are due to Sj\"{o}lin \cite{Sjo98}.  Theorem \ref{thm:main2} is the first result to go beyond $1/4$ for intertwining rank $r \geq 2$, for all dimensions $n \geq 3$. The strength of our result decreases as $r$ increases, and subsides to the requirement $s \geq 1/4$ when $r=n.$
   For intertwining rank $r=1$,   Theorem \ref{thm:main2} recovers the special case of diagonal symbols considered in \cite{ACP23}, and   the symbols of the form (\ref{intro_special}) considered in \cite{EPV22a}.  For degree $k=2$, by the spectral theorem, any quadratic leading form is diagonalizable under $\mathrm{GL}_n(\R)$, which (after further renormalization) reduces the case of quadratic forms to the case of intertwining rank $r=1$.
For dimension $n=2$, the only cases are intertwining rank $r=1$, in which case Theorem \ref{thm:main2} recovers a result of \cite{EPV22a}, and $r=2$, in which case Theorem \ref{thm:main2} states $s \geq 1/4$, which was previously known. It remains an interesting open question whether the regularity condition in Theorem \ref{thm:main2} can be increased further.

In all dimensions,  (\ref{intro_convergence}) holds for all $s>1/2$, for a wide class of differentiable functions, including  any real
polynomial $P$ of principal type of order $\al$ for $\al>1$ (meaning $|\nabla P(\xi)| \gg (1+|\xi|)^{\al-1}$ for all sufficiently large $|\xi|$), by \cite[Thm. D]{BenDev91} and \cite[Remark 2.2]{RVV06}.  
  Positive results proving bounds related to (\ref{thm_T_bound}) for $s \leq 1/2$, such as the celebrated work in the case $P(\xi)=|\xi|^2$ in \cite{DuZha19},   must proceed by entirely different methods. For further notes on the vast literature on convergence results, maximal operators and connections to local smoothing,   we refer to \cite[\S 1.2]{ACP23}. 
  
\subsection{The role of Dwork-regular forms} Dwork-regular forms have been extensively developed by Dwork \cite{Dwo62} and later Katz (e.g. \cite{Kat08,Kat09}). 
To set the context for their definition, first recall that  a form $P_k$ is said to be nonsingular over $\Q$  if the polynomials $P_k$, $\partial P_k /\partial X_1,\ldots, \partial P_k/ \partial X_n$ have no common zeroes in $\mathbb{P}_{\overline{\Q}}^{n-1}$ (correspondingly the projective hypersurface defined by $P_k=0$ in $\mathbb{P}_{\overline{\Q}}^{n-1}$ is nonsingular). (Here and throughout, $\overline{L}$ denotes a fixed algebraic closure of a given field $L$, and $\mathbb{P}_{\overline{L}}^{n-1}$ denotes the $(n-1)$-dimensional projective space over the field $\overline{L}$.)
In comparison, $P_k$ is said to be Dwork-regular over $\Q$ in the variables $X_1,\ldots,X_n$ if  there are no simultaneous solutions in $\mathbb{P}^{n-1}_{\overline{\Q}}$ to
\beq\label{DR_dfn}
        P_k(X_1,\ldots,X_n)=0, \qquad X_i\frac{\partial P_k}{\partial X_i}(X_1,\ldots,X_n)=0, \qquad 1 \leq i \leq n.
\eeq
A comparison of the definitions shows that any Dwork-regular form over $\Q$ is nonsingular over $\Q$.  As mentioned before, any nonsingular form becomes Dwork-regular  under an appropriate change of variables (see \S \ref{sec:dwork}).
Our interest in passing to Dwork-regular forms is that they are particularly amenable to applications of the Weil-Deligne bound (Lemma \ref{lem:deligne}) even after fixing one or more variables (a consequence of Proposition \ref{prop:dwork_deligne}). This  allows us to make new progress on the convergence problem (\ref{intro_convergence}) despite a central difficulty that appears if each variable ``interacts'' with other variables in the leading form $P_k$.    Intertwining rank captures  the amount of such interaction.  The   novelty in our present work is that we can prove new results for forms of  all intertwining ranks $1<r<n$.

\subsection{Genericity: an underlying motivating question}
We are motivated by the question: what is the behavior of the initial value problem (\ref{PDE}) when $P_k$ is a ``generic''   form in $\Q[X_1,\ldots,X_n]$?   Technically, a class of forms is said to be generic if it corresponds to a  open set (in the Zariski topology) in the moduli space of all degree $k$ homogeneous forms in $\Q[X_1,\ldots,X_n].$ 

As one example,   nonsingular forms are generic. It is equivalent to show that being singular is a condition on the coefficients of $P_k$ that defines a closed set in the Zariski topology.  Since $P_k$ is singular if and only if $P_k, \partial P_k/\partial X_1,...,\partial P_k/\partial X_n$ have a common nonzero root, then $P_k$ is singular if and only if  the resultant of $P_k, \partial P_k/\partial X_1,...,\partial P_k/\partial X_n$  vanishes. This resultant is a (nonzero) polynomial in the coefficients of $P_k, \partial P_k/\partial X_1,...,\partial P_k/\partial X_n$, so that $P_k$ being singular is characterized by its coefficients lying in the vanishing set of a polynomial, proving the claim. 

As another example,   \emph{indecomposable} forms are generic; we describe this property thoroughly in \S \ref{sec:indecomposable} below. The union of two Zariski closed sets (e.g. the set of singular forms and the set of decomposable forms) is closed, and so the complement (e.g. the set of forms that are nonsingular and indecomposable) is open, and hence generic. 
Indeed, any generic condition will include nonsingular forms (generically), and indecomposable forms (generically). This fact cuts in two directions, one convenient and one inconvenient.  First, on the one hand, even if we are interested in studying generic forms, it is reasonable only  to   consider nonsingular forms (which is advantageous for an application of Lemma \ref{lem:deligne}). But on the other hand, it shows that to understand the generic situation, we must  understand the case of indecomposable forms. 

One strength of Theorem \ref{thm:main2} is that it proves the first  (nontrivial) counterexamples to (\ref{intro_convergence})  that apply to  leading forms $P_k$ that are indecomposable. Thus in the next section we describe decomposability/indecomposability in more detail.  
Nevertheless, Theorem \ref{thm:main2} falls short of proving nontrivial results for a generic class of forms: it is only nontrivial for forms of rank strictly smaller than $n$, and these are not generic. In \S \ref{sec_codim}  we compute the codimension of Dwork-regular forms of   intertwining rank $r<n$, among degree $k$ forms in $\Q[X_1,\ldots,X_n]$. This codimension quantifies  that  Theorem \ref{thm:main2} proves nontrivial results for a class of forms that is not generic, but that nevertheless contains    ``many more'' real symbols than were tractable in previous works.

\subsection{Indecomposable forms of degree $k$:   definition, remarks and examples}\label{sec:indecomposable}

A form is called decomposable (or sometimes of Sebastiani-Thom type) over a field $L$ if there is a $\mathrm{GL}_n(L)$ change of variables so that the form can be written as a sum of at least two forms in disjoint sets of variables, for example
\beq\label{PQQ}
P(X_1,...,X_n) = Q_1(X_{1},...,X_{m})+Q_2(X_{m+1},...,X_{n})
\eeq
 for some $1 \leq m < n$.
 Otherwise, a form is indecomposable over $L$.  
All of the forms considered in \cite{ACP23,EPV22a} were of the special shape (\ref{intro_special}),  
and thus have intertwining rank $r=1$. All forms with intertwining rank $r=1$ are decomposable. 
Yet we are   motivated to tackle indecomposable forms, since in the moduli space of degree $k \geq 3$ forms in $\Q[X_1,\ldots,X_n],$   indecomposable forms are generic for $n \geq 2$, at least for $(n,k)\neq (2,3)$. (This is related e.g. to \cite[\S 6]{Wan15}, \cite[Thm. 3.2]{HLYZ21}, \cite[p. 303]{ORySha03}; see further details in \S \ref{sec_details_examples}.)

  Note that the intertwining rank of any decomposable polynomial is at most $\lfloor n/2 \rfloor$; this can be seen by inspecting (\ref{PQQ}). We use this to deduce the following immediate corollary of Theorem \ref{thm:main2}, verified in \S \ref{sec_cor}. 
\begin{cor}\label{corollary_decomp}
Fix $n\geq 2$ and $k\geq 2$.
Let $P\in \R[X_1,...,X_n]$ be a polynomial whose leading form $P_k \in \Q[X_1,\ldots,X_n]$ is decomposable and nonsingular in $X_1,\ldots,X_n$ over $\Q$. Suppose there is a constant $C_s$ such that for all $f\in H^s(\R^n)$,
   \[
        \|\sup_{0<t<1}|T_t^{P}f|\|_{L^1(B_n(0,1))} \leq C_s \|f\|_{H^s(\R^n)}.
\]
Then $s\geq \frac{1}{4} + \frac{n}{4((k-1)(n+2)+2)}$. 
\end{cor}

Since it   has been remarked that it is challenging to exhibit indecomposable forms (see e.g. \cite[p. 348]{Pum06}, \cite[p. 576]{Wan15}), we provide explicit examples.
For each degree $k \geq 3$ and rank $2\leq r \leq n,$ we exhibit indecomposable forms of degree $k$ that are Dwork-regular over $\Q$ in $X_1,\ldots, X_n$ with   intertwining rank $r$, namely \begin{align*}
    P_k(X_1,\ldots,X_n)&=X_1^k+\cdots +   X_n^k + \sum_{2\leq j \leq r} X_1X_j^{k-1} + \sum_{2\leq i < j \leq n} X_iX_j^{k-1},\quad \text{$k \geq 3$ odd;}
   \\
        P_k(X_1,\ldots,X_n)&=X_1^k+\cdots +   X_n^k + \sum_{2\leq j \leq r} X_1^2X_j^{k-2}+ \sum_{2\leq i < j \leq n} X_i^2X_j^{k-2},
   \quad \text{$k \geq 4$ even.}
   \end{align*}
For example, in dimension $n=3$ the examples with intertwining rank $2$ are  
\begin{align*}
    P_k(X_1,X_2,X_3)&=X_1^k+X_2^k +   X_3^k +   X_1X_2^{k-1}+ X_2X_3^{k-1},\quad \text{$k \geq 3$ odd;}
   \\
        P_k(X_1,X_2,X_3)&=X_1^k+X_2^k +   X_3^k +   X_1^2X_2^{k-2}  + X_2^2X_3^{k-2},
   \quad \text{$k \geq 4$ even.}
\end{align*}
In  \S \ref{sec_details_examples}, we use a criterion of Harrison \cite{Har75,HarPar88} to verify that these are indecomposable forms over $\Q$ (and hence in particular cannot be brought to have intertwining rank $1$ by any $\mathrm{GL}_n(\Q)$ change of variables).

\subsection{Further directions} By  the invariance of (\ref{T_bound_inv}) under $\mathrm{GL}_n(\Q)$-action on $P$, the main result of Theorem \ref{thm:main2} furthermore applies  to any polynomial $P$ with leading form $P_k \in \Q[X_1,\ldots,X_n]$ lying in the $\mathrm{GL}_n(\Q)$-orbit of a Dwork-regular form of  intertwining rank $r$. How big is such an orbit? This points to an interesting question, which   is in fact typical when one encounters a notion of rank for a higher degree form. Given a particular notion of rank, in an application one often wants to manipulate the original form (or class of forms) to make the (particular) rank more advantageous; the limits of this procedure may depend on the underlying field (and whether it is algebraically closed). For example, in the case of Schmidt rank, there is recent work on regularization and a new relative rank in \cite{LamZie21x}, and the relation to algebraic closure in \cite{LamZie22x}. In the case of Waring rank, see the celebrated work of Alexander and Hirschowitz \cite{AleHir95}  over $\C$ (and a nice overview in \cite{RanSch00}), or more recent work for monomials in \cite{CCG12} and partial progress over $\R$ or $\Q$ in \cite{HanMoo22},  with intriguing remarks on the dependence on the underlying field in \cite{Rez13}.
For our particular setting, this becomes the question: how does the intertwining rank behave under minimization via $\mathrm{GL}_n(\Q)$?   We pursue this question, which requires completely different methods, in other work.

\subsection{Notation} 
In this paper we employ the convention $e(t) = e^{it}$. Correspondingly, $\hat{f}(\xi) = \int_{\R^m}f(x) e^{-i x\cdot \xi} dx$ and $f(x) = (2\pi)^{-m}\int_{\R^m} \hat{f}(\xi) e^{i x \cdot \xi} d\xi$, so that Plancherel's theorem takes the form $\|f\|^{2}_{L^2(\R^m)} = (2\pi)^{-m} \|\hat{f}\|_{L^2(\R^m)}^2.$ The Sobolev space $H^s(\R^m)$ is defined to be all $f \in \mathcal{S}'(\R^m)$ with finite Sobolev norm
\[ \|f\|_{H^s(\R^m)}^2 = \frac{1}{(2\pi)^m} \int_{\R^m} (1+|\xi|^2)^s |\hat{f}(\xi)|^2 d\xi.\]

We use the convention that $B_m(c,r)$ is the Euclidean ball of radius $r$ centered at $c$ in $\R^m$. The notation $A \ll_\kappa B$ denotes that $|A| \leq C(\kappa) B$ for a constant $C(\kappa)$. It is harmless in our argument to allow constants to depend on the dimension $n$, the symbol $P$ of degree $k$, the intertwining rank $r$, and a Schwartz function $\phi$ we will fix once and for all. Certain small constants, which we can choose freely,  we will denote by $c_0,c_1,c_2,...$; we will demarcate these explicitly in inequalities when we are preparing to exploit their small size. 
For $v=(v_1,...,v_m),w=(w_1,...,w_m)\in \R^m$, we define $v\circ w = (v_1w_1,...,v_mw_m)$. For a multi-index $\al$ we set $y^\al = y_1^{\al_1} \cdots y_n^{\al_n}$, $\al! = \al_1! \cdots \al_n!,$ $|\al| = \al_1 + \cdots + \al_n,$ and $\partial^\al  = \partial_1^{\al_1}  \cdots \partial_n^{\al_n};$ for two multi-indices $\al,\be$, $\al \geq \be$ and $\al - \be$ denote coordinate-wise relations.

\section{Method of proof}\label{sec:method}

To prove Theorem \ref{thm:main2}, we construct a family of test functions $\{f_j\}$ that are Fourier-supported in an annulus $\{(1/C)R_j\leq |\xi|\leq CR_j\}$ of radius $R_j$ for a sequence of $R_j \maps \infty$ as $j \maps \infty$. By definition
\[R_j^s \|f_j\|_{L^2}\ll_{C,s} \|f_j\|_{H^s} \ll_{C,s} R_j^s \|f_j\|_{L^2}.\]
Hence Theorem  \ref{thm:main2} follows immediately from an explicit construction:

\begin{thm}\label{thm:main3}
Let $n \geq 2$. Let $P\in \R[X_1,...,X_n]$ be a polynomial of degree $k \geq 2$ whose leading form $P_k\in \Q[X_1,\ldots,X_n]$   is Dwork-regular over $\Q$ in the variables $X_1,\ldots,X_n$, and has intertwining rank  $r \leq n$. Fix any $s< \frac{1}{4} + \del(n,k,r)$ with $\del(n,k,r)$ as in Theorem \ref{thm:main2}. Then there exists an infinite sequence of $j \maps \infty$ such that for  $R_j=2^j$, there exists a function $f_j\in L^2(\R^n)$, where $\|f_j\|_{L^2}=1$ and $\hat{f}_j$ is supported in an annulus $\{(1/C)R_j\leq |\xi|\leq CR_j\}$, and with the property that
    \[
        \lim_{j\rightarrow \infty} \frac{\|\sup_{0<t<1} |T_t^{P}f_j(x)|\|_{L^1(B_n(0,1))}}{R_j^s} = \infty.
\]
\end{thm}

To prove Theorem \ref{thm:main3}, we define each test function $f$ so that $|T_t^Pf(x)|$ can be approximated by  an $(n-r)$-dimensional exponential sum, which we   show is ``large'' for many $x \in B_n(0,1)$, after choosing $t$ appropriately (depending on   $x$).  This strategy is motivated by ideas of Bourgain (for degree $k=2$) as explained in \cite{Pie20}, and  the flexible construction (applicable to degree $k\geq 3$) developed in our earlier work \cite{ACP23}, for the diagonal case $P(X_1,\ldots,X_n)=X_1^k + \cdots + X_n^k.$
However, a difficulty arises if the leading form $P_k$ of the real symbol has intertwining rank $r>1$: in order to optimize the test functions $f$   to violate the supposed upper bound (\ref{thm_T_bound}) for $s$ as large as possible,    one is naturally led to dilate one variable within the exponential sum, say $X_1,$ by a large parameter. This large dilation contributes large error terms to certain approximation arguments.  We overcome this by   using the notion of intertwining rank.
To assist the reader in tracking the main ideas of the method, we now present a series of heuristic computations that are  not rigorous, but simply emphasize numerology. The  remainder of the paper carries out each step rigorously.

\subsection{Heuristic overview}\label{sec_heuristic}
 
Let $\Phi_n(x_1,\ldots,x_n)$ be a non-negative Schwartz function with $\Phi_n (0)=1$ and $\hat{\Phi}_n$ supported in $[-1,1]^n.$ As in Theorem \ref{thm:main3}, we think of $R$ as a parameter that will go to infinity.
Define a test function 
\beq\label{intro_f_dfn} f(x) = \Phi_n(S \circ x) \sum_{\bstack{m \in \Z^n}{m_j \approx R/\Lambda_j}} e((\Lambda \circ m) \cdot x)
\eeq
for some parameters $S=(S_1,\ldots,S_n)$ and $\Lambda = (\Lambda_1,\ldots,\Lambda_n)$, with each $S_i,\Lam_i$ chosen later to be $1, R$ or a small power of $R$. Let $\|S\| = \prod S_j$ for the moment, and similarly for $\|\Lambda\|$. The Fourier transform $\hat{f}$ is supported in an annulus of radius $\approx R$ if each $S_j\ll R$, and $\|f\|_{H^s} \approx R^s \|S\|^{-1/2} R^{n/2}\|\Lambda\|^{-1/2}$, so that by normalizing appropriately, $f$ fits the hypotheses of Theorem \ref{thm:main3}.   For this test function,
\[ T_t^Pf(x)  = \frac{1}{(2\pi)^n} \int_{\R^n} \hat{\Phi}_n(\xi) e((S \circ \xi)\cdot x)
    \sum_{\bstack{m \in \Z^n}{m_j \approx R/\Lambda_j}} e((\Lambda \circ m) \cdot x)
        e(P(S \circ \xi + \Lambda \circ m)t)d\xi.\]
        For simplicity we temporarily assume $P$ is a homogeneous form of degree $k\geq 2$, defined in terms of coefficients and multi-indices by
        \[ P(y_1,\ldots,y_n) = \sum_{|\al| =k} c_{\al} y^\al.\]
Step 1: Use partial summation to remove all terms in the sum over $m$ that depend on $\xi$ (i.e. that are ``not arithmetic''); this replaces the sum over $m$ (up to an error term) by 
\beq\label{intro_arith}
\approx w(R/\Lambda_1,\ldots, R/\Lambda_n)  \sum_{\bstack{m \in \Z^n}{m_j \approx R/\Lambda_j}} e((\Lambda \circ m) \cdot x+P(\Lambda \circ m)t),
\eeq
in which the new sum is the ``arithmetic contribution'' while \beq\label{intro_w_dfn}
w(y_1,\ldots,y_n) =  e((P(S \circ \xi + \Lambda \circ y) - P(\Lambda \circ y))t) \eeq
is the ``weight'' that has been removed by partial summation. The weight $w(R/\Lambda_1,\ldots, R/\Lambda_n) $ contributes to the integral over $\xi$ a factor with a linear phase in $\xi$, namely $\approx e( S \circ \xi \cdot t \nabla P(\underline{R})),$ 
where $\underline{R}=(R,\ldots,R).$

Step 2: Use integration by parts to remove all 
terms in the phase of the integral over $\xi$ that are order 2 or higher in $\xi$ (up to an error term).

Step 3: After Step 2, one may immediately apply Fourier inversion to the remaining integral over $\xi$, so that the main contribution to $T_t^Pf(x)$ is a product of the arithmetic contribution and  
\[ \approx \Phi_n(S \circ (x + t \nabla P(\underline{R}))).\]
Then place constraints on  $x$ and $t$ so that $S \circ (x + t \nabla P(\underline{R})) \approx 0$, so that applying $\Phi_n(0)=1$ (and continuity of $\Phi_n$) implies that $\Phi_n(S \circ (x + t \nabla P(\underline{R})))\approx 1$ and hence
\beq\label{intro_T_sum}
|T_t^Pf(x)| \approx  |\sum_{\bstack{m \in \Z^n}{m_j \approx R/\Lambda_j}} e((\Lambda \circ m) \cdot x+P(\Lambda \circ m)t)|.
\eeq

Step 4: Construct a set of $x \in \R^n$ that is a positive proportion of $B_n(0,1)$ so that for each $x$ in the set there exists $t \in (0,1)$ for which the arithmetic contribution (\ref{intro_T_sum}) is large.  Up to some simple changes of variables, the set is a union of boxes centered at rationals $(a_1/q,a_2/q,\ldots,a_n/q)$ for primes $q \approx Q,$ where $Q$ is a small power of $R$ to be chosen later. 

Step 5: Optimize the choices of $S=(S_1,\ldots,S_n),$ $\Lambda = (\Lambda_1,\ldots,\Lambda_n)$, and $Q$,
subject to the constraints that the error terms in all previous steps are acceptable. 
 
To unravel the chain of dependencies that make these steps efficient and compatible, first consider Step 3, which requires that for each $j=1,\ldots,n$
\beq\label{t_relation_heuristic}
S_j (x_j + t R^{k-1}) \ll 1 \iff t \approx -\frac{x_j}{R^{k-1}} + O(  \frac{1}{R^{k-1}S_j}). \eeq
(This sketch assumes in particular that $\partial_1 P(\underline{R})\gg R^{k-1}$; to achieve this, we develop Lemma \ref{lem:coefficients}.)
 Given $(x_1,\ldots,x_n)$, if we choose $t$ to satisfy this for $x_1$, then the only way it can simultaneously satisfy it for $x_2,\ldots,x_n$   is  for $x_2,\ldots,x_n$ to all lie in $O(S_j^{-1})$ neighborhoods of $x_1$. This is too limiting in Step 4 unless we set $S_2=\cdots=S_n=1$, which we now do, so $S=(S_1,1,\ldots,1)$ and $\|S\|=S_1.$ From now on, because of the ``large'' rescaling factor $S_1$, the first coordinate $x_1$ will play a special role. 

Next consider Step 2, in which we use iterated integration by parts (coordinate by coordinate) to remove a ``weight'' from the integral that contains all terms in the phase that are order 2 or higher in $\xi$; this weight takes the approximate form 
\[ W(\xi_1,\ldots,\xi_n)  = e([P(S \circ \xi +  \underline{R}) - L_0(\xi) - L_1(\xi)]t)  = e(t \sum_{\bstack{|\be + \ga| = k}{|\be|\geq 2}}c_{\be + \ga}C(\be,\ga) (S \circ \xi)^\be \underline{R}^\ga),\]
in which $L_0$ (respectively $L_1$) represents terms in $P(S \circ \xi + \underline{R})$ that are order  0 (respectively order 1) in $\xi$, and $C(\be,\ga)$ are positive combinatorial constants. As usual, the error term when a weight is removed by integration by parts (or summation by parts) will be smaller if the weight is slowly-varying, and thus we must control the derivatives of $W$. The error accrued in Step 2 must be at most a small proportion of the main term (\ref{intro_T_sum}). This will be achieved if for each multi-index $\kappa \in \{0,1\}^n,$ for all $\xi \in \supp \hat{\Phi}_n \subseteq [-1,1]^n$,
\beq\label{intro_W_sum}
|\frac{\partial^{|\kappa|}}{\partial \xi^\kappa} W(\xi)| \ll 1 \iff t \sum_{\bstack{|\be + \ga| = k}{|\be| \geq 2, \be \geq \kappa}}c_{\be + \ga}C(\be,\ga) (S \circ \xi)^{\be - \kappa}S^\kappa \frac{\be!}{( \be - \kappa)!} \underline{R}^\ga \ll 1.
\eeq
From Step 3 we know that $t\approx R^{-(k-1)},$ so we require the sum above   to satisfy $\ll R^{k-1}$. Each term in the sum is roughly of size $S^\be \underline{R}^\ga = S_1^{\be_1} R^{k-|\be|}$ for some $|\be|\geq 2$. There are two scenarios: if $|\be|>\be_1$, this term is $\ll R^{k-1}$ as long as $S_1 \ll R$. If $\be_1 = |\be|$, which can only occur if $\be_1\geq 2,$ then this term is $\ll R^{k-1} (S_1^{\be_1}/R^{\be_1-1})$, which is $ \ll R^{k-1}$ as long as $S_1 \ll R^{\frac{\be_1 -1}{\be_1}},$ which is most restrictive when $\be_1=2$. Thus we  impose the condition $S_1 \ll R^{1/2}.$

Next consider Step 1, in which the error introduced by iterated partial summation must be at most a small proportion of the main term (\ref{intro_T_sum}). This will be achieved 
 if for each multi-index $\kappa \in \{0,1\}^n,$ for all $y$ with $y_j \approx R/\Lam_j$, 
\beq\label{intro_w_sum} |\frac{\partial^{|\kappa|}}{\partial y^\kappa} w(y)| \ll \Lambda^\kappa R^{-|\kappa|}
\iff t \sum_{\bstack{|\be + \ga| = k}{|\be| \geq 1, \ga \geq \kappa}}c_{\be + \ga}C(\be,\ga) (S \circ \xi)^{\be}  (\Lambda \circ y)^{\ga-\kappa} \Lambda^\kappa \frac{\ga!}{( \ga - \kappa)!}\ll \Lambda^\kappa R^{-|\kappa|}.
\eeq
Note that in contrast to (\ref{intro_W_sum}), in this case the phase in the weight $w(y)$ includes terms of order 1 in $\xi$, so that $|\be|=1$ is allowed in the sum immediately above. Each term in the sum immediately above is roughly of size $S^\be \Lambda^\kappa (\Lambda \circ y)^{\ga - \kappa } \approx S_1^{\be_1}  \Lambda^\kappa \underline{R}^{\ga - \kappa}\approx S_1^{\be_1} R^{k - |\be|} \cdot  \Lambda^\kappa R^{- |\kappa|}.$ Thus the  condition (\ref{intro_w_sum}) will be met, recalling $t\approx R^{-(k-1)}$, as long as $S_1^{\be_1} R^{k - |\be|} \ll R^{k-1}$ for all $|\be| \geq 1$. There are again several scenarios: if $|\be| > \be_1$ or if $\be_1 = |\be| \geq 2$, this term is $\ll R^{k-1}$ by arguing as in Step 2, under our assumption $S_1 \ll R^{1/2}$. The problem is that there is now also a third case,  with $\be_1 = |\be|=1$ in which case the requirement is asking that $S_1^{1} R^{k-1} \ll R^{k-1}$. These problematic terms can be seen as the contribution to the weight (\ref{intro_w_dfn}) that is varying the fastest with respect to $y$, namely the portion of the phase that is highest order in $y$ (total degree $k-1$) and linear in $\xi$. (We also provide an explicit example of such terms in (\ref{P_example}).) One way to achieve the requirement $S_1^{1} R^{k-1} \ll R^{k-1}$ is to  impose $S_1 \ll 1$, but this is inefficient in Step 5. The strategy we adopt is to modify the definition of the test function $f$ so that such terms never appear. 

Precisely, in the definition (\ref{intro_f_dfn}) of the test function $f$, we now restrict  the sum over $m \in \Z^n$  to sum only over those coordinates $m_j$ with the following property: for each multi-index $\al=(\al_1,\ldots,\al_n)$, if  $\al_j \geq 1$ and $\al_1 \geq 1$ then the coefficient $c_\al=0$ in the original polynomial $P(y)$. Equivalently, we define the sum to be only over those coordinates $m_j$ such that $X_j$ never appears in a monomial with $X_1$ in the original polynomial $P(y)$. 
(Equivalently, set $\Lambda_j \approx R$ for each $j$ such that $X_j$ appears in a monomial with $X_1$, and in all other coordinates take $\Lambda_j=L$, for $L$ a small power of $R$ to be chosen later.)

Because the exponential sum is a source of gain in Steps 4 and 5, we wish to sum over as many coordinates as possible, so depending on $P(y)$ we relabel coordinates in the beginning so that $X_1$ is the variable that appears in monomials with as few other coordinates as possible, say $X_2,\ldots, X_r$ with $r<n$; this is the motivation for defining intertwining rank. We now let   $m = (m_{r+1},\ldots,m_n)$, and only sum over these coordinates. The conclusion is that in place of (\ref{intro_T_sum}) we arrive at a main contribution of the form
\beq\label{intro_T_sum_mod}
|T_t^Pf(x)| \approx  |\sum_{\bstack{(m_{r+1},\ldots,m_n) \in \Z^{n-r}}{m_j \approx R/L}} e(m \cdot (Lx_{r+1},\ldots,Lx_n)+P(R/L,\ldots,R/L,m)L^k t)|.
\eeq
(Here we used homogeneity of $P$ of degree $k$; to make $R/L$ integral, see Remark \ref{remark_integer}.)

In Step 4, to construct a set of $x$ for which the above sum in (\ref{intro_T_sum_mod}) is ``large,'' imagine that each $Lx_j =a_j/q$ and $L^kt=a_1/q$ are rationals of prime denominator $q \approx Q$, so that the sum can be regarded as $\approx ((R/L)q^{-1})^{n-r} $ copies of a sum where each coordinate $m_j$ runs over a complete set of residues modulo $q$. Since $P(y_1,\ldots,y_n)$ is Dwork-regular, even after specializing the first $r$ variables, the remaining polynomial is well-behaved (Proposition \ref{prop:dwork_deligne}). A major feature of our argument shows that for a positive proportion  ($\gg q^{n-r+1}$) of choices of  $a_1, a_{r+1},\ldots, a_n$ the $(n-r)$-dimensional sum mod $q$ is of the optimal size $\approx q^{(n-r)/2}$ (Proposition \ref{prop:bound_T}). Hence at precisely such a point $x$ and for such a  $t,$
\beq\label{intro_T_size}
|T_t^Pf(x)| \approx \left(\frac{R}{Lq}\right)^{n-r} q^{(n-r)/2} \approx \left(\frac{R}{LQ^{1/2}}\right)^{n-r}.
\eeq

To achieve a similar result for a positive measure of $x$, we need to show this continues to hold for $Lx_j$ and $L^kt$ merely ``close'' to rationals with denominator $q$. Typically, deducing this from the case where they are precisely rationals would follow by applying partial summation to the sum in (\ref{intro_T_sum_mod}). The error incurred by partial summation will be too large if the ``weight'' removed involves   terms of the highest order in $m$. Thus we must choose $t$ so that (i): $L^kt=a_1/q$ precisely. Fortunately this is possible because of the wiggle room allowed in (ii): $L^kt \approx -L^kx_1/R^{k-1} + O(L^k/R^{k-1}S_1)$, from (\ref{t_relation_heuristic}) in Step 3. Given $x_1$ with $-L^kx_1/R^{k-1}$ in an interval of length $O(1/q)$ centered at $a_1/q$ we can always choose $t$ meeting both requirements (i) and (ii) as long as $Q^{-1} \ll L^k/R^{k-1}S_1,$ which we now assume. 

In contrast, the coordinates $Lx_{r+1},\ldots,Lx_n$ appear as coefficients of the lowest-order (linear) terms in $m$ so that   partial summation will contribute reasonable errors when we allow $Lx_j$ to vary in an interval around $a_j/q$. If the interval is of length $V$, say, the contributed error will be proportional to  $(R/L)^{n-r}V^{n-r}$ times the size of the main term, so we require $(R/L)^{n-r}V^{n-r}\ll 1$. At the same time, we want $V$ to be large in order for the boxes so constructed to cover a positive proportion of $(x_{r+1},\ldots,x_n) \in [0,1]^{n-r}$. In this regard the principle of simultaneous Dirichlet approximation in $n-r$ dimensions motivates the choice $V=1/qQ^{1/(n-r)} \approx Q^{-1-1/(n-r)}$ (see e.g. \cite[Appendix B]{Pie20}). Taken together, these two requirements force the condition $Q^{-1-1/(n-r)} \ll L/R.$ 

Cumulatively, this construction yields boxes in the coordinates $x_1,x_{n-r},\ldots,x_n$, each of measure $\approx  Q^{-1} V^{n-r}$, centered at around $\gg q^{n-r+1}$ rational tuples with denominator $q$ for each prime $q \approx  Q$. A naive calculation suggests  the total measure of the union of the boxes could be   $\approx (Q/\log Q) Q^{n-r+1} Q^{-1} (Q^{-1-1/(n-r)})^{n-r} \approx 1/\log Q$. Since the boxes can overlap significantly, a sophisticated justification is required although the conclusion agrees with the above (Proposition  \ref{prop:Omega}).

Upon reaching Step 5, these heuristics suggest that for the test function $f$ so constructed, with $\|f\|_{H^s} \approx R^s S_1^{-1/2} (R/L)^{(n-r)/2}$, on a set $x \in B_n(0,1)$ of measure $\gg 1/\log Q,$  
 $|T_t^Pf(x)| \gg (R/LQ^{1/2})^{n-r}$. This occurs under the constraints $S_1 \leq R^{1/2}$, $Q^{-1-1/(n-r)} \ll L/R$, 
$Q^{-1} \ll L^k/R^{k-1}S_1$. The claim of Theorem \ref{thm:main3} consequently holds for each $s$ such that 
\[ \frac{(\frac{R}{LQ^{1/2}})^{n-r}(\log Q)^{-1}}{S_1^{-1/2} (R/L)^{(n-r)/2}} \gg R^{s'}\]
for some $s'>s.$ Upon setting $L=R^\lambda, Q=R^\kappa, S_1= R^\sig$ this is equivalent to a linear condition on $\lambda,  \kappa, \sig$ and $s$ subject to linear constraints, and it can be optimized, which we do in detail in \S \ref{sec:parameters}.

 \subsection{Outline of the paper} In \S \ref{sec:dwork} we state and prove all the key properties of Dwork-regular polynomials we will use, including  that upon fixing one or more variables, the resulting polynomial is a Deligne polynomial over $\F_q$ (for all but finitely many primes $q$). In \S \ref{sec:exponential} we  prove upper and lower bounds on (complete and incomplete) exponential sums involving Deligne polynomials. In \S \ref{sec:reducing} we approximate $T_t^{P}f$, for appropriate test functions $f$, by an exponential sum. In \S \ref{sec:arithmetic} we define a  set of $x \in B_n(0,1)$ for which we can approximate this sum by complete exponential sums to which we can apply the arithmetic results of \S \ref{sec:exponential}. In \S \ref{sec:parameters} we then  optimize the choices of all parameters, thus proving Theorem \ref{thm:main3}. Finally, in \S \ref{sec_details_examples} we provide details on examples of Dwork-regular, indecomposable forms of arbitrary intertwining rank and degree, and remark on the codimension of such forms.

\section{Properties of Dwork-regular forms}\label{sec:dwork}
In this section we first gather three   algebraic properties of Dwork-regular forms that we will apply throughout the proof. Then in \S \ref{sec_gradient}, we prove a lower bound on a partial derivative of a Dwork-regular form, in \S \ref{sec_change_var} we  show that boundedness (or unboundedness) of the maximal operator is invariant under a $\GL_n(\Q)$ change of variables, in \S \ref{sec_cor} we verify Corollary \ref{corollary_decomp}, and finally in \S \ref{sec:dispersive} we remark that the PDE's we consider are dispersive.

It is convenient to work temporarily in a more abstract setting, and simply fix a field $L$, which could for example be $\Q$, $\R$ or a finite field $\F_q$. We will later call upon the lemmas we prove both in the setting of infinite fields such as $\Q$ and $\R$ and finite fields $\F_q$ for $q$ prime. Let $L$ be a field and $H\in L[X_1,...,X_n]$ a homogeneous polynomial of degree $k \geq 2$.
Then $H$ is nonsingular over $L$  if $H, \partial H /\partial X_1,\ldots, \partial H/ \partial X_n$ have no common zeroes in $\mathbb{P}_{\overline{L}}^{n-1}$ (correspondingly the projective hypersurface defined by $H=0$ in $\mathbb{P}_{\overline{L}}^{n-1}$ is nonsingular). 
Recall that $H$ is  Dwork-regular in the variables $X_1,...,X_n$ over $L$ if there are no solutions in $\mathbb{P}^{n-1}_{\overline{L}}$ to the simultaneous equations
\[
        H(X_1,\ldots,X_n)=0, \qquad X_i\frac{\partial H}{\partial X_i}(X_1,\ldots,X_n)=0 \qquad 1 \leq i \leq n.
\]
If $H$ is Dwork-regular over $L$,  then $H$ is  nonsingular over $L$. However, it can be that $H$ is nonsingular but not Dwork-regular: for example,   $X_1^k +\cdots + X_{n-1}^{k} + X_{n-1}X_n^{k-1}$ over $\Q$. 
However,   if the field $L$ is infinite, given any nonsingular form, there exists a   $\GL_n(L)$ change of variables under which the form becomes Dwork-regular (see \cite[pp. 67--68]{Dwo62} and \cite[Lemma 3.1]{Kat09}). (In fact, for a given form, there are many such changes of variables: the proof of \cite[Lemma 3.1]{Kat09} can be adapted to  show that the set of such elements   is dense  in $\GL_n(L)$.)
We quote from Katz in \cite[p. 1252]{Kat08}, that the archetypical Dwork-regular  polynomial $H$ would be of the form $H(X) = \sum_{i=1}^n X_i ^k + \tilde{H}(X)$, where $\tilde{H}$ is any polynomial of degree at most $k-1$. The antithesis to a Dwork-regular polynomial, when $n=2m$ is even and $L$ has odd characteristic, is something of the form $H(X) = \sum_{i=1}^{m} X_i X_{m+i}$.
 By Euler's identity, for a form $H$ of degree $k$, 
$k H = \sum_{i=1}^n X_i (\partial/\partial X_i)H,$ so that when discussing Dwork-regular forms it is natural to assume that $\mathrm{char} L \ndiv k$, if $L$ is finite.

In this section, we  first present an equivalent characterization of   Dwork-regularity that is easier to work with. (This property has previously been remarked in the context of \cite[Lemma 3.1]{Kat09}.)

\begin{lemma}\label{lem:dwork_equiv}
    Let $L$ be a field and let $H(X_1,...,X_m)\in L[X_1,...,X_m]$ be a homogeneous polynomial of degree $d \geq 2$. For every nonempty $S\subseteq \{1,...,m\}$, define $H_S:=H|_{X_j=0,j\not\in S}$. Then $H$ is Dwork-regular in $X_1,\ldots,X_m$ over $L$ if and only if  
        \begin{enumerate}
            \item for all $S$ with $|S|=1$, $H_S$ is not the zero polynomial, and
            \item for all $S$ with $|S|\geq 2$, the hypersurface defined by $H_S=0$ is nonsingular in $\mathbb{P}^{|S|-1}_{\Bar{L}}$ in the variables $X_i,i\in S$.
        \end{enumerate}
\end{lemma}
 
The first condition of Lemma \ref{lem:dwork_equiv} implies that a degree $d$ form that is Dwork-regular in $X_1,...,X_m$ necessarily contains a nonzero multiple of  each monomial $X_1^d,...,X_m^d$.  

Second, we verify that  if a form $H \in \Z[X_1,\ldots,X_m]$ is Dwork-regular  over $\Q$ in the variables $X_1,\ldots, X_m$, then  its reduction modulo $q$ is Dwork-regular over $\F_q$ in the variables $X_1,\ldots,X_m$  for all but finitely many primes $q$;  in particular this is true for all primes $q \geq K_1$ for a finite constant $K_1 = K_1(H)$. We can describe this abstractly over any field $L$ as follows:

\begin{lemma}\label{lemma_Dwork_reduced}
Let $L$ be a field and let $H(X_1,...,X_m)\in L[X_1,..,X_m]$ be Dwork-regular over $L$  in the variables $X_1,...,X_m$. Then there exists a finite set $S$ of finite places of $L$ such that for all finite places $\fp\not\in S$, the reduced polynomial $H\modd{\fp}$ is Dwork-regular over the residue field $\mathcal{O}_S/\fp$.
\end{lemma}

We reserve the precise definitions of the notion of a finite place, a residue field, and the ring $\mathcal{O}_S$ to  \S \ref{sec:proof_of_lemma}. In the case when $L=\Q$, a finite place corresponds to a prime number, and the conclusion of the lemma is that once a finite number of ``bad'' primes are excluded, then for all remaining prime numbers $q$, the reduction of $H$ modulo $q$ is Dwork-regular over the finite field $\F_q$.
 
We next  recall that a polynomial $P\in \Z[X_1,...,X_m]$ of degree $d \geq 2$ is a \textit{Deligne} polynomial over a finite field $L$ of characteristic $q$ if
    \begin{enumerate}
        \item $q \nmid d$, and
        \item the hypersurface defined by the leading form  $P_d(X_1,\ldots,X_m)=0$  is nonsingular in $\mathbb{P}_{\overline{L}}^{m-1}$.
    \end{enumerate} 
(In the case that $m=1$, (2) is replaced by $P_d(X_1) \not\con 0.$ Recall that the  leading form of a polynomial is the homogeneous part of highest degree.)
A crucial fact we apply later is that after specializing one or more coordinates of a Dwork-regular form, the remaining polynomial is Deligne:

\begin{prop}\label{prop:dwork_deligne} Let $L$ be a finite field and let $H(X_1,...,X_m) \in L[X_1,...,X_m]$ be a homogeneous  polynomial of degree $d$  that is Dwork-regular over $L$ in the variables $X_1,...,X_m$, with $\mathrm{char} L \nmid d$.   Fix $1 \leq r \leq m-1$. Then for any constants $c_1,...,c_r\in L$, $H|_{X_1=c_1,...,X_r=c_r}$ is a Deligne polynomial in $X_{r+1},...,X_m$ over $L$.
 
\end{prop}
In particular, we remark that by Lemma \ref{lem:dwork_equiv}, the constants $c_i$ in Proposition \ref{prop:dwork_deligne} can be 0 in $L$. 
We now turn to the proof of these results.

\subsection{Proof of Lemma \ref{lem:dwork_equiv}}\label{sec:dwork_equiv}

   We suppose $H$ is not Dwork-regular and then show  either (1) or (2) is violated. For $H$ not Dwork-regular, there exists an $a=[a_1:\cdots:a_m]\in \mathbb{P}^{m-1}_{\Bar{L}}$ such that
     \[
            H(a)=0,\qquad (X_i \frac{\partial H}{\partial X_i})(a)=0 \text{ for } 1\leq i \leq m.
      \]
    In particular, $a_j\neq 0$ for some $1\leq j \leq m$ and so for this $j$, $\frac{\partial H}{\partial X_j}(a)=0$. Define the set $S=\{j:a_j\neq 0 \}\subseteq \{1,...,m\}$. If $|S|=1$ then $H_S$ is either identically zero or a monomial in one variable, say in $X_j$. Moreover, when evaluated at the point $a=[0:\cdots : a_j : \cdots :0]\in \mathbb{P}^{m-1}_{\Bar{L}}$,  the monomial in $a_j \neq 0$ satisfies $ H_S(a) = H(a)=0.$ Thus the coefficient of the monomial is zero, and $H_S \con 0$ so that (1) is violated.
    If on the other hand $|S|\geq 2$, say $S = \{ i_1,\ldots, i_{|S|}\}$ then let $b=[a_{i_1}:\cdots :a_{i_{|S|}}] \in \mathbb{P}^{|S|-1}_{\Bar{L}}$. Then upon regarding $H_S$ as a polynomial in $X_{i_1},\ldots, X_{i_{|S|}},$
  $
            H_S(b) = H(a)=0.
       $
   For each $i\in S$, when we evaluate at the point $b$ (or $a$ respectively),
        \[
            \frac{\partial H_S}{\partial X_i}(b)   =             \frac{\partial H}{\partial X_i}(a)    = 0.
        \]
  This produces a singular point on the projective hypersurface  $H_S=0$  in $\mathbb{P}^{|S|-1}_{\overline{L}},$
 violating (2).

    Finally, suppose $H$ is Dwork-regular; we will argue that (1) and (2) must hold, by contradiction. Indeed, suppose  either $H_S$ is identically zero for some $S$ with $|S|=1$ or $H_S=0$ is singular for some $S$ with $|S|\geq 2$. Write $S=\{i_1,...,i_{|S|}\}\subseteq \{1,...,m\}$. For the case $|S|=1$, define $a=[a_1:\cdots : a_m]\in \mathbb{P}^{m-1}_{\Bar{L}}$ by $a_j=1$ if $j\in S$ and 0 otherwise.  Note that $H(a)=H_S(a)=0$. For $i\not\in S$, the coordinate $X_i(a)=0$ while for $i\in S$, at the point $a$,
 \[
            \frac{\partial H}{\partial X_i} (a) = \frac{\partial H_S}{\partial X_i} (a) =0
\]
   since $H_S \con 0$. Consequently,
      \begin{align}\label{eqn:a_dwork}
           H(a)=0,\qquad (X_i \frac{\partial H}{\partial X_i})(a)=0 \text{ for } 1\leq i \leq m,
        \end{align}
        which violates Dwork-regularity of $H$, a contradiction.
  On the other hand, if  $|S|\geq 2$, let $b=[b_{i_1}:\cdots:b_{i_{|S|}}]\in \mathbb{P}^{|S|-1}_{\Bar{L}}$ be such that
 \[
            H_S(b)=0,\qquad  \frac{\partial H_S}{\partial X_i}(b)=0 \text{ for } i\in S.
\]
    Define $a=[a_1:\cdots:a_m]\in \mathbb{P}^{m-1}_{\Bar{L}}$ by $a_j=b_j$ if $j\in S$ and 0 otherwise. 
    Then $H(a)=H_S(b)=0$. Additionally, for  $i\not\in S$, the coordinate $X_i(a)=0$ while for $i\in S$,  $(\partial H/\partial X_i)(a)=(\partial H_S/\partial X_i)(b)=0$. Thus $a$ satisfies \eqref{eqn:a_dwork} and violates Dwork-regularity of $H$, a contradiction.

\subsection{Proof of   Lemma \ref{lemma_Dwork_reduced}}\label{sec:proof_of_lemma}
 Lemma \ref{lemma_Dwork_reduced}  follows by a standard type of argument, which applies a version of  Nullstellensatz; we provide the projective version we apply.
    \begin{lemma}\label{lem:nullstellensatz}
    Let $L$ be a field. Let $I\subseteq L[X_1,...,X_m]$ be a homogeneous ideal. Define $Z(I)=\{x\in \mathbb{P}_{\Bar{L}}^{m-1} : f(x)=0 \text{ for all homogeneous } f\in I\} \subseteq \mathbb{P}_{\Bar{L}}^{m-1}$. Then $Z(I)$ is the empty set if and only if $(X_1^d,...,X_m^d)\subseteq I$ for some $d$.
\end{lemma}
    \begin{proof}
    Suppose $Z(I)=\emptyset$ in $\mathbb{P}^{m-1}_{\Bar{L}}$. Define the affine set $Z_{\mathbb{A}}(I) = \{(a_1,...,a_m)\in \mathbb{A}^{m}_{\Bar{L}}: f(a)=0 \text{ for all homogeneous } f\in I\}$. Then $Z_{\mathbb{A}}(I) =(0,...,0)$ and so for each $i$ the monomial $X_i$ vanishes on $Z_{\mathbb{A}}(I) =(0,...,0)$. Then by affine Nullstellensatz  \cite[Theorem 1.5]{Lan02},  $X_i^{d_i}\in I$ for some $d_i$. For the other direction, suppose $(X_1^d,...,X_m^d)\subseteq I$ for some $d$. Let $x\in Z(I)$ so that $f(x)=0$ for all homogeneous $f\in I$. Then in particular the monomials $X_i^d$   vanish on $x$ and so $x_i=0$ for all $i$. This is a contradiction to $x$ belonging to $\mathbb{P}_{\overline{L}}^{m-1}$, so $Z(I) = \emptyset.$
    \end{proof}
 Now to prove Lemma \ref{lemma_Dwork_reduced}, let us first recall some terminology. For the field $L$, we will denote a finite place by $\mathfrak{p}$ and its associated valuation  by $v_\fp$. For example, in the case we will apply,  $L=\Q$ so if we pick a finite place  (prime number) $p$, then the associated valuation is $v_p(x) = \max\{a\in \Z : p^a|x\}$ for  $x \in \Q$; for example, $v_p(x) \geq 0$ for all primes $p$ precisely when $x$ is an integer. When working with polynomials with rational coefficients, it  can be convenient to multiply an identity of polynomials  by a sufficiently large integer (say $N$) to ``clear denominators;'' alternatively, we could work in an enlarged set of ``integers'' that include rational numbers with denominators only divisible by  primes $p|N$. For example, we could consider  rational numbers with denominators only divisible by powers of  $5$ and $7$; we  call the set of all such rationals $S$-integers for the set $S=\{5,7\}$.
 In general let $S$ be a finite set of finite places of a field $L$. An associated ring $\mathcal{O}_S$ called the $S$-integers is defined by  $\mathcal{O}_S = \{x\in L: v_\fp(x)\geq 0 \text{ for all finite places }\fp \not\in S\}$. Finally, for any finite place $\fp \not\in S$ we may consider the quotient $\mathcal{O}_S/\fp$, which is the residue field. In the case $L=\Q$ where we apply Lemma \ref{lemma_Dwork_reduced}, given a particular prime $p \not\in S$, $\Ocal_S/\fp = \Ocal_S/p\Ocal_S$ is isomorphic to $\mathcal{O}_L/\fp=\Z/p\Z \cong \F_p$ since the map $ \Ocal_S \rightarrow \Z/p\Z$ given by  $a/b \mapsto a\pmod p$ is surjective, and its kernel is $p\Ocal_S$.

 To prove the lemma, initially define $S$ to be the set of places $\fp$ such that either $v_\fp(c)<0$ for some coefficient $c$ of $H$ or $v_\fp(c)>0$ for all coefficients of $H$. Then $H\in \mathcal{O}_S[X_1,...,X_m]$ and hence $X_i \frac{\partial H}{\partial X_i} \in \mathcal{O}_S[X_1,...,X_m]$. Define the ideal
        \[
            I:=\left(H,X_1\frac{\partial H}{\partial X_1},...,X_m\frac{\partial H}{\partial X_m}\right) \subseteq \mathcal{O}_S[X_1,...,X_m].
        \]
    Since by assumption $H$ is Dwork-regular over $L$, $Z(I)=\emptyset$ in $\mathbb{P}^{m-1}_{\Bar{L}}$ where we view $I$ as an ideal in $L[X_1,...,X_m]$. Then by Lemma \ref{lem:nullstellensatz}, $(X_1^d,...,X_m^d)\subseteq I$ for some $d$. In particular for each $i$, there exist $Q_i,Q_{i,1},...,Q_{i,m}\in L[X_1,...,X_m]$ such that
       \[
            X_i^d = HQ_i + X_1\frac{\partial H}{\partial X_1}Q_{i,1} + \cdots + X_m\frac{\partial H}{\partial X_m}Q_{i,m}.
      \]
    Now we enlarge the set $S$ so that it includes places $\fp$ such that $v_\fp(c)<0$ for some coefficient of at least one of $Q_i,Q_{i,1},...,Q_{i,m}$. Then $(X_1^d,...,X_m^d)\subseteq I \subseteq \mathcal{O}_S[X_1,...,X_m]$. For $\fp \not\in S$,
        \[
            X_i^d \equiv HQ_i + X_1\frac{\partial H}{\partial X_1}Q_{i,1} + \cdots + X_m\frac{\partial H}{\partial X_m}Q_{i,m} \modd{\fp}
        \]
        and so $(X_1^d,...,X_m^d) \modd{\fp}\subseteq (H,X_1\frac{\partial H}{\partial X_1},...,X_m\frac{\partial H}{\partial X_m}) \modd{\fp}$; that is to say, working over the residue field $L'=\mathcal{O}_S/\fp$ and viewing the ideals now in $L'[X_1,\ldots,X_m],$ the inclusion $(X_1^d,...,X_m^d) \subseteq (H,X_1\frac{\partial H}{\partial X_1},...,X_m\frac{\partial H}{\partial X_m})$ holds. We apply Lemma \ref{lem:nullstellensatz} again, now with $L'=\mathcal{O}_S/\fp,$ to deduce that $H,X_1\frac{\partial H}{\partial X_1},...,X_m\frac{\partial H}{\partial X_m}$ have no common zeros in $\mathbb{P}^{m-1}_{\overline{\mathcal{O}_S/\fp}}$ and hence by definition $H$ is Dwork-regular over $\mathcal{O}_S/\fp$ in the variables $X_1,\ldots,X_m$.

\subsection{Proof of Proposition \ref{prop:dwork_deligne}}
 Let $c_1,...,c_{r}\in L$ be given, and let $G\in L[X_{r+1},\ldots,X_m]$ denote the polynomial $H(c_1,...,c_r,X_{r+1},\ldots,X_m)$.  Note that $G$ has degree $d:=\deg H$; by the remark following Lemma  \ref{lem:dwork_equiv}, $H$ must contain a nonzero multiple of $X_{j}^d$ for each $j$ and in particular for $r+1 \leq j \leq m$, so $\deg G=d$. In particular, $\mathrm{char} L \ndiv \deg G$. 
 Next,   $H$ can be written as
$
        H = G_dF_0 + G_{d-1}F_1 + \cdots + G_0F_d
$
where $G_i\in L[X_{r+1},\ldots,X_m]$ is homogeneous of degree $i$ and $F_i\in L[X_1,...,X_r]$ is homogeneous of degree $i$ (in particular $F_0 \con 1, G_0 \con 1$). The leading form of $G$ is then precisely $G_d(X_{r+1},\ldots,X_m)$, and $G_d= H_S$ for $S=\{r+1,\ldots,m\}$, in the notation of Lemma \ref{lem:dwork_equiv}. Thus the leading form of $G$ defines a nonsingular hypersurface in $\mathbb{P}_{\overline{L}}^{m-r-1}$ (or is not the zero polynomial, if $m-r=1$), by Lemma \ref{lem:dwork_equiv}, completing the proof.

 \subsection{A lower bound for a partial derivative}\label{sec_gradient}
When constructing   counterexamples, it will be important to find a lower bound for a partial derivative of the leading form $P_k$ of the polynomial symbol. (Explicitly, this  is used to guarantee that for each $x$ in a small neighborhood of the origin in $\R^n$, a choice of $t$ meets all the requirements of (\ref{eqn:t_constraints}) simultaneously.) As a motivating example, in the special case where the symbol has diagonal leading form $P_k(X) = X_1^k + \cdots + X_n^k,$ one immediately has $(\partial_1 P_k)(R,X_2,\ldots,X_n) \gg_k R^{k-1}$ for all $(X_2,\ldots,X_n) \in \R^{n-1}$, which suffices for our application; here and throughout the paper, $R$ is a large parameter that will tend to infinity at the end of the argument. If $P_k$ has higher intertwining rank, we must proceed differently; the following lemma provides an integral point where the partial derivative is sufficiently large.
\begin{lemma}\label{lem:coefficients}
  Let $P_k \in \Q[X_1,\ldots,X_n]$ be a Dwork-regular form of degree $k \geq 2$ with intertwining rank $r<n$; without loss of generality, say  $X_1$ intertwines with $X_1,X_2,...,X_r$ but not with $X_{r+1},\ldots,X_n$. 
  Then there exists a tuple $(M_1,...,M_r)\in \Z^r$, with  $M_i\geq 1$ for all $i$, such that for all $R \geq 1,$
   \[
        |(\partial_1 P_k)(M_1R,M_2R,...,M_rR,X_{r+1},\ldots,X_n)| \gg R^{k-1},
\]
    uniformly in $(X_{r+1},\ldots,X_n) \in \R^{n-r}.$
 \end{lemma}
   
 Let $X'=(X_2,...,X_n)$ and rewrite $P_k$ in terms of its coefficients $c_\be$ as 
   \[
        P_k(X_1,\ldots, X_n) 
        =\sum_{j=0}^{k} X_1^{k-j} \sum_{|\beta|=j} c_\beta  X'^{\beta}
    \]
where $\beta$ is a multi-index $(\beta_2,...,\beta_n)$ with order $|\be| = \be_2 + \cdots + \be_n,$ and $X'^\beta=X_2^{\beta_2}\cdots X_n^{\beta_n}$. 
By the hypothesis that $X_1$ does not intertwine with $X_{r+1},...,X_n$, for each $|\be| <k,$ $c_\be=0$ for all $\be$ with $\be_\ell>0$ for some $\ell >r.$
  Consequently, the derivative $\partial_1 P_k$   is a function of $X_1,\ldots, X_{r}$ and is independent of $X_j$ for $j>r$:
    \begin{align}\label{eqn:partial_Pk}
       ( \partial_1 P_k)(X_1,\ldots,X_n) 
        = \sum_{j=0}^{k-1} (k-j) X_1^{k-j-1} \sum_{|\beta|=j} c_\beta  X_2^{\be_2}\cdots X_{r}^{\beta_{r}}.
    \end{align}
    Then for any value of a parameter $R \geq 1$, after plugging in any $(M_1,...,M_{r})\in \Z^{r}$, 
  \begin{align*}
        (\partial_1 P_k)(M_1R,\ldots,M_{r}R,X_{r+1},\ldots,X_n)&=\sum_{j=0}^{k-1} (k-j)(M_1R)^{k-j-1} \sum_{|\beta|=j} c_\beta   (M_2R)^{\beta_2}\cdots (M_{r}R)^{\beta_{r}} \\
        &=R^{k-1}\sum_{j=0}^{k-1} (k-j)M_1^{k-j-1} \sum_{|\beta|=j} c_\beta M_2^{\beta_2}\cdots M_{r}^{\beta_{r}}.
            \end{align*}
Thus we can conclude $|(\partial_1 P_k)(M_1R,...,M_{r}R,X_{r+1},\ldots,X_n)|\gg R^{k-1}$   (uniformly in $X_{r+1},\ldots,X_n$) as long as
  \[
        \sum_{j=0}^{k-1} (k-j)M_1^{k-j-1} \sum_{|\beta|=j} c_\beta M_2^{\beta_2}\cdots M_{r}^{\beta_{r}}  \neq 0.
 \]
From \eqref{eqn:partial_Pk}, we see that this condition is equivalent to 
$
        (\partial_1 P_k)(M_1,...,M_{r})\neq 0
$
where we view $\partial_1P_k$ as an element of $\Q[X_1,...,X_{r}]$.  Hence it remains to prove that there exists an integral point $(M_1,...,M_{r})\in \Z^{r}$ with $M_i\geq 1$ for all $i$ such that $(\partial_1 P_k)(M_1,...,M_{r})\neq 0$.

First we note that $\partial_1 P_k$ is not the zero polynomial in $X_1,\ldots, X_r$; this follows from Dwork-regularity, since by (1) of Lemma \ref{lem:dwork_equiv}, $P_k$ contains a term $cX_1^k$ for some nonzero $c \in \Q$. If $r=1$, $(\partial_1 P_k)(X_1)$ has at most $k-1$ roots, so the claim holds. For $r\geq 2$, by a trivial upper bound (see e.g. \cite[Lemma 10.1]{BCLP22} for a standard statement), for any $B\geq 1,$ there are at most $\ll B^{r-1}$ integral solutions in $[1,B]^r$ to $(\partial_1 P_k)(X_1,\ldots,X_r)=0.$ Since there are $B^r$ integral points in that box, there exists a sufficiently large $B$ such that there is an integral point, say $(M_1,\ldots,M_r) \in [1,B]^r$ with $(\partial_1 P_k)(M_1,\ldots,M_r)\neq 0.$ The lemma is proved.

 \subsection{Invariance under $\mathrm{GL}_n(\Q)$: nonsingular to Dwork-regular}\label{sec_change_var}
  Let $P\in \R[X_1,\ldots,X_n]$ be given, with leading form $P_k \in \Q[X_1,\ldots,X_n]$.   Now let $A \in \GL_n(\Q)$ and let $Q$ denote the polynomial after changing variables, i.e. $Q(\eta_1,...,\eta_n)=P(\xi_1,...,\xi_n)$ with $\xi_i = \sum_j a_{ij}\eta_j$ or equivalently $\eta=A^{-1}\xi$. (In particular, by \cite[Lemma 3.1]{Kat09}, if $P_k$ is nonsingular, there exists a choice of $A$ so that $Q(\eta_1,\ldots,\eta_n)$ is Dwork-regular over $\Q$ in $\eta_1,\ldots, \eta_n$.) 
  Then for any fixed $t>0,$
 \[
        T_t^Pf(x) 
       = \frac{1}{(2\pi)^n}\int_{\R^n} \hat{f}(\xi) e^{i(\xi \cdot x + P(\xi)t)} d\xi\\
        =\frac{1}{(2\pi)^n}\int_{\R^n} \hat{f}(A\eta) e^{i((A\eta) \cdot x + Q(\eta)t)}|A| d\eta.
\]
Now define a function $g$ such that $\hat{g}(\eta)  = |A|\hat{f}(A\eta)$ where $|A|$ denotes the determinant of $A$; equivalently $g(x) =  f((A^{-1})^Tx).$    Then 
 \[ (T_t^Pf)(x) = \frac{1}{(2\pi)^n}\int_{\R^n} \hat{g}(\eta) e^{i(\eta \cdot A^Tx + Q(\eta)t)}  d\eta = (T_t^Qg)(A^Tx).\]
Now let $\tau: \R^n \maps (0,1)$ be a measurable stopping-time function, so that proving $f(x) \mapsto \sup_{0<t<1}|T_t^Pf(x)|$ is bounded from $H^s$ to $L^1_{\mathrm{loc}}$ is equivalent to proving that  
$f(x) \mapsto  |T_{\tau(x)}^Pf(x)|$ is bounded from $H^s$ to $L^1_{\mathrm{loc}}$, independent of the function $\tau$. For a given stopping-time function $\tau$, the computation above shows that 
\[ \|T_{\tau(x)}^Pf(x)\|_{L^1(B_n(0,1))} = \int_{B_n(0,1)} |(T_{\tau(x)}^Qg)(A^Tx)|dx
= |A^T|^{-1} \int_{A^TB_n(0,1)} |(T_{\sig(u)}^Qg)(u)|du,\]
for the stopping-time $\sig(u) = \tau((A^T)^{-1}u).$
Thus 
\beq\label{TP_bdd} \|T_{\tau(x)}^Pf(x)\|_{L^1(B_n(0,1))} \leq C_s \|f\|_{H^s}
\eeq
is true for all $f \in H^s$, uniformly over all stopping-time functions, if and only if 
\beq\label{TQ_bdd}  \|T_{\sig(x)}^Qg(x)\|_{L^1(A^TB_n(0,1))} \leq C_s C'_A \|g\|_{H^s}
\eeq
is true for all $g \in H^s$, uniformly over all stopping-time functions. (Of course,    $\|g\|_{H^s} \ll_{A} \|f\|_{H^s} \ll_{A} \|g\|_{H^s}.$)
In particular, if we show for a given real $s>0$ that there is no constant $C_s$ such that 
\[\|\sup_{0<t<1}|T_t^Qg(x)|\|_{L^1(B_n(0,1))} \leq C_s    \|g\|_{H^s}\]
for all $g \in H^s$, then (\ref{TQ_bdd}) fails for all constants $C_s, C_A'$ and consequently (\ref{TP_bdd}) fails for all constants $C_s$. This invariance under $\GL_n(\Q)$ changes of variables shows that it is reasonable to study maximal operators for the class of Dwork-regular polynomial symbols in place of the class of nonsingular polynomial symbols, as we do. (The invariance demonstrated above holds for $\GL_n(\R)$ as well.)

 \subsection{Verification of Corollary \ref{corollary_decomp}}\label{sec_cor}
 Suppose as in the hypothesis that $P_k(X_1,\ldots,X_n)$ is decomposable over $\Q$, so that after an appropriate $\GL_n(\Q)$ change of variables, $P_k(X_1,\ldots,X_n) = Q_1(X_1,\ldots,X_m) + Q_2(X_{m+1},\ldots,X_n)$  holds for some $1 \leq m<n.$
 Note that  $P_k$ is nonsingular as a function of $X_1,\ldots,X_n$ if and only if $Q_1$ is nonsingular as a function of $X_1,\ldots, X_m$ and $Q_2$ is nonsingular as a function of $X_{m+1},\ldots,X_n$. By the disjointness of the variables in $Q_1$ and $Q_2$, we can apply $\GL_m(\Q)$ (respectively $\GL_{n-m}(\Q)$) changes of variables separately to $Q_1$ (respectively $Q_2$) so that both become Dwork-regular over $\Q$. This provides a block diagonal $\GL_n(\Q)$ transformation that makes $P_k$ Dwork-regular and still decomposed into a form in $X_1,\ldots,X_m$ and a second form in $X_{m+1},\ldots,X_n$. In particular, its intertwining rank remains bounded above by $r\leq \lfloor n/2 \rfloor$. Since the parameter $\delta(n,k,r)$ in Theorem \ref{thm:main2} is decreasing as a function of $r$, then  $\delta(n,k,r)\geq \delta(n,k,n/2)$ for $r\leq \lfloor n/2 \rfloor$, and this verifies the corollary.

\subsection{Dispersivity}\label{sec:dispersive}
A description of the broad principle of dispersion, in the context of an initial value problem like (\ref{PDE}),   can be found for example in \cite[\S 3.5]{Pal97} or \cite[Principle 2.1]{Tao06}. The following lemma verifies that for each real symbol considered in the main theorem, (\ref{PDE}) is  a dispersive PDE, in the sense of  the criterion presented in  \cite[Theorem 4.1]{KPV91}.
\begin{lemma}\label{lem:dispersive}
 Let $P(X_1,...,X_n)\in \R[X_1,...,X_n]$ be a polynomial of degree $k\geq 2$ such that its leading form $P_k(X_1,\ldots,X_n) \in \Q[X_1,\ldots,X_n]$ is Dwork-regular over $\Q$ in the variables $X_1,...,X_n$. Then $\nabla P_k(x_1,...,x_n) \neq 0$ for all  $(x_1,...,x_n)\in \R^n \setminus \{0\}$. Further, there exists a finite $M \in \mathbb{N}$ such that for each $i=1,...,n,$ for all $(c_1,...,c_{i-1},c_{i+1},...,c_n)\in \R^{n-1}$ and $c'\in \R$, the equation 
     \beq\label{P_sol}
         P(c_1,...,c_{i-1},x,c_{i+1},...,c_n)=c',  
        \eeq
   has at most $M$ solutions.
\end{lemma}

    \begin{proof} Since $P_k$ is Dwork-regular over $\Q$, then $P_k$ is nonsingular over $\Q$, so that $\nabla P_k(X_1,\ldots,X_n)=0$ has no nontrivial solutions over $\overline{\Q}$; since $P_k$ has rational coefficients, this implies that there are no nontrivial solutions over $\R$ either.   Fix an $1 \leq i \leq n.$   By the remark following Lemma \ref{lem:dwork_equiv}, the leading form $P_k$ contains a term that is a nonzero multiple of   $X_i^k$. Hence   $P(c_1,...,c_{i-1},x,c_{i+1},...,c_n)-c'$ contains a nonzero multiple of $x^k$, so that (\ref{P_sol}) has at most $k$ solutions. 
    \end{proof}

\section{Upper and lower bounds for exponential sums}\label{sec:exponential}

Now we prove three results about exponential sums involving a Deligne polynomial, which by Proposition \ref{prop:dwork_deligne} apply to Dwork-regular forms over $\Q$ with fixed variables, for all sufficiently large prime moduli. First, we prove in Proposition \ref{prop:bound_T} that complete exponential sums with a Deligne polynomial in the phase   are often ``large''. Second, we  prove in Proposition \ref{prop:bound_incomplete} that incomplete exponential sums are ``not too large'' (and in particular, only a logarithmic factor larger than complete exponential sums). Third, in Proposition \ref{prop:2.7} we approximate an exponential sum with real coefficients by complete exponential sums, up to an acceptable error.  

\subsection{A lower bound for many complete exponential sums}

\begin{prop}\label{prop:bound_T} Let $Q_k(X_1,...,X_m)\in \Z[X_1,...,X_m]$ be a polynomial of degree $k\geq 2$ and suppose that for all primes $q \geq K_1(Q_k)$, the reduction of $Q_k$ modulo $q$ is  a Deligne polynomial over $\F_q$.  Define for each prime $q$ and  $(a,b) \in \F_q \times \F_q^{m}$,
   \[
        \bfT(a,b;q):=\sum_{x \modd{q}^m} e\left(\frac{2\pi}{q}\left(a Q_k(x) + b \cdot x\right)\right).
  \]
Then there exist constants $0<\alpha_1, \alpha_2<1$ with $\alpha_2=\alpha_2(k,m)$, and a constant  $K_2(k,m)$ such that for every prime $q > \max\{k,  K_1(Q_k), K_2(k,m)\}$, there exist at least $\alpha_2 q^{m+1}$ choices of $(a,b)\in \F_q \times \F_q^{m}$ such that
    \begin{align}\label{eqn:T_large}
       \alpha_1 q^{m/2} \leq  |\bfT(a,b;q)| \leq (k-1)^mq^{m/2} .
    \end{align}
\end{prop} 
In fact, we may take $\al_1=1/2$ and any $\al_2 \leq (1/8)(k-1)^{-2m}.$  
The Weil-Deligne bound is a key ingredient in the proof, which we now recall, following \cite[Thm. 8.4]{Del74} as stated in \cite[Thm. 11.43]{IwaKow04}.

\begin{lemma}\label{lem:deligne}   Let $f(X_1,...,X_m)\in \Z[X_1,...,X_m]$ be a Deligne polynomial of degree $k\geq 2$   over $\F_q$ for $q$ prime.  
Then
\[
        \left|\sum_{\mathbf{x}\in \F_q^m} e(\frac{2\pi}{q}f(x_1,...,x_m))\right|\leq (k-1)^{m}q^{m/2}.
 \]
\end{lemma}
The proof of the proposition is a mild variation on \cite[Prop. 2.2]{ACP23}.
We first claim that
    \beq\label{T_average}
            \sum_{\substack{a \modd{q}\\b \modd{q}^m}} |\bfT(a,b;q)|^2 = q^{2m+1}.
        \eeq
The left-hand side may be expanded as
\[ \sum_{x, \tilde{x}} \sum_{a\modd{q}} e\left(\frac{2\pi a}{q}(Q_k(x) - Q_k(\tilde{x})) \right) \sum_{b\modd {q}^{m}} e\left(\frac{2\pi  }{q}b \cdot (x-\tilde{x}) \right).
\]
The innermost sum vanishes unless $x_i\equiv \tilde{x}_i\modd{q}$ for all $1\leq i\leq m$; in this case the left-hand side evaluates to $q^{2m+1}$, as claimed. Next we observe that since $Q_k$ is Deligne over $\F_q$ of degree $k\geq 2$, so is $aQ_k(x)+b\cdot x$ for $a\neq 0  \in \F_q$, so we can apply Lemma \ref{lem:deligne} to bound $\bfT(a,b;q)$; for $a=0$ the sum vanishes unless $b=0$ and we address this below. Since (\ref{T_average}) shows that $\bfT(a,b;q)$ is of size $q^{m/2}$ on average, and the Weil-Deligne bound shows it cannot ever be much larger, we can deduce that most of the time it is not much smaller either. Precisely, suppose for a given pair $0<\alpha_1,\alpha_2<1$ there are   $<\alpha_2q^{m+1}$ choices of $(a,b) \modd{q}^{m+1}$ such that $|\bfT(a,b;q)|\geq \alpha_1 q^{m/2}$. Then
\[
        \sum_{\substack{a \modd{q}\\b \modd{q}^m}} |\bfT(a,b;q)|^2 \\
        =   q^{2m} 
        + \sum_{\substack{a \not\con 0\\ |\bfT(a,b;q)|\geq \alpha_1 q^{m/2}}} |\bfT(a,b;q)|^2 
        + \sum_{\substack{a\not\con 0\\ |\bfT(a,b;q)|< \alpha_1 q^{m/2}}} |\bfT(a,b;q)|^2.
  \]
Here the first term is the contribution from $a \con 0$.
The first sum on the right-hand side is $<\alpha_2 q^{m+1} ((k-1)^mq^{m/2})^2$; the second sum on the right-hand side is  $<q^{m+1} (\alpha_1 q^{m/2})^2$. Hence  
 \[
        \sum_{\substack{a \modd{q}\\b \modd{q}^m}} |\bfT(a,b;q)|^2
         <  
        (q^{-1}+ \alpha_2  (k-1)^{2m} +\alpha_1^2)q^{2m+1}  <(1/3+\alpha_2(k-1)^{2m}+\alpha_1^2)q^{2m+1}.
  \]
 This is $<q^{2m+1}$, contradicting (\ref{T_average}), for sufficiently small $\alpha_1,\alpha_2$; for example, we may take $\al_1=1/2$ and any $\al_2 \leq (1/4)(k-1)^{-2m}.$ For such $\al_1,\al_2$, there must then be $\geq \al_2 q^{m+1}$ choices of $(a,b) \modd{q}^{m+1}$ for which the left-hand inequality in (\ref{eqn:T_large}) holds. 
 
 Finally, the only case in which $|\mathbf{T}(a,b;q)|>(k-1)^mq^{m/2}$ is when $a=0 \in \F_q$ (in which case $\mathbf{T}(a,b;q)$ is a linear exponential sum which vanishes unless $b = 0 \in \F_q^m$). Thus aside from $(a,b)=(0,0)$, all those $(a,b)$ satisfying the lower bound in (\ref{eqn:T_large}) also satisfy the upper bound. We can summarize this by saying that at least $\al_2 q^{m+1}$ choices satisfy both bounds, with the modified choice $\al_2= (1/8)(k-1)^{-2m},$ as long as we assume $q$ is sufficiently large that $(1/4)(k-1)^{-2m}q^{m+1} \geq 2,$ which is true for all $q \geq K_2$, for an appropriate choice of $K_2=K_2(k,m).$
 The proposition is proved.

 \subsection{An upper bound for incomplete exponential sums}
 
 To show the maximal operator associated to $T_t^{P}f$ is large, we will repeatedly approximate integrals and sums by complete exponential sums. We first state general formulas for partial summation and partial integration; these are proved simply by iteration one coordinate at a time, and we omit the proofs. Given a subset $J\subseteq \{1,\ldots,n\}$, we define $I = \{1,\ldots,n\}\setminus J$ and use the notation $(\mathbf{x}_{(J)},\mathbf{x}_{(I)}) \in \R^{n}$ to indicate $\xbf_{(J)} \in \R^{|J|}$ is indexed by $j \in J$ and $\xbf_{(I)} \in \R^{n-|J|}$ is indexed by $j \in \{1,\ldots,n\}\setminus J.$  

\begin{lemma}\label{lem:partial_summation}
Let $a(\mathbf{m})$ be a sequence of complex numbers indexed by $\mathbf{m}=(m_1,...,m_n)\in \Z^n$. Let $h(\mathbf{y})$ be a $C^{(n)}$ function on $\R^n$ such that for every $\boldsymbol{\kappa}= (\kappa_1,...,\kappa_n) \in \{0,1\}^n$, there exists a positive real number $B_{\boldsymbol{\kappa}}$ such that for
    \begin{align*}
        \partial_{\boldsymbol{\kappa}} h(\mathbf{y}):=\frac{\partial^{\kappa_1+\cdots+\kappa_n}}{\partial y_1^{\kappa_1} \cdots \partial y_n^{\kappa_n}} h(\mathbf{y}),
    \end{align*}
we have
$
        \left|\partial_{\boldsymbol{\kappa}} h(\mathbf{y})\right| \leq B_{\boldsymbol{\kappa}}$
        uniformly in $\mathbf{y}\in [\mathbf{M},\mathbf{M}+\mathbf{N}],$
        where $\mathbf{M}=(M,M,\ldots,M)$ and $\mathbf{N} = (N_1,N_2,\ldots,N_n).$
Then
    \begin{align*}
        \sum_{\mathbf{M}\leq \mathbf{m}\leq \mathbf{M+N}} a(\mathbf{m}) e(h(\mathbf{m})) = e(h(\mathbf{M+N}))\sum_{\mathbf{M}\leq \mathbf{m}\leq \mathbf{M+N}} a(\mathbf{m})  + E
    \end{align*}
    where
        \begin{align*}
            |E|\ll_n \sup_{\substack{J\subseteq \{1,...,n\}\\|J| \geq 1}}\left\{ \prod_{j \in J}N_j \cdot \sup_{\mathbf{u}_{(J)}\leq \mathbf{N}_{(J)}}|A (\mathbf{u}_{(J)},\mathbf{N}_{(I)})| \cdot   \sup_{1\leq \ell \leq |J|} \sup_{\substack{\boldsymbol{\alpha_1},...,\boldsymbol{\alpha_\ell} \in \{0,1\}^n\\\boldsymbol{\alpha_1}+ \cdots +\boldsymbol{\alpha_\ell} = \boldsymbol{1}_J}} \prod_{i=1}^\ell B_{\boldsymbol{\alpha_i}} \right\},
        \end{align*}
        in which 
        \[A (\mathbf{u}_{(J)},\mathbf{N}_{(I)})=\sum_{\substack{\mathbf{M}_{(I)}\leq \mathbf{m}_{(I)} \leq \mathbf{M}_{(I)}+\mathbf{N}_{(I)}\\
        \mathbf{M}_{(J)}\leq \mathbf{m}_{(J)} \leq  \mathbf{M}_{(J)}+\mathbf{u}_{(J)}}} a(\mathbf{m}_{(J)},\mathbf{m}_{(I)}).\]
\end{lemma}

\begin{lemma}\label{lem:partial_integration} Let $a<b$ be real numbers. Let $f(\mathbf{t})$ be an integrable function supported on $[a,b]^n$ and let $h(\mathbf{t})$ be a $C^{(n)}$ function on $\R^n$ such that for every $\boldsymbol{\kappa} = (\kappa_1,...,\kappa_n)\in \{0,1\}^n$, there exists a positive real number $B_{\boldsymbol{\kappa}}$ such that
$
        \left|\partial_{\boldsymbol{\kappa}} h(\mathbf{t})\right| \leq B_{\boldsymbol{\kappa}},
$ uniformly in $\mathbf{t} \in [a,b]^n.$
Then
    \begin{align*}
          \int_{[a,b]^n} f(\mathbf{t}) e(h(\mathbf{t})) d\mathbf{t} = e(h(\mathbf{b}))  \int_{[a,b]^n} f(\mathbf{t})  d\mathbf{t} + E
    \end{align*}
where
    \begin{align*}
        |E|\ll_n  &        \sup_{\substack{J\subseteq \{1,...,n\}\\|J| \geq 1}} \left \{(b-a)^{|J|}\sup_{\mathbf{u}_{(J)} \leq \mathbf{b}_{(J)}} |F(\mathbf{u}_{(J)},\mathbf{b}_{(I)})| \cdot  \sup_{1\leq \ell \leq |J|} \sup_{\substack{\boldsymbol{\alpha_1},...,\boldsymbol{\alpha_\ell} \in \{0,1\}^n\\\boldsymbol{\alpha_1}+ \cdots +\boldsymbol{\alpha_\ell} = \boldsymbol{1}_J}} \prod_{i=1}^\ell B_{\boldsymbol{\alpha_i}} \right\},
    \end{align*}
    in which 
    \[F(\mathbf{u}_{(J)},\mathbf{b}_{(I)})
        =\int_{[a,b]^{|I|}}  \int \cdots \int_{[a,u_{j}],j\in J} f(\mathbf{t}_{(J)},\mathbf{t}_{(I)}) d\mathbf{t}_{(J)}d\mathbf{t}_{(I)}.\]
\end{lemma}
 
We bound an incomplete exponential sum by completing the sum and applying the Weil bound:
\begin{prop}\label{prop:bound_incomplete} Let $Q_k(X_1,...,X_n)\in \Z[X_1,...,X_n]$ be a Deligne polynomial of degree $k\geq 2$ over $\F_q$ for a prime $q$.   Let $J\subseteq \{1,...,n\}$ and $I=\{1,...,n\}\setminus J$. Then for $\Hbf=(H_1,...,H_n)$ with $1\leq H_i \leq q$,
\[ \left|\sum_{\substack{\mathbf{1}_{(J)} \leq \bfm_{(J)}\leq \mathbf{q}_{(J)},\\\mathbf{1}_{(I)} \leq \bfm_{(I)}\leq \mathbf{H}_{(I)}}} e\left(\frac{2\pi}{q} Q_k(\bfm_{(J)},\bfm_{(I)})\right)\right| \ll_{k,n} q^{n/2}(\log q)^{|I|} .
       \]
\end{prop}
    
 The sum on the left-hand side may be written as
 \[ \sum_{\substack{\mathbf{1}_{(J)} \leq \bfm_{(J)}\leq \mathbf{q}_{(J)}}}\sum_{\substack{\mathbf{1}_{(I)} \leq \mathbf{a}_{(I)}\leq \mathbf{q}_{(I)}}} e\left(\frac{2\pi}{q} Q_k(\bfm_{(J)},\mathbf{a}_{(I)}) \right) \sum_{\substack{\mathbf{1}_{(I)} \leq \bfm_{(I)}\leq \mathbf{H}_{(I)},\\ \bfm_{(I)}\equiv \mathbf{a}_{(I)}\modd{q}^{|I|}}} 1.\]
By the identity 
        \[
            \mathbf{1}_{m\equiv a \modd{q}} = \frac{1}{q} \sum_{1\leq h  \leq q} e\left(\frac{2\pi }{q}h\cdot (m-a)\right)
        \]
   we can expand the sum as
        \begin{align*}
          \frac{1}{q^{|I|}}     \sum_{ \mathbf{h}_{(I)} \leq \mathbf{q}_{(I)}} \left(
            \sum_{\substack{  \bfm_{(J)} \leq \mathbf{q}_{(J)},\\  \mathbf{a}_{(I)} \leq \mathbf{q}_{(I)}}} 
            e\left(\frac{2\pi }{q}(Q_k(\bfm_{(J)},\mathbf{a}_{(I)})-\mathbf{h}_{(I)}\cdot \mathbf{a}_{(I)})\right) \sum_{\mathbf{m}_{(I)} \leq \mathbf{H}_{(I)}} e\left(\frac{2\pi}{q} \mathbf{h}_{(I)}\cdot \bfm_{(I)}\right) \right).
        \end{align*}
    Since $Q_k(X_1,\ldots,X_n)$ is a Deligne polynomial of degree $k \geq 2$, so is $Q_k(X_1,\ldots,X_n) - \hbf_{(I)} \cdot \mathbf{X}_{(I)}$ so by Lemma \ref{lem:deligne}, the complete sum over $\mbf_{(J)},\mathbf{a}_{(I)}$ is bounded above by $(k-1)^n q^{n/2}$. 
    For the remaining double sum over $\mathbf{h}_{(I)}$ and $\bfm_{(I)}$, we recall that for each $ h,H \geq 1,$
   \[
        \sum_{1\leq m \leq H} e(2\pi hm/q) \ll \min \{H, \|h/q\|^{-1}\}.
 \]
    (Here $\|t\|$ temporarily indicates the distance from $t$ to the nearest integer.)
By considering cases $1\leq h \leq q/2$ and $q/2<h\leq q$, note that $\sum_{1\leq h \leq q} \min \{H, \|h/q\|^{-1}\} \ll q \log q$. Hence the double sum over $\mathbf{h}_{(I)}$ and $\bfm_{(I)}$ is bounded above by $(q\log q)^{|I|}$. In total, the sum considered in the lemma is thus bounded by 
$\ll (k-1)^n q^{n/2}(\log q)^{|I|}$, and the proof is complete.

 \subsection{Approximation of a sum with real coefficients by complete sums}
We next approximate an exponential sum with real coefficients in the linear term by complete exponential sums with prime moduli.  
\begin{prop}\label{prop:2.7}
    Let $Q_k(X_1,...,X_m)\in \Z[X_1,...,X_m]$ be a Deligne polynomial of degree $k\geq 2$ over $\F_q$ for a prime $q$. Let $1\leq a <q$. Let $V \geq 0$ and $y \in \R^m$, and suppose that for each $1\leq i \leq m$ there exists $1\leq b_i \leq q$   such that $|y_i-2\pi b_i/q|<V$. Then for any $\bfM=(M,...,M)\in \R^{m}$ and  $\bfN=(N_1,...,N_m)\in \R^{m}$ with $0 \leq N_i \leq N$,
    \[
            \left|\sum_{\bfM \leq \bfm \leq \bfM+\bfN} e\left(\frac{2\pi a}{q} Q_k(\bfm)+ y \cdot \bfm\right)\right| 
            = \prod_{i=1}^m\left\lfloor \frac{N_i}{q} \right\rfloor \cdot 
             \left| \sum_{\bfm \modd q^m} e\left(\frac{2\pi}{q}(a Q_k(\bfm)+b \cdot \bfm)\right) \right| + E
    \]
    where,  under the assumption $VN \leq 1$, 
 \[
         |E|\ll_{k,m}   VN  \cdot \prod_{i=1}^m\left\lfloor  \frac{N_i}{q}\right\rfloor \cdot
           q^{m/2}
           +\sup_{|J|<m}  \prod_{j \in J}\left\lfloor \frac{N_j}{q}\right\rfloor \cdot
             q^{m/2} (\log q)^{m-|J|} .\]

\end{prop}
\begin{remark}
    \label{remark_size_of_V}
In our later application we will take $m=n-r$ and coordinates $(m_{r+1},\ldots,m_n) \in \Z^{n-r}$. Set  $N=\max\{N_{r+1},...,N_n\}.$ The bound for $|E|$ is increasing as a function $N_j$ so we can replace each $N_j$ by $N$. Recalling the hypothesis $VN \leq 1,$  the bound for $|E|$ can be crudely estimated by
   \[
        |E| \ll_{k,n,r}   \left\lfloor \frac{N}{q} \right\rfloor^{n-r}
           q^{(n-r)/2}(VN+   \left\lfloor \frac{N}{q} \right\rfloor^{-1}
              (\log q)^{n-r}) .
  \]
We need the error to be at most a small proportion (say half the size) of the main term. To achieve this, we will choose $(a,b)\in \F_q \times \F_q^{n-r}$ so that the exponential sum in the main term is large (using Proposition \ref{prop:bound_T}) and  we will impose conditions that force  $V \leq d_0 N^{-1}$ for a constant $d_0<1$ as small as we like, and $N/q \gg q^{\Del_0}$ for some $\Del_0>0$ (see \S \ref{sec:arithmetic}).
\end{remark}

To prove the proposition, we apply partial summation with respect to $\mathbf{m}$, following Lemma \ref{lem:partial_summation}, with the function $h(\mathbf{m}) = (y-\tfrac{2\pi}{q}b)\cdot \mathbf{m}.$  Since this is linear,
for a multi-index $\boldsymbol{\alpha} \in \{0,1\}^n$, the $\partial_{\boldsymbol{\al}}$ partial derivative as a function of $\mathbf{m}$ vanishes unless $|\boldsymbol{\al}|=1$, say $\partial_{\boldsymbol{\al}}= \partial_i,$ in which case
       \[
         |\partial_i((y-\tfrac{2\pi}{q}b)\cdot \mathbf{m})| =   |y_i-2\pi b_i/q| < V.
      \]
Thus   using the notation of the lemma, for any nonempty subset $J \subseteq\{1,\ldots,m\},$ we only need to consider the following expression in the case each $|\boldsymbol{\alpha}_i|=1$ (so that $\ell=|J|$):
   \[
             \sup_{1\leq \ell \leq |J|} \sup_{\substack{\boldsymbol{\alpha_1},...,\boldsymbol{\alpha_\ell} \in \{0,1\}^n\\\boldsymbol{\alpha_1}+ \cdots +\boldsymbol{\alpha_\ell} = \boldsymbol{1}_J}}  \prod_{i=1}^{\ell} B_{\boldsymbol{\alpha_i}} = \prod_{i=1}^{|J|} V = V^{|J|}.
        \]
    Now apply partial summation to see that
        \begin{multline}\label{eqn:2.7_ps}
            \sum_{\bfM\leq \bfm \leq \bfM+\bfN} e\left(\frac{2\pi a}{q} Q_k(\bfm)+ y\cdot \bfm\right) \\
            = e\left( (y-\frac{2\pi}{q}b)\cdot (\bfM+\bfN) \right) \sum_{\bfM \leq \bfm \leq \bfM+\bfN} e\left(\frac{2\pi}{q}(a Q_k(\bfm)+b \cdot \bfm)\right) + E_1
        \end{multline}
    where $E_1$ is dominated by
    \[ \sup_{|J| \geq 1} V^{|J|}\prod_{j \in J}N_j \cdot \sup_{\mathbf{u}_{(J)} \leq \mathbf{N}_{(J)}} \left| \sum_{\substack{\mathbf{M}_{(I)} \leq \bfm_{(I)}\leq \mathbf{M}_{(I)}+\mathbf{N}_{(I)},\\\mathbf{M}_{(J)} \leq \bfm_{(J)}\leq \mathbf{M}_{(J)}+ \mathbf{u}_{(J)}}} e\left(\frac{2\pi }{q}(a Q_k(\bfm_{(J)},\bfm_{(I)}) + b\cdot (\bfm_{(J)},\bfm_{(I)}))\right)\right|  
 \]
    with nonempty $J\subseteq \{1,...,m\}$ and $I= \{1,...,m\} \setminus J$.
The first factor is $\leq (VN)^{|J|} \leq VN \leq 1.$    We may estimate the sum in the error $E_1$ by first breaking it into a main term of as many complete sums (complete in all $m$ coordinates) as possible, that is,
   \[ \prod_{i\in I }\left\lfloor \frac{N_i}{q} \right\rfloor\cdot 
            \color{black}\prod_{j\in J}
            \left\lfloor \frac{u_j}{q} \right\rfloor\cdot
            \sum_{\bfm \modd q^{m}} e\left(\frac{2\pi }{q}(a Q_k(\bfm_{(J)},\bfm_{(I)}) +  b \cdot(\bfm_{(J)},\bfm_{(I)}))\right),\]  
plus an error term of incomplete sums, that is, of the form
    \[    \sum_{\substack{\tilde{I}\subseteq I, \tilde{J}\subseteq J,\\
            |\tilde{I}\cup \tilde{J}|<m }}  
             \prod_{i\in \tilde{I}}\left\lfloor \frac{N_i}{q} \right\rfloor \cdot 
             \prod_{j\in \tilde{J}}
            \left\lfloor \frac{u_j}{q} \right\rfloor\cdot
            \sideset{}{^*}\sum  e\left(\frac{2\pi }{q}(a Q_k(\bfm_{(J)},\bfm_{(I)}) + b\cdot (\bfm_{(J)},\bfm_{(I)}))\right)
  \]
 in which the starred sum is over $\mathbf{1} \leq \bfm_{(\tilde{I})}\leq \mathbf{q}_{(\tilde{I})}$, $\mathbf{1} \leq \bfm_{(I\setminus \tilde{I})}\leq \mathbf{H}_{(I\setminus \tilde{I})}$,
            $\mathbf{1} \leq \bfm_{(\tilde{J})}\leq \mathbf{q}_{(\tilde{J})}$, $\mathbf{1} \leq \bfm_{(J \setminus \tilde{J})}\leq \mathbf{H}_{(J\setminus \tilde{J})}
            $,  for some $\mathbf{H}=(H_1,...,H_m)$ with $1\leq H_i < q$.
Since $Q_k$ is a Deligne polynomial of degree $k\geq 2$, we may apply the Weil-Deligne bound in Lemma \ref{lem:deligne} so that the contribution of the complete sums is at most  \[  \ll_{k,m}    \prod_{i\in I }\left\lfloor \frac{N_i}{q} \right\rfloor \cdot
             \prod_{j\in J}
            \left\lfloor \frac{u_j}{q} \right\rfloor \cdot q^{m/2}
            \ll \prod_{i=1}^m\left\lfloor \frac{N_i}{q} \right\rfloor\cdot
            q^{m/2} .\]
Again using that $Q_k$ is a Deligne polynomial, we may apply Proposition \ref{prop:bound_incomplete} so the contribution of the incomplete sums is 
\[\ll \sup_{\substack{\tilde{I}\subseteq I,\tilde{J}\subseteq J,\\ | \tilde{I}\cup \tilde{J}|< m}}  \prod_{i\in \tilde{I}}\left\lfloor \frac{N_i}{q} \right\rfloor\cdot  \prod_{j\in \tilde{J}}
            \left\lfloor \frac{u_j}{q} \right\rfloor \cdot q^{m/2} (\log q)^{m-(|\tilde{J}|+|\tilde{I}|)}
            \ll \sup_{|J|< m}  \prod_{j \in J} \left\lfloor \frac{N_j}{q} \right\rfloor\cdot  q^{m/2} (\log q)^{m-|J|}.
     \]
     This is sufficient for bounding $E_1.$
Finally we can similarly separate the main term of the right-hand side of \eqref{eqn:2.7_ps} into complete and incomplete sums, that is,
\[
             \prod_{i=1}^m\left\lfloor \frac{N_i}{q} \right\rfloor \cdot \left| \sum_{\bfm \modd q^m} e\left(\frac{2\pi}{q} (a Q_k(\bfm)+ b\cdot \bfm)\right) \right| \\
             + E_1',\]
            where 
            \[ E_1' = \sum_{|J|<m}  \prod_{j\in J}\left\lfloor \frac{N_j}{q}\right\rfloor \cdot \sum_{\substack{\mathbf{1}_{(J)} \leq \bfm_{(J)}\leq \mathbf{q}_{(J)},\\\mathbf{1}_{(I)} \leq \bfm_{(I)}\leq \mathbf{H}_{(I)}}} e\left(\frac{2\pi}{q}(a Q_k(\bfm_{(J)},\bfm_{(I)})+b\cdot (\bfm_{(J)},\bfm_{(I)}))\right)
    \]
for some $H_i <q$. Another application of Proposition \ref{prop:bound_incomplete} shows that 
\[ |E_1'|
            \ll \sup_{|J|< m} \prod_{j\in J}\left\lfloor \frac{N_j}{q}\right\rfloor \cdot q^{m/2}(\log q)^{m-|J|}.
     \]
Combined with the bound for $E_1$, this proves the proposition.

\section{Reducing the maximal operator to an exponential sum}\label{sec:reducing}

\subsection{Initial definition of the test function}

We now construct a collection of test functions  to prove Theorem \ref{thm:main3}. The symbol $P\in \R[X_1,...,X_n]$ has leading form $P_k \in \Q[X_1,\ldots,X_n]$ that   is Dwork-regular in $X_1,\ldots,X_n$ over $\Q$  and has intertwining rank $r$. By relabelling variables we may assume that $X_1$ intertwines with  $X_1,X_2,...,X_{r}$ and does not intertwine with $X_{r+1},\ldots,X_n$. By a $\mathrm{GL}_n(\Q)$ change of variables we may clear denominators and assume from now on that $P_k$ has integral coefficients.
Fix  $R\geq 1$, which will later be a parameter that tends to infinity. Let
\[
        S_1 = R^\sigma, \qquad L = R^\lambda
 \]
where $0 < \sigma< 1$, $0< \lambda < 1$ are parameters we will choose optimally later.   Once and for all, fix a Schwartz function $\phi$ on $\R$ with the properties (i) $\phi \geq 0$, (ii) $\phi(0) = (2\pi)^{-1}\int \hat{\phi}(\xi)d\xi=1,$ (iii) $\mathrm{supp}(\hat{\phi}) \subseteq [-1,1]$; such a function may be constructed in a standard way, such as in \cite[\S 2.1]{Pie20}. Then define for each $m \geq 1$ and variables $u_1,\ldots,u_m$ that $\Phi_m(u_1,\ldots,u_m) = \phi(u_1) \cdots \phi(u_m)$.

We will define the test function $f$, tailored to the fact that $X_1$ does not intertwine with $X_{r+1},\ldots,X_n$.
Fix an integral tuple $M=(M_1,M_2,...,M_r)\in\Z^r$ with $M_i\geq 1$ as provided by Lemma \ref{lem:coefficients}. 
Define the test function $f$ to be  
    \begin{multline*} 
        f(x) := \Phi_r(S_1x_1,x_2,\ldots,x_r)e((M\circ R)\cdot (x_1,\ldots,x_r))
        \\\times \Phi_{n-r}(x_{r+1},\ldots,x_n)\sum_{\substack{m\in \Z^{n-r},\\R/L\leq m_j<2R/L}} e(Lm\cdot (x_{r+1},\ldots,x_n))
    \end{multline*}
where we recall the notation $M\circ R = (M_1R,...,M_rR)$.
Let $M_* = \max\{2,M_1,M_2,...,M_{r}\}$. The Fourier transform of $f$ is supported in  
  \[
       [M_1R-S_1,M_1R+S_1]\times [M_2R-1,M_2R+1]\times \cdots \times [M_{r}R-1,M_{r}R+1]\times [R-1,2R+1]^{n-r},\]
       which is contained in $B_n(0,\sqrt{n} M_*R + \sqrt{n} S_1) \setminus B_n(0,\sqrt{n}R - \sqrt{n}S_1).$
Since $S_1 = R^\sig$ with $\sig<1$ there exists some $R_1(\sig)$ such that for $R \geq R_1(\sig)$,  $S_1 \leq (1/2)R$ so that this region is contained in the annulus $\{ (1/C)R \leq |x|\leq C R\}$ with $C=2M_*\sqrt{n}$, say, so that $C$ depends only on the symbol $P$ and the dimension $n$.

Finally, we note  for later reference the size of the norm of $f$. Since $\hat{f}$ is supported in the annulus stated above, for all $R \geq 1,$ $R^s\|f\|_{L^2(\R^n)} \ll_s \|f\|_{H^s(\R^n)} \ll_s R^s \|f\|_{L^2(\R^n)}.$ Moreover,
\beq\label{f_norm}
S_1^{-1/2} \lfloor R/L \rfloor^{\frac{n-r}{2}} \|\phi\|_{L^2(\R)}^n \leq \|f\|_{L^2(\R^n)} \leq S_1^{-1/2} \lceil R/L \rceil^{\frac{n-r}{2}} \|\phi\|_{L^2(\R)}^n.
\eeq
Explicitly, using the notation $\xi = (\xi_1,\ldots,\xi_r)$ and $\eta = (\xi_{r+1},\ldots,\xi_n),$
\[ \hat{f}(\xi, \eta)  = \sum_{\substack{m \in \Z^{n-r}\\R/L \leq m_j < 2R/L}} g_{m}(\xi,\eta),\]
in which $g_{m}(\xi,\eta) = S_1^{-1} \hat{\Phi}_{r}(S^{-1}\circ(\xi - M\circ R)) \hat{\Phi}_{n-r}(\eta - Lm)$, where $S^{-1}=(S_1^{-1},1,...,1)$. Thus $g_m$ is supported in the set $\Bcal + (M\circ R,Lm)$, where $\Bcal$ is the box $[-S_1,S_1] \times [-1,1]^{n-1}.$ In particular, for $n \geq 2, n-r\geq 1,$ and $L \geq 4$, the supports of $g_{m}$ as $m$ varies are distinct. Thus by Plancherel's theorem,
\[ (2\pi)^n\|f\|_{L^2(\R^n)}^2=\|\hat{f}\|^2_{L^2(\R^n)} = \sum_{\substack{m \in \Z^{n-r}\\R/L \leq m_j < 2R/L}} \|g_{m}\|_{L^2(\R^n)}^2,\]
and the claim (\ref{f_norm}) holds since by Plancherel's theorem,
\[\|g_{m}\|_{L^2(\R^n)}^2 = S_1^{-1}\|\hat{\phi}\|_{L^2(\R^n)}^{2n} = S_1^{-1}(2\pi)^n\|\phi\|_{L^2(\R^n)}^{2n}.\]
\begin{remark}\label{remark_n_r_1}
Note that if $n=r$ or $n=1$ the above sum is empty; computing the $\|f\|_{H^s}$ norm as above and taking $\sig=1/2$   produces a counterexample to (\ref{thm_T_bound}) for all $s < 1/4,$ which is the claim of Theorem \ref{thm:main2} in this case. Thus from now on we may assume $r<n$ and $n \geq 2$.
\end{remark}

\subsection{Approximation of the maximal operator by an exponential sum} 
Since $f$ treats the intertwined variables $x_1,\ldots,x_r$ differently, it is convenient to define $v=(x_1,\ldots,x_r)$ and $w = (x_{r+1}, \ldots,x_n)$, and similarly use $\xi=(\xi_1,...,\xi_{r})$, $\eta=(\xi_{r+1},...,\xi_n)$; finally, we will continue to denote $S \circ \xi = (S_1\xi_1,\xi_2,\ldots,\xi_r)$ and $M\circ R=(M_1R,...,M_{r}R)$ for $(M_1,\ldots,M_r)$ provided by Lemma \ref{lem:coefficients}. By definition, for the test function $f$, 
\begin{multline}\label{eqn:T_int} 
  T_t^{P}f(x) =   \frac{1}{(2\pi)^{n}} \int_{\R^{n}} \hat{\Phi}_{n}(\xi,\eta) e((S\circ \xi + M\circ R)\cdot v + \eta\cdot w ) \\
  \times \sum_{\substack{m \in \Z^{n-r} \\R/L \leq m_j < 2R/L}} e(Lm\cdot w) e(P(S\circ \xi+M\circ R,\eta+Lm)t)d\xi d\eta.
\end{multline}
The main result of this section is that $T_t^{P}f(x)$ is well-approximated by an exponential sum defined as follows: for any $u \in \R^{n-r}$ with each $u_j> R/L$, and any $w \in \R^{n-r},t \in \R$ set
\[
    \Sbf(u;w,t):=\sum_{\substack{m \in \Z^{n-r} \\R/L \leq m_j < u_j}} e\left(Lm\cdot w + P_k(M\circ R,Lm) t\right).
\]
For simplicity, when $u = (2R/L,\ldots, 2R/L)$, we denote this sum by $\mathbf{S}(2R/L;w,t)$.

Fix $0<c_0<1/2$. Then since $\phi$ is smooth and $\phi(0)=1$, there exists a small constant $\delta_0=\delta_0(c_0, \phi)$ such that 
    \begin{align}\label{eqn:phi_delta}
        \phi(y)\geq 1-c_0/2 \qquad \text{for all $|y|\leq \delta_0$.}
    \end{align}
 The main result of this section is:

\begin{prop}\label{prop:approx_maximal} Let $0<c_0<1/2$ be a small constant and $\delta_0$ be as in \eqref{eqn:phi_delta}. There exist constants $0< c_1(\delta_0,k,P),c_2(k,P)<1/2$ such that for all  $c_1<c_1(\delta_0,k,P),c_2<c_2(k,P)$ and $0<c_3<1$ as small as we like, the following holds. 

Let $x\in (-c_1,-c_1/2]\times [-c_1,c_1]^{n-1}$, $\sigma \leq 1/2$ and $R\geq R_2(c_1,c_2,k, P,\sigma)$. Let $t\in (0,1)$ satisfy  
    \begin{align}\label{eqn:t_constraints}
        t &= \frac{-x_1}{(\partial_1 P_k)(M\circ R,\tilde{R})} + \tau, \qquad |\tau| \leq \frac{c_2\delta_0}{S_1R^{k-1}},\qquad \text{and} \qquad
        t \leq \frac{c_3}{R^{k-1}},
    \end{align}
    in which   $\tilde{R}:=(L(\lceil 2R/L \rceil -1),...,L(\lceil 2R/L \rceil -1)) \in \R^{n-r}$.
Then with $w=(x_{r+1},\ldots,x_n),$
 \[
        |T_t^{P}f(x)| \geq (1-c_0)^n |\Sbf(2R/L;w,t)| + \Ebf_1
 \]
where  
    \[
        |\Ebf_1| \ll_{\phi,n,k,r,P}c_3  \sup_{u \in [R/L,2R/L]^{n-r}} |\Sbf(u;w,t)|.
   \]
\end{prop}
Since we will bound $|\Sbf(u;w,t)|$ by a function that is increasing as a function of  $u_j$, we will obtain $|\Ebf_1| \ll_{\phi,n,k,r,P} c_3  |\Sbf(2R/L;w,t)|$, so that the error term is at most half the main term, by taking $c_3$ appropriately small, as we may.
    For the conditions in (\ref{eqn:t_constraints}) to be compatible we  need $(\partial_1 P_k)(M\circ R,\tilde{R})$ to be nonzero and moreover to be of size $R^{k-1}$. This is assured by Lemma \ref{lem:coefficients}.

To prove the proposition, we first use partial summation to  remove all terms in the sum over $m$ in (\ref{eqn:T_int}) that do not appear in $\Sbf(2R/L;w,t)$; to accrue only an acceptable error, we require the third constraint in \eqref{eqn:t_constraints}, and intertwining rank plays a key role. The next step is to use integration by parts to remove all terms in the phase in (\ref{eqn:T_int}) that are nonlinear in $\xi,\eta$, so that we may apply Fourier inversion to the integral of $\hat{\Phi}_n$; this applies the third constraint in  \eqref{eqn:t_constraints}. Finally, after Fourier inversion, we arrive at a main term involving the evaluation of $\Phi_n$, and to bound this from below we use the fact that $\Phi_n(0)=1$ forces $t$ to lie in a specified range around a specific point, which is encoded in the first and second constraints in \eqref{eqn:t_constraints}.

\subsection{Removing non-arithmetic terms in the exponential sum}\label{sec:expsum_partial_sum}
  We will use the notation 
    \begin{align*}
        P=P_k+P_{k-1}+\cdots + P_0
    \end{align*} 
where each $P_i$ with $0 \leq i \leq k-1$ is a homogeneous polynomial of degree $i$ with real coefficients.
Then by Taylor expansion (multinomial expansion),
\[
        P(S\circ \xi +M\circ R,\eta+Lm)
        =P_k(M\circ R,Lm) + Q(S\circ \xi+M\circ R,\eta+Lm)
  \]
in which we define
\[
     Q(u+v)
        := \sum_{j=0}^{k-1}P_j(u) + \sum_{j=0}^k \sum_{1\leq |\alpha|\leq j} \frac{(\partial_\alpha P_j)(u)}{\alpha !} v^{\alpha}, \qquad \text{for any $u,v \in \R^n$,}\]
        so that
\beq\label{eqn:Q}
        Q(S\circ \xi +M\circ R,\eta+Lm)
        = \sum_{j=0}^{k-1}P_j(M\circ R,Lm) + \sum_{j=0}^k \sum_{1\leq |\alpha|\leq j} \frac{(\partial_\alpha P_j)(M\circ R,Lm)}{\alpha !} (S\circ \xi,\eta)^{\alpha}.
\eeq
 Our goal is to extract the non-arithmetic weight $e( Q(S\circ \xi +M\circ R,\eta+Lm)t)$ from the sum over $m$ in (\ref{eqn:T_int}) via partial summation, using the assumption on intertwining rank.
To understand the role of intertwining rank, it is helpful to see a motivating example. 
\subsubsection{Example}\label{sec_example_P}
Suppose the dimension $n=2$ and $P(X_1,X_2) =X^\be$ is simply a monomial, where $\beta$ is a multi-index $(\be_1,\be_2)$ with $\be_1 + \be_2 =k$. Then 
 \beq\label{P_example}  P(S_1\xi_1 +M_1R,\xi_2+Lm_2)
        = \sum_{j=0}^{\be_1} {\binom{\be_1}{j}} (M_1R)^{\be_1-j}(S_1\xi_1)^j \cdot \sum_{\ell=0}^{\be_2} {\binom{\be_2}{\ell}}(Lm_2)^{\be_2-\ell}\xi_2^\ell.\eeq
We isolate the arithmetic term $(M_1R)^{\be_1}(Lm_2)^{\be_2}$ (the $j=\ell=0$ term), leaving the non-arithmetic terms
\[ 
H(m_2):=\sum_{\substack{\al \leq \be\\\al \neq (0,0)}} \frac{\be !}{\al ! (\be - \al)!} (M_1R,Lm_2)^{\be - \al}(S_1\xi_1,\xi_2)^{\al}.
\]
In this example, we model the sum in (\ref{eqn:T_int}) by summing over $m_2:$
\[
\sum_{R/L \leq m_2 < 2R/L} e(Lm_2x_2+ (M_1R)^{\be_1}(Lm_2)^{\be_2}t)e(h(m_2))
\]
with $h(m_2) = H(m_2)t.$
By partial summation with respect to $m_2$, the weight $e(h(m_2))$ can be removed with an error term depending on the derivative $(\partial/\partial m_2)h(m_2)$. This derivative is
\[ 
tL\sum_{\substack{\al \leq \be - (0,1)\\\al \neq (0,0)}} \frac{\be !}{\al ! (\be - \al -  (0,1))!} (M_1R,Lm_2)^{\be - \al-(0,1)}(S_1\xi_1,\xi_2)^{\al}
\ll tL\sum_{\substack{\al \leq \be - (0,1)\\\al \neq (0,0)}} R^{k - |\al|-1}S_1^{\al_1}
,
\]
where the upper bound holds uniformly for $(\xi_1,\xi_2) \in [-1,1]^2,$ $m_2 \in [R/L,2R/L).$ 
In order for the resulting error term (after an application  of  Lemma \ref{lem:partial_summation}) to be acceptable, we need
\[
t  L\sum_{\substack{\al \leq \be - (0,1)\\\al \neq (0,0)}} R^{k - |\al|-1}S_1^{\al_1} \ll L/R .
\]
Since we will ultimately have $t\approx R^{-(k-1)}$, this requires $R^{k - |\al|}S_1^{\al_1}\ll R^{k-1}$ for each $\al$.
If the original monomial $P(X_1,X_2)=X_1^{\be_1}X_2^{\be_2}$ has  $\be_1 \geq 1$ and $\be_2\geq 1$, there will be a term with $|\al| = \al_1=1$, and this term will force the condition $S_1\ll 1$. If however $X_2$ does not intertwine with $X_1$ in $P(X_1,X_2)$, $\be_1=0$ so that $\al_1=0$, removing this difficulty.  This concludes our example.

 In general, we apply Lemma \ref{lem:partial_summation} to
    \begin{align}\label{eqn:sum_remove_dep}
        \sum_{\substack{m \in \Z^{n-r} \\R/L \leq m_j < 2R/L}} e(Lm\cdot w) e(P_k(M\circ R,Lm)t) e(h(m))
    \end{align}
    with 
 $ h(m):=Q(S\circ \xi+M\circ R,\eta+Lm) t.$
We claim that  for each  multi-index $\kappa=(\kappa_{r+1},...,\kappa_{n})\in \{0,1\}^{n-r}$ with $|\kappa|\geq 1$,
    \begin{align}\label{eqn:mixed_partial_P_tilde}
 \sup_{\substack{(\xi,\eta) \in [-1,1]^n \\y \in [R/L,2R/L)^{n-r}}} \left| \frac{\partial^{|\kappa|}}{\partial y_{r+1}^{\kappa_{r+1}} \cdots \partial y_{n}^{\kappa_{n}}}
      Q((S\circ \xi,\eta)+(M\circ R,Ly)) \right| \ll_{n,k,P}   R^{k-1}(L/R)^{|\kappa|} .
    \end{align}
    Note that the derivatives controlled by $\kappa$ apply only to the last $n-r$ coordinates.
It is in proving this bound that we crucially use the definition of intertwining rank, and the construction of the test function $f$ to accommodate the fact that $X_1$ does not intertwine with $X_{r+1},\ldots,X_n$ in $P_k$; in particular, we will use the fact that for each $\ell>r$, the mixed partial $\partial_\ell \partial_1 P_k$ vanishes identically.

Assume (\ref{eqn:mixed_partial_P_tilde}) for the moment; then for all such $\kappa$,
\[   \sup_{\substack{(\xi,\eta) \in [-1,1]^n \\y \in [R/L,2R/L)^{n-r}}} \left| \frac{\partial^{|\kappa|}}{\partial y_{r+1}^{\kappa_{r+1}} \cdots \partial y_{n}^{\kappa_{n}}} h(y) \right|\ll   R^{k-1}(L/R)^{|\kappa|}t =:B_{\kappa}, 
\]
in the notation of Lemma \ref{lem:partial_summation}.
(Here we also note that $B_\kappa \con 0$ if any entry in a multi-index $\kappa$ exceeds the highest corresponding degree in the polynomial $Q$. However, unlike our previous application of partial summation in Proposition \ref{prop:2.7}, in this case the total degree of $Q$ is $k-1$ in $m$, so  we could have $\ell=1$ in the application below.) Consequently, for each nonempty $J \subseteq \{r+1,\ldots,n\},$ and $\ell \leq |J|,$
\[
        \sup_{\substack{\boldsymbol{\alpha}_1,...,\boldsymbol{\alpha}_\ell\in \{0,1\}^{n-r}\\\boldsymbol{\alpha}_1+ \cdots +\boldsymbol{\alpha}_\ell = \boldsymbol{1}_J}} \prod_{i=1}^\ell B_{\boldsymbol{\alpha_i}} 
        = \sup_{\substack{\boldsymbol{\alpha}_1,...,\boldsymbol{\alpha}_\ell\in \{0,1\}^{n-r}\\\boldsymbol{\alpha}_1+ \cdots +\boldsymbol{\alpha}_\ell = \boldsymbol{1}_J}} \prod_{i=1}^\ell (R^{k-1}(L/R)^{|\boldsymbol{\alpha}_i|}t )
        = R^{(k-1)\ell}(L/R)^{|J|}t^{\ell} .
\]
This is $\ll R^{k-1}t(L/R)^{|J|}$
under the assumption $t \ll R^{-(k-1)},$ which we assume from now on.

Hence by Lemma \ref{lem:partial_summation} the sum in \eqref{eqn:sum_remove_dep} is identical to 
 \[
        \Sbf(2R/L;w,t)e(Q(S\circ \xi + M\circ R,\eta+\tilde{R})t) + E_1
    \]
where we define $\tilde{R}:=(L(\lceil 2R/L \rceil -1),...,L(\lceil 2R/L \rceil -1)) \in \R^{n-r}$ and we may take
 \[
        |E_1|\ll_{n,k,r,P} 
        \sup_{u \in [R/L,2R/L)^{n-r}} |\Sbf(u;w,t)|\cdot
      R^{k-1}t,
   \]
 uniformly in $(\xi,\eta) \in [-1,1]^n.$ Consequently, after integrating $E_1$ trivially in $(\xi,\eta)$ by applying the compact support of $\hat{\Phi}_n$ in $[-1,1]^n$,
we conclude that $ T_t^{P}f(x) $ is precisely
    \beq\label{eqn:T_int_reduced} 
 \frac{1}{(2\pi)^{n}} \Sbf(2R/L;w,t) \int_{\R^{n}} \hat{\Phi}_{n}(\xi,\eta) e((S\circ \xi + M\circ R)\cdot v +\eta\cdot w)\\
      \times
       e(Q(S\circ \xi +M\circ R,\eta+\tilde{R})t)   d\xi d\eta + E_2
   \eeq
where
 \beq\label{E2}
        |E_2| \ll_{\phi,n,k,r,P} \sup_{u \in [R/L,2R/L)^{n-r}} |\Sbf(u;w,t)|\cdot
     R^{k-1}t.
\eeq

It remains to verify (\ref{eqn:mixed_partial_P_tilde}). We recall the expansion (\ref{eqn:Q}), and bound the mixed partial of each term. First note that for $y=(y_{r+1},\ldots,y_n)$ and any $\ell > r$, for each degree $j$,
\[|(\partial/\partial y_\ell) (P_j(M\circ R,Ly))| = L |( \partial_\ell P_j)(M\circ R,Ly)| \ll R^{j-1}L,\]
uniformly for $y \in [R/L,2R/L)^{n-r}.$
Thus for each fixed $j \leq k-1$,    for each $\kappa = (\kappa_{r+1},\ldots,\kappa_n) \in \{0,1\}^{n-r}$ with $|\kappa| \leq j$, 
\[\sup_{y \in [R/L,2R/L)^{n-r}}\left|\frac{\partial^{|\kappa|}}{\partial y_{r+1}^{\kappa_{r+1}} \cdots \partial y_{n}^{\kappa_{n}}} P_j(M\circ R,Ly)\right| \ll_{n,P_j} L^{|\kappa|} R^{j-|\kappa|}  \ll R^{k-1}(L/R)^{|\kappa|}, 
\]
which suffices for (\ref{eqn:mixed_partial_P_tilde}).
For the other terms in the expansion (\ref{eqn:Q}), fix some multi-index $\al=(\al_1,\ldots,\al_n) \in \R^n$ with $1 \leq |\al| \leq k$, and recall the multi-index $\kappa=(\kappa_{r+1},...,\kappa_{n})\in \{0,1\}^{n-r}$ with $|\kappa|\geq 1$ taking derivatives only with respect to coordinates of $y=(y_{r+1},\ldots,y_n)$. For each $0 \leq j \leq k,$ for $|\kappa| \leq j - |\al|,$
    \begin{multline*} 
    \sup_{\substack{(\xi,\eta) \in [-1,1]^n\\y \in [R/L,2R/L)^{n-r}}}\left|  \frac{\partial^{|\kappa|}}{\partial y_{r+1}^{\kappa_{r+1}} \cdots \partial y_{n}^{\kappa_{n}}}(\partial_\alpha P_j)(M\circ R,Ly)\frac{(S\circ \xi,\eta)^{\alpha}}{\al!} \right|\\
    \ll_{n,k,P} L^{|\kappa|}R^{j-|\al |- |\kappa|}   S_1^{\al_1}
    \ll R^j (L/R)^{|\kappa|} S_1^{\al_1}R^{-|\al|},
    \end{multline*}
in which we recall that $S \circ \xi = (S_1\xi_1,\xi_2,\ldots,\xi_r).$ 
If $j\leq k-1$, this already suffices for (\ref{eqn:mixed_partial_P_tilde}), since $S_1 \ll R$ shows that the right-hand side is visibly $\ll R^{k-1}(L/R)^{|\kappa|}.$ (Effectively, for $j \leq k-1$ we are using that $P_j$ is already of degree strictly smaller than $k$.)

Finally for the highest degree $j=k$ piece, we must be more careful, and use the fact that in the leading form $P_k$, $X_1$ does not intertwine with the last $n-r$ coordinates, over which we are carrying out the partial summation. Consider the expression $R^{k-1} (L/R)^{|\kappa|} S_1^{\al_1}R^{-(|\al|-1)},$ recalling $|\al|\geq 1$ in the cases we currently consider. 
As explained in the example in \S \ref{sec_example_P} and the heuristics in \S \ref{sec_heuristic}, the problematic case  is when $|\alpha|=\alpha_1=1$. But since $X_1$ does not intertwine with $X_{r+1},\ldots,X_n$, $\partial_j \partial_1 P_k \con 0$ for each $j \geq r+1$. Consequently, for any $\al$ with $\al_1 \geq 1,$ for every $\kappa = (\kappa_{r+1},...,\kappa_n)\in \{0,1\}^{n-r}$ with $|\kappa|\geq 1$, \[
\frac{\partial^{|\kappa|}}{\partial y_{r+1}^{\kappa_{r+1}} \cdots \partial y_{n}^{\kappa_{n}}}(\partial_\alpha P_k)(M\circ R,Ly)\con 0.\]  Thus for the highest degree form $P_k$, we only need to verify that $R^{k-1} (L/R)^{|\kappa|} S_1^{\al_1}R^{-(|\al|-1)}\ll R^{k-1} (L/R)^{|\kappa|}$ in cases where $|\al| \geq 1$ and $\al_1=0$, and this certainly holds.  This completes the proof of (\ref{eqn:mixed_partial_P_tilde}).

\subsection{Removing nonlinear terms in the phase, to apply Fourier inversion}\label{sec:expsum_partial_int}

We return our focus to (\ref{eqn:T_int_reduced}), and now we remove all nonlinear terms in $(\xi,\eta)$ by integration by parts, in order to apply Fourier inversion. First we Taylor expand, writing 
    \begin{align*}
        Q(S\circ \xi + M\circ R, \eta + \tilde{R})
        &=\sum_{j=0}^{k-1}P_j(M\circ R,\tilde{R}) 
        + \sum_{1\leq |\alpha|\leq k} \frac{(\partial_\alpha P)(M\circ R,\tilde{R})}{\alpha !} (S\circ \xi,\eta)^{\alpha} \\
        &=(\nabla P_k)(M\circ R,\tilde{R})\cdot (S\circ \xi,\eta) + \tilde{Q}(S\circ \xi + M\circ R, \eta + \tilde{R})
    \end{align*}
where we define $ \tilde{Q}(S\circ \xi + M\circ R, \eta + \tilde{R})$ to be
   \[ \sum_{j=0}^{k-1}P_j(M\circ R,\tilde{R}) + \sum_{j=1}^{k-1}(\nabla P_j)(M\circ R,\tilde{R})\cdot(S\circ \xi,\eta) 
        + \sum_{2\leq |\alpha|\leq k} \frac{(\partial_\alpha P)(M\circ R,\tilde{R})}{\alpha !} (S\circ \xi,\eta)^{\alpha}.
 \]
Then 
 the main term of \eqref{eqn:T_int_reduced} can be written as
    \begin{multline*} 
        \Sbf(2R/L;w,t)e((M\circ R)\cdot v) \frac{1}{(2\pi)^{n}} \int_{[-1,1]^{n}} \hat{\Phi}_{n}(\xi,\eta)
        e((S\circ \xi, \eta)\cdot ((v,w)+ (\nabla P_k)(M\circ R,\tilde{R})t) )\\
        \times
        e(\tilde{Q}(S\circ \xi+M\circ R,\eta+\tilde{R}) t)   d\xi d\eta.
    \end{multline*}
We will apply integration by parts as in Lemma \ref{lem:partial_integration} to remove  the nonlinear terms in $(\xi,\eta),$ with
$
        h(\xi,\eta) := \tilde{Q}(S\circ \xi + M\circ R,\eta+\tilde{R}) t.
$
We claim that for every $\kappa = (\kappa_1,...,\kappa_n) \in \{0,1\}^n$ with $|\kappa|\geq 1$,
    \begin{align}\label{eqn:mixed_partial_P_tilde2}
     \sup_{(\xi,\eta) \in [-1,1]^n}\left|\frac{\partial^{|\kappa|}}{\partial \xi_1^{\kappa_1} \cdots \partial \xi_r^{\kappa_r}\partial \eta_{r+1}^{\kappa_{r+1}}\cdots \partial \eta_n^{\kappa_n}} h(\xi,\eta)\right| \ll R^{k-1}t =: B_\kappa,
    \end{align}
    in the notation of Lemma \ref{lem:partial_integration}.
Then by integration by parts, the integral above becomes
\[
        e(\tilde{Q}(S\circ 1+M\circ R,1+\tilde{R})t)
        \frac{1}{(2\pi)^n}\int_{\R^{n}} \hat{\Phi}_{n}(\xi,\eta)
        e((S\circ \xi, \eta)\cdot ((v,w)+ t(\nabla P_k)(M\circ R,\tilde{R})) )d\xi d\eta + E_3
   \]
    where
$
            |E_3|\ll_{\phi, n} R^{k-1} t .
$
 By Fourier inversion, the integral above evaluates precisely to 
  \[
        \Phi_n((S_1,1,\ldots,1) \circ [(x_1,x_2,...,x_n)+t(\nabla P_k)(M\circ R,\tilde{R})]).
 \] 
 Consequently, in absolute value, \eqref{eqn:T_int_reduced} is
    \beq\label{eqn:reduced_exp}
        |\Sbf(2R/L;w,t)| \cdot | \Phi_n((S_1,1,\ldots,1) \circ [(x_1,x_2,...,x_n)+t(\nabla P_k)(M\circ R,\tilde{R})])| + E_2 + E_4
    \eeq
with $E_2$   bounded as in (\ref{E2}) and
 \beq\label{E4}
        |E_4| \ll |\Sbf(2R/L;w,t)|\cdot|E_3| \ll_{\phi, n} | \Sbf(2R/L;w,t) |  R^{k-1} t.
   \eeq

 Finally, we verify (\ref{eqn:mixed_partial_P_tilde2}).    
We bound each term in the expansion defining $\tilde{Q}$. The first sum is independent of $(\xi,\eta)$ so it vanishes when we take partials $\partial_\kappa$ with $|\kappa| \geq 1$. For each $1\leq j \leq k-1$, recall that $(S \circ \xi)=(S_1\xi_1,\xi_2,\ldots,\xi_r)$ and  $\tilde{R} \ll R$, so that at most
 \[
       \sup_{(\xi,\eta) \in [-1,1]^n}\left|\frac{\partial^{|\kappa|}}{\partial \xi_1^{\kappa_1} \cdots \partial \xi_r^{\kappa_r}\partial \eta_{r+1}^{\kappa_{r+1}}\cdots \partial \eta_n^{\kappa_n}} (\nabla P_j)(M\circ R, \tilde{R})\cdot(S\circ \xi,\eta) \right|
       \ll  R^{j-1}S_1
       \ll R^{k-2}S_1.
\]
This is     $\ll R^{k-1}$ as desired, since $S_1=R^\sig < R$.
Now fix a multi-index $2\leq |\alpha|\leq k$. For each $0\leq j \leq k$,
\[
         \sup_{(\xi,\eta) \in [-1,1]^n}\left|\frac{\partial^{|\kappa|}}{\partial \xi_1^{\kappa_1} \cdots \partial \xi_r^{\kappa_r}\partial \eta_{r+1}^{\kappa_{r+1}}\cdots \partial \eta_n^{\kappa_n}} \frac{(\partial_\alpha P_j)(M\circ R, \tilde{R})}{\alpha !} (S\circ \xi,\eta)^{\alpha} \right|
        \ll R^{j-|\alpha|}S_1^{\alpha_1}.
   \]
For all the lower-degree terms with $j \leq k-1,$ this is $ \ll R^{k-1}S_1^{\alpha_1}R^{-|\alpha|}\ll R^{k-1}.$
Finally, consider $R^{k-|\alpha|}S_1^{\alpha_1}$ in the case $2\leq |\alpha| \leq k$ and $j=k$. (Note that unlike in the partial summation in the previous section, we do not need to consider the case $|\al|=1$ and $j=k$, which is singled out as the main term here.) 
If $\alpha_1=0$, this is $\ll R^{k-1}$. If $|\alpha|>\alpha_1\geq 1$, this is $\ll R^{k-1}(S_1/R)$. If $|\alpha| = \alpha_1\geq 2$, then this is $\ll R^{k}(S_1/R)^2=R^{k-1}(S_1^2/R)$. For this to be $\ll R^{k-1}$ we impose $S_1=R^\sig$ with
 \beq\label{eqn:cond_sigma}
        \sigma \leq 1/2.
    \eeq
 This proves (\ref{eqn:mixed_partial_P_tilde2}) and hence (\ref{eqn:reduced_exp}).

\subsection{Restrictions on $t$ to complete the proof of Proposition \ref{prop:approx_maximal}}\label{sec:expsum_proof} 
Recall from  \eqref{eqn:phi_delta} that $\phi(y) \geq 1-c_0/2$ if $|y| \leq \del_0.$ 
Thus
$\Phi_n((S_1,1,\ldots,1) \circ [(x_1,x_2,...,x_n)+t(\nabla P_k)(M\circ R,\tilde{R})]) \geq (1-c_0/2)^n$, which suffices for Proposition \ref{prop:approx_maximal}, if we choose 
   \begin{align}\label{eqn:cond_t_1}
        t=\frac{-x_1}{(\partial_1 P_k)(M\circ R,\tilde{R})} + \tau, \qquad |\tau| \leq \frac{c_2\delta_0}{S_1 R^{k-1}}
    \end{align}
where $c_2<1$ is sufficiently small that $c_2|(\partial_jP_k)(M\circ R, \tilde{R})|/R^{k-1}<1/2$ for each $1 \leq j \leq n$, and furthermore restrict 
$x\in [-c_1,c_1]^n$ where $0<c_1<\delta_0/4$ and $c_1|(\partial_jP_k)(M\circ R, \tilde{R})/(\partial_1P_k)(M\circ R, \tilde{R})|<\del_0/4$ for each $2 \leq j \leq n$. (Recall from Lemma \ref{lem:coefficients} that $M$ is chosen so that $|(\partial_1 P_k)(M\circ R,\tilde{R})| \gg R^{k-1} \neq 0$ for all $R$, and in the other direction it is always true that $|(\partial_j P_k)(M\circ R,\tilde{R})| \ll R^{k-1}$ for each $1 \leq j \leq n.$) For then, simultaneously 
\[
        |S_1(x_1+(\partial_1 P_k)(M\circ R,\tilde{R}) t)| = |S_1\cdot (\partial_1 P_k)(M\circ R,\tilde{R})\tau| <\delta_0/2,
    \]
and for $j=2,...,n$,
 \[
        |x_j+(\partial_j P_k)(M\circ R,\tilde{R}) t|< \del_0/4 +\del_0/4 +  \delta_0/2.
 \]
    In fact if $x_1\in [-c_1,-c_1/2]$ then we can ensure $t\in (0,1)$ since for appropriate $c_1,c_2,$ for all sufficiently large $R\geq R_2=R_2(c_1,c_2,\sigma, P,k)$, we have $c_1/(\partial_1 P_k)(M\circ R,\tilde{R}) + c_2/(S_1R^{k-1})<1$ and $c_1/(2(\partial_1 P_k)(M\circ R,\tilde{R})) - c_2/(S_1R^{k-1})>0$.

Finally, to bound $E_2$ and $E_4$, we impose the condition that
    \begin{align}\label{eqn:cond_t}
    t\leq \frac{c_3}{R^{k-1}}
    \end{align}
for some small constant $c_3$ of our choice, and then $\mathbf{E}_1 \ll E_2 + E_4$ is bounded as stated in Proposition \ref{prop:approx_maximal}.
For conditions \eqref{eqn:cond_t} and \eqref{eqn:cond_t_1} to be compatible, we verify that we can choose $c_1$ and $c_2$ so that $|R^{k-1}t|\leq c_1(R^{k-1}/(\partial_1 P_k)(M\circ R,\tilde{R}))+c_2(\delta_0/S_1)$ is as small as we want.

\section{The arithmetic contribution}\label{sec:arithmetic}
We now estimate the exponential sum $\Sbf(2R/L;w,t)$ which we extracted from $T_t^{P}f(x)$ in Proposition \ref{prop:approx_maximal}. It is convenient to define the notation:
    \begin{align}\label{eqn:change_var}
        s:=L^k\tau, \qquad 
        y_1:= -\frac{ L^k}{(\partial_1 P_k)(M\circ R,\tilde{R})}x_1 \modd{2\pi}, \qquad 
        y_j:= Lx_j \modd{2\pi}, 
    \end{align}
for $r+1\leq j \leq n$, and set $y=(y_{r+1},\ldots,y_n)$. (Implicitly, we also take $y_j=x_j \modd{2\pi}$ for $2 \leq j \leq r,$ although these variables do not play a role here.)
Then $y_1+s=L^k t \modd{2\pi}$ and so by recalling $w=(x_{r+1},\ldots,x_n)$ and using homogeneity of $P_k$,
    \begin{align*}
        \Sbf(2R/L;w,t) 
      &  = \sum_{\substack{m \in \Z^{n-r} \\R/L \leq m_j < 2R/L}} e(Lm\cdot w + L^k P_k(M\circ (R/L),m) t)\\
       & = \sum_{\substack{m \in \Z^{n-r} \\R/L \leq m_j < 2R/L}} e(m\cdot y +  P_k(M\circ (R/L),m)(y_1+s)).
    \end{align*}
  
\begin{remark}\label{remark_integer}  Here we encounter a feature that has not arisen in previous special cases of $P_k$ considered in \cite{ACP23,EPV22a}, for which the first coordinate only appears in $P_k$ in a diagonal term $X_1^k$, so that it can be pulled out of the exponential sum entirely.   In our general setting, to evaluate the sum using number-theoretic methods, we require each $M_iR/L$ to be an integer. We will consider a sequence of $R=R_j \maps \infty$ where each  $R_j=2^j$ for a sequence of integers $j$.  Since $L=R^\lambda$ for some $0< \lambda <1$, we have $R/L=2^{j(1-\lambda)}$ which is an integer if and only if $j\lambda$ is an integer. We will achieve this by choosing $\lambda=\lambda_1/\lambda_2$ to be rational, and then restricting to $j \maps \infty$ along the arithmetic progression $j \equiv 0 \modd{\lambda_2}$. 
 \end{remark}

For now, we assume $R$ and $L$ are fixed, and $R/L$ is an integer. We define for any prime $q$, and $a\in \F_q,b \in \F_q^{n-r},$ the complete exponential sum
\[ \bfT(a,b;q)
        = \sum_{m \modd q^{n-r}} e\left(\frac{2\pi}{q}\left(b \cdot m+a P_k(M\circ (R/L),m) \right)\right).\]
Let $K_1(P_k)$ be the constant provided by Lemma \ref{lemma_Dwork_reduced}, so that the reduction of $P_k(X_1,\ldots,X_n)$ is Dwork-regular over $\F_q$ in $X_1,\ldots,X_n$ for every prime $q \geq K_1(P_k).$ Then for such $q$, by 
Proposition \ref{prop:dwork_deligne}, $P_k(M\circ (R/L),X_{r+1},\ldots, X_n) $ is a Deligne polynomial in $X_{r+1},\ldots,X_n$ over $\F_q$. Thus we will apply Proposition \ref{prop:bound_T} to show that for many $a,b$ the sum $\bfT(a,b;q)$ must be ``large.'' Consequently, we will construct a set $\Omega$ of $(y_1,\ldots,y_n)\in [0,2\pi]^n$ with (nearly) full measure on which $|\mathbf{S}(2R/L;w,t)|$, and hence $|T_t^{P}f(x)|,$ is large. We state two propositions, first focused on the set $\Omega$, and then focused on the operator $T_t^{P}f(x).$

   \begin{prop}\label{prop:Omega} Let $R,L$ be fixed with $R/L$ an integer. There exists a parameter $K_3(P_k,k)$ such that the following holds for all $Q > K_3$. For every prime $q\in [Q/2,Q]$, let $\mathcal{G}(q)$ denote the set of all $(a,b) \in \F_q \times \F_q^{n-r}$ for which
\[
       (1/2) q^{(n-r)/2} \leq     |\bfT(a,b;q)|\leq (k-1)^{n-r}q^{(n-r)/2}.
\]
For any $0<c_4,c_5<1$   sufficiently small of our choice, define a set $\Omega \subseteq [0,2\pi]^n$ by
     \[
            \Omega = \Union_{\substack{Q/2 \leq q \leq Q\\\text{$q$ prime}}} \hspace{1em}
 \Union_{(a,b)\in \mathcal{G}(q) } B(a,b;q),
      \]
        where each box B(a,b;q) is defined to be
        \[ \{|y_1-2\pi a/q|\leq c_4q^{-1}, (y_2,\ldots,y_r) \in [0,c_1]^{r-1}, |y_j-2\pi b_j/q|\leq c_5q^{-1-1/(n-r)}, r+1 \leq j \leq n\}.\]
    Then $
            |\Omega| \gg_{c_4,c_5,n,k} (\log Q)^{-1}.
$
    \end{prop}

  Next, recall that $L=R^\lam$ for a parameter $0<\lam<1$ we will choose later, and from now on we let $Q=R^\kappa$ for a parameter $0<\kappa<1$ we will choose later. 
\begin{prop}\label{prop:evaluate_S}
Let $R,L$ be fixed with $R/L$ an integer. Suppose
    \begin{align}\label{eqn:cond_LQRS}
        \frac{1}{Q} \ll \frac{L^{k}}{S_1R^{k-1}}, 
        \qquad \frac{1}{Q^{1+1/(n-r)}} \ll (R/L)^{-1}, 
        \qquad R/L \gg Q^{1+\Delta_0}
    \end{align}
where $0<\Delta_0\leq 1/(n-r)$.
Assume $R,Q$ are sufficiently large, say $R \geq R_3(\delta_0,c_1,c_2,k,P,\sigma,\lam,\kappa)$ and $Q \geq K_4(n,k,r,P_k,\Del_0)$. Then there exists a set $\Omega^* \subseteq B_n(0,1)$ with
 $
        |\Omega^*| \gg_{c_1,c_4,c_5,n,k,r} (\log Q)^{-1}$ and such that
for every $x\in \Omega^*$, there exists a $t\in (0,1)$ such that
 \[
        |T_t^{P}f(x)|
        \geq (1-c_0)^n  2^{-(n-r)-1}\bigg(\frac{R}{LQ^{1/2}}\bigg)^{n-r} -| \Ebf_1| - |\Ebf_2|
\]
where
 \[
        |\Ebf_1| + |\Ebf_2| \ll_{\phi,n,k,r,P} (c_3+c_5+Q^{-\Delta_0/2})\bigg(\frac{R}{LQ^{1/2}}\bigg)^{n-r}.
\]
\end{prop}
By taking $R$ sufficiently large and choosing the absolute constants $c_3, c_5$ sufficiently small, we will ultimately make the error terms  less than half the size of the main term. We prove the propositions and then turn in \S \ref{sec:parameters} to the final choice of parameters, optimizing the counterexample and completing the proof of Theorem \ref{thm:main3}.

\subsection{Proof of Proposition \ref{prop:Omega}}
Since $M \circ (R/L) \in \Z^{r}$ is fixed and $P_k$ is Dwork-regular, by Proposition \ref{prop:dwork_deligne},  $P_k(M\circ (R/L),X_{r+1},\ldots, X_n) $ is a Deligne polynomial in $X_{r+1},\ldots,X_n$ over $\F_q$ for all primes $q \geq K_1(P_k).$   By Proposition \ref{prop:bound_T}, for all   primes $q>\max\{k,K_1,K_2\}$,  $\al_2 q^{n-r+1} \leq |\Gcal(q)| \leq q^{n-r+1}$ for some $\al_2>0$ depending only on $n,k,r$. 
We have defined the set $\Omega$ accordingly, so that for $(y_1,\ldots,y_n)\in  B(a,b;q)$ for some $(a,b) \in \mathcal{G}(q),$ a small $s$ can be chosen so that $y_1 + s=2\pi a/q$ precisely, and $(y_{r+1},\ldots,y_n)$ is well-approximated by $(2\pi b_{r+1}/q,\ldots,2\pi b_n/q)$. We will use this Diophantine behavior in Proposition \ref{prop:evaluate_S} to show that for $(y_1,\ldots,y_n) \in B(a,b;q)$, $\mathbf{S}(2R/L;w,t)$ is well-approximated by $\lfloor R/Lq\rfloor^{n-r}$ copies of $\mathbf{T}(a,b;q)$ and is hence $\gg \lfloor R/Lq\rfloor^{n-r} q^{(n-r)/2}$. For now we compute a lower bound on the measure of $\Omega$. This is not immediate, since if $n-r>1,$ many of the boxes $B(a,b;q)$ can overlap as $q$ varies; our construction of the set $\mathcal{G}(q)$ is also completely inexplicit about which $a,b$ are chosen. However, we can apply \cite[Lemma 4.1]{ACP23}, which shows that if a set of boxes is ``well-distributed,'' then the measure of their union is at least a positive proportion of the sum of their measures. 

Precisely, choose $K_0$ sufficiently large that (by the prime number theorem) $\pi(x) \geq (1/4) x/\log x$ for all $x \geq K_0/2$. Let $\mathcal{I}$ denote the index set of $(a,b;q)$ over which the unions are taken in the definition of $\Omega.$ Then by \cite[Lemma 4.1(ii)]{ACP23}, 
\begin{align*}
|   \Omega| &= |\Union_{\substack{Q/2 \leq q \leq Q\\\text{$q$ prime}}} \hspace{1em}
 \Union_{(a,b)\in \mathcal{G}(q) } B(a,b;q)|
    \gg 
        \sum_{\substack{Q/2 \leq q \leq Q\\ \text{$q$ prime}}} \hspace{1em}
 \sum_{(a,b)\in \mathcal{G}(q) } |B(a,b;q)|\\
& \gg  (Q/\log Q) (\al_2 Q^{n-r+1}) Q^{-1} (Q^{-1-1/(n-r)})^{n-r} \gg (\log Q)^{-1}
\end{align*}
(with implied constants depending  on $c_1,c_4,c_5,n,k,r$) as long as 
all the boxes $B(a,b;q)$ have comparable size and  
 \beq\label{welldist}
 \# \{ (a,b;q), (a',b';q') \in \mathcal{I} : B(a,b;q) \intersect B(a',b';q') \neq \emptyset \} \ll |\mathcal{I}|,
 \eeq
 with an acceptable implied constant. It is clearly true that $1 \ll |B(a,b;q)|/|B(a',b';q')|\ll 1$ for all pairs of indices $(a,b;q)$ and $(a',b';q') \in \mathcal{I}$. The bound (\ref{welldist}) is a statement that the boxes are well-distributed since a trivial upper bound for the cardinality would be $|\mathcal{I}|^2$; on the other hand, if all the boxes were pairwise disjoint, the cardinality would be precisely $|\mathcal{I}|.$ 
 
 It   remains to verify (\ref{welldist}). Upon requiring $Q > 2\max\{k,K_0,K_1,K_2\}$ and recalling the construction of the sets $\mathcal{G}(q)$, $ (Q/\log Q) Q^{n-r+1}\ll |\mathcal{I}| \ll  (Q/\log Q) Q^{n-r+1}.  $
 The contribution to (\ref{welldist}) when $(a,b;q)=(a',b';q')$ as tuples is clearly $\ll \mathcal{I},$ so we consider instead the case when the tuples are distinct, and we suppose that $ B(a,b;q) \intersect B(a',b';q') \neq \emptyset,$ so that in particular,
 \begin{align*}
     |a/q-a'/q'| & \leq   \frac{c_4}{2\pi}(1/q + 1/q'),    \\
     |b_j/q - b_j'/q'| & \leq   \frac{c_5}{2\pi}(1/q^{1+1/(n-r)} + 1/(q')^{1+1/(n-r)}), \qquad r+1 \leq j \leq n.
 \end{align*}
 If $q = q'$ then by taking $c_4, c_5<1$, the above relations impose $|a-a'|<1$ and $|b_j - b_j'|<1$ for all $j$, so that $(a,b;q)=(a',b';q')$, which is a case we already considered. So we now assume $q \neq q'$ are distinct primes in $[Q/2,Q]$. Then the above relations show that 
 \begin{align}
     |aq'-a'q| & \leq   \frac{c_4}{\pi}Q,  \nonumber  \\
     |b_jq' - b_j'q| & \leq   \frac{c_5C}{\pi}Q^{1-1/(n-r)}, \qquad r+1 \leq j \leq n \label{relations}
 \end{align}
where $C$ is a constant depending on $n,r$. Recall that given an integer $m$, and distinct primes $q,q'$, there is a unique choice of a pair $a,a'$ with $1 \leq a \leq q$ and $1 \leq a' \leq q'$ with $aq'-a'q=m.$ Indeed, if there were another representation $a_0, a_0'$ then $aq'-a'q=a_0q'-a_0'q$ would imply that $(a-a_0)q'=(a'-a_0')q$, so that $q' | (a'- a_0')$ and $q|(a-a_0)$, implying $a=a_0$ and $a' =a_0'$, as claimed. Thus once a tuple $(m_1,m_{r+1},\ldots,m_n)$ of integers is chosen with $|m_1| \leq (c_4/\pi)Q$  and $|m_j| \leq (c_5C/\pi)Q^{1-1/(n-r)}$ for each $j=r+1,\ldots,n$, there is at most one choice of a pair $(a,b)\in \F_q \times \F_q^{n-r}$ and $(a',b')\in \F_{q'} \times \F_{q'}^{n-r}$ satisfying the $n-r+1$ relations in (\ref{relations}) above. Taking into account all possible  values of such  $(m_1,m_{r+1},\ldots,m_n)$, this shows that given $q\neq q'$, at most \[\ll_{c_4,c_5,n,r} Q \cdot (Q^{1-1/(n-r)})^{n-r} \ll Q^{n-r}\]
 pairs of index tuples $(a,b;q)$ and $(a',b';q')$ can have $ B(a,b;q) \intersect B(a',b';q') \neq \emptyset$. 
 Taking a union over all pairs of primes $q \neq q' \in [Q/2,Q]$ bounds the left-hand side of (\ref{welldist}) by $\ll (Q/\log Q)^2 Q^{n-r} \ll |\mathcal{I}|.$
 This proves (\ref{welldist}) and completes the proof of Proposition \ref{prop:Omega}.

\subsection{Proof of Proposition \ref{prop:evaluate_S}}\label{sec_Omega_star} 
The existence and measure of the set $\Omega^* \subseteq B_n(0,1)$ follows directly from the construction of the set $\Omega$ in Proposition \ref{prop:Omega}. Indeed, $\Omega^*$ is defined to be the set of those $x \in B_n(0,1)$ such that   the change of variables  defining the $y$-coordinates in (\ref{eqn:change_var}) map $x$ to a point $(y_1,y_2,\ldots,y_n) \in \Omega$. To bound the measure of $\Omega^*$ from below, one only needs to compute the measure of the pre-image of $\Omega$ under the change of variables (\ref{eqn:change_var}). Under the assumption that $\lam> (k-1)/k$ (which will hold for our final choice of $\lam$), this simple rescaling argument follows  precisely the argument given in \cite[\S 4.5]{ACP23}, and we omit it. 

For every $x \in \Omega^*$ there is thus a corresponding $y =(y_1,\ldots,y_n) \in \Omega$ such that $y \in B(a,b;q)$ for some tuple $a,b,q$ for which $(1/2) q^{(n-r)/2} \leq |\mathbf{T}(a,b;q)| \ll_k q^{(n-r)/2}.$ By the construction of the box $B(a,b;q)$, $|y_1 - 2\pi a/q| \leq c_4 q^{-1}$. Thus we may \emph{choose} a value of $s$ with $|s| \leq c_4q^{-1}$, such that $y_1 + s=2\pi a/q$ exactly. We make this choice for $s$ (which corresponds precisely via (\ref{eqn:change_var}) to a choice of the time parameter $t$).
Again by  the construction of the box $B(a,b;q)$, $|y_j -2\pi b_j/q| \leq c_5 q^{-1-1/(n-r)}$ for each $r+1 \leq j \leq n$, so an application of  Proposition \ref{prop:2.7} with $N_j=R/L$ for all $r+1\leq j \leq n$ and $V=c_5 q^{-1-1/(n-r)}$,  yields
 \[
         \Sbf(2R/L;w,t) 
            =\left\lfloor \frac{R}{Lq} \right\rfloor^{n-r} \bfT(a,b;q) + \mathbf{E}_2
\]
  with $\mathbf{E}_2=E$ as in the proposition. 
 For every $q \in [Q/2,Q],$ in the notation of Proposition \ref{prop:2.7} we have $V\ll c_5 N^{-1}$ by the second constraint in (\ref{eqn:cond_LQRS}), and $N \gg q^{1+\Del_0}$ by the third constraint in (\ref{eqn:cond_LQRS}). Hence we may apply the simplified upper bound for the error term given in 
Remark \ref{remark_size_of_V}, which in the present setting yields
   \beq\label{E2_work}
         |\mathbf{E}_2| \ll   \left\lfloor \frac{R}{Lq} \right\rfloor^{n-r}
          q^{(n-r)/2} (c_5   + \left\lfloor \frac{R}{Lq} \right\rfloor^{-1}
            (\log q)^{n-r} )  \ll (c_5+Q^{-\Del_0/2}) \left\lfloor \frac{R}{LQ^{1/2}}\right\rfloor^{n-r}.
   \eeq  
 Here we have applied the third condition in \eqref{eqn:cond_LQRS} to see that for all $Q$ sufficiently large,
   \[ \left(\frac{R}{LQ}\right)^{-1}(\log Q)^{n-r} \ll Q^{-\Del_0}  (\log Q)^{n-r} \ll Q^{-\Del_0/2}. \]
  Finally, for all $R$ sufficiently large (with respect to $\lam,\kappa$), $\lfloor \frac{R}{Lq} \rfloor \geq (1/2) R/Lq$, so that  the main term of  $|\Sbf(2R/L;w,t)|$  satisfies
    \begin{align*}
      \left\lfloor \frac{R}{Lq} \right\rfloor^{n-r}  |\bfT(a,b;q)|
        \geq 2^{-(n-r)-1} \left( \frac{R}{Lq}\right)^{n-r} q^{(n-r)/2} 
        \geq 2^{-(n-r)-1}\bigg(\frac{R}{LQ^{1/2}}\bigg)^{n-r}.
    \end{align*}

The last step of proving Proposition \ref{prop:evaluate_S} is to  control the error term $\mathbf{E}_1$ from    Proposition \ref{prop:approx_maximal}. By that proposition,
    $  |\Ebf_1| \ll
        \sup_{R/L\leq u_j \leq 2R/L} c_3|\Sbf(u;w,t)|,
    $
    and thus it will suffice to prove that uniformly in $R/L\leq u_j \leq 2R/L,$ the sum obeys the upper bound $|\Sbf(u;w,t)|\ll (R/LQ^{1/2})^{n-r}.$ For this we can again apply Proposition \ref{prop:2.7} with $N_j=u_j$ and $V=c_5 q^{-1-1/(n-r)}$,  so that
 \[
         |\Sbf(u;w,t) |
            =\prod_{j=r+1}^{n}\left\lfloor \frac{u_j}{q} \right\rfloor \cdot |\bfT(a,b;q)| + E_5
\]
  with $E_5=E$ as in the proposition. We apply the bound $|\bfT(a,b;q)|\ll_k q^{(n-r)/2}$, valid for each pair $(a,b) \in \mathcal{G}(q)$. Upon noting that the expressions for both the main term and $E$ given in Proposition \ref{prop:2.7} are increasing as each range $N_j$ increases, we bound both from above by taking $u_j = 2R/L$ in each case. Hence we may in fact apply the upper bound (\ref{E2_work}) also to $E_5$. In conclusion, 
\[
        |\Ebf_1| 
        \ll c_3\left\lfloor \frac{R}{Lq}\right\rfloor^{n-r} q^{(n-r)/2} + c_3(c_5+Q^{-\Delta_0/2})\left( \frac{R}{LQ^{1/2}}\right)^{n-r} \ll c_3\left( \frac{R}{LQ^{1/2}}\right)^{n-r}.
 \]

\section{Choosing parameters and conclusion of the proof}\label{sec:parameters}
From Proposition \ref{prop:evaluate_S}, by taking $R$ sufficiently large (relative to $\phi,\delta_0,c_1,c_2,n,k,r,P,\sigma,\lam,\kappa$) and choosing $c_3,c_5$ sufficiently small (relative to $\phi,c_0,n,k,r,P$) we may conclude that under the hypotheses of the proposition,
    \begin{align*}
        \sup_{0<t<1}|T_t^{P}f(x)| \geq \frac{1}{2}(1-c_0)^n 2^{-(n-r)-1} \left( \frac{R}{LQ^{1/2}}\right)^{n-r}.
    \end{align*}
It then follows from the measure of $\Omega^*$ and the computation (\ref{f_norm}) of $\|f\|_{L^2}$ that
    \begin{align*}
        \frac{\|\sup_{0<t<1}|T_t^{P}f(x)|\|_{L^1(B_n(0,1))}}{\|f\|_{L^2}} \gg  \left( \frac{R}{LQ^{1/2}}\right)^{n-r}S_1^{1/2}(R/L)^{-(n-r)/2}(\log Q)^{-1}.
    \end{align*}
    Set 
    \[\del(n,k,r)=\frac{n-r}{4((k-1)(n-(r-1))+1)}.\]
To finish the proof of Theorem \ref{thm:main3}, it suffices to show that we can define the parameters $S_1 = R^\sig$, $L=R^\lam$ and $Q=R^\kappa$  such that $R/L$ is an integer, (\ref{eqn:cond_sigma}) and (\ref{eqn:cond_LQRS}) are satisfied, and for every $s< \frac{1}{4} + \del(n,k,r)$,
    \begin{align}\label{eqn:main3_equiv}
        \bigg(\frac{R}{LQ^{1/2}}\bigg)^{n-r} S_1^{1/2}(R/L)^{-(n-r)/2}(\log Q)^{-1}\geq A_s R^{s'}
    \end{align}
for some $s'>s$ (and some nonzero constant $A_s$). Note that verifying \eqref{eqn:main3_equiv} is equivalent to choosing $\sigma, \lambda, \kappa$ such that
    \begin{align}\label{eqn:s_maximize}
        s &< \frac{\sigma}{2} + \frac{n-r}{2} - (\kappa + \lam) \frac{n-r}{2},
    \end{align}
    while (\ref{eqn:cond_sigma}) imposes $\sig \leq 1/2$ and (\ref{eqn:cond_LQRS}) imposes 
\beq\label{kls_relations}
        \kappa + k\lambda  \geq k- 1 + \sigma, \qquad \frac{n-(r-1)}{n-r}\kappa + \lambda \geq 1 ,\qquad \lambda \leq 1- \kappa (1+\Del_0).
   \eeq
By taking a linear combination of the first two relations in the line above (namely $1/(k-1)$ times the first relation plus $n-r$ times the second relation), we obtain
  \beq \label{kl_relation}
        \kappa + \lambda \geq \frac{n-(r-1) + \sigma /(k-1)}{n-(r-1)+1/(k-1)}.
    \eeq
To maximize the right-hand side of \eqref{eqn:s_maximize}, we choose $\kappa, \lambda$ so that equality holds in this relation, and substitute the resulting value for $\kappa+\lam$ into (\ref{eqn:s_maximize}). For all $k \geq 2$ the coefficient of $\sig$ on the right-hand side of (\ref{eqn:s_maximize}) is then positive, so in order to enlarge the region in (\ref{eqn:s_maximize}) as much as possible, we take $\sig=1/2$. 
Now solving for $\kappa$ and $\lambda$ that obey the first two relations in (\ref{kls_relations}) and satisfy equality in (\ref{kl_relation}) reveals
 \[
        \kappa = \frac{n-r}{2((k-1)(n-(r-1))+1)}, \qquad \lambda = 1 - \frac{n-(r-1)}{2((k-1)(n-(r-1))+1)}.
   \]
It is then true that $0<\kappa, \lambda <1$. 
Moreover, $\lam = 1 - \kappa (1+\Del_0)$ with $\Del_0 = 1/(n-r)$, so that the third relation in (\ref{kls_relations}) holds. Additionally, $\lam>(k-1)/k$, as was required in \S \ref{sec_Omega_star}. Finally, $\lambda=\lam_1/\lam_2$ is a rational number and hence we take a sequence of integers $j \maps \infty$ with $j\equiv 0 \pmod{2((k-1)(n-(r-1))+1)}$. Then for each $R=R_j=2^j$ and $L=L_j=R_j^\lambda$ as $j \maps \infty$ in this sequence, we have $R_j/L_j = 2^{j(1-\lambda)}$ is an integer, as required in Remark \ref{remark_integer}. 
Finally, we conclude that with these choices, (\ref{eqn:main3_equiv}) holds for all $s< 1/4 + \del(n,k,r),$
which ends the proof of Theorem \ref{thm:main3}.

\section{Forms and intertwining rank: examples and remarks}\label{sec_details_examples}

\subsection{Examples of Dwork-regular forms of any intertwining rank}
For each $k \geq 3$ and $2 \leq r \leq n,$ we now prove the following forms of degree $k$   are Dwork-regular over $\Q$ in $X_1,\ldots, X_n$ with   intertwining rank $r$, namely \begin{align*}
    P_k(X_1,\ldots,X_n)&=X_1^k+\cdots +   X_n^k + \sum_{2\leq j \leq r}X_1X_j^{k-1} + \sum_{2\leq i < j \leq n} X_iX_j^{k-1},\quad \text{$k \geq 3$  odd;}
   \\
        P_k(X_1,\ldots,X_n)&=X_1^k+\cdots +   X_n^k +  \sum_{2\leq j \leq r} X_1^2X_j^{k-2}+ \sum_{2\leq i < j \leq n} X_i^2X_j^{k-2},
   \quad \text{$k \geq 4$  even.}
\end{align*}
These visibly have intertwining rank $r$.
In the next sections we additionally prove these forms are indecomposable, and we compute the codimension of  all Dwork-regular forms of intertwining rank $r<n$, thus quantifying the set of forms for which Theorem \ref{thm:main2} proves a new result.

First we prove that each example $P_k$ defined above is Dwork-regular. We provide a full proof in the case $k \geq 3$ is odd; this relies on the fact that $k-1$ is then even. In the case that $k \geq 4$ is even, the proof is analogous, and relies on the fact that $k-2$ is even. 
        By Lemma \ref{lem:dwork_equiv}, it suffices to check that for all nonempty $S\subseteq \{1,...,n\}$, $P_S:=P|_{X_i=0,i\not\in S}$ is nonzero if $|S|=1$ and $P_S$ is nonsingular if $|S|\geq 2$. It is clear that $P_S$ is nonzero for $|S|=1$, so henceforth we assume $|S|\geq 2$. Suppose $S=\{\ell_1,...,\ell_m\}$ with $\ell_1<\cdots <\ell_m$.
            If $1\not\in S$ then $P_S$ is of the form
                \begin{align}\label{eqn:PS_form_1}
                  X_{\ell_1}^k+\cdots +X_{\ell_{m}}^k + \sum_{\substack{2\leq \ell_i<\ell_j\leq n,\\\ell_i,\ell_j\in S}} X_{\ell_i}X_{\ell_j}^{k-1}.
                \end{align}
            If $1\in S$ and $S\setminus \{1\}\subseteq \{r+1,...,n\}$ then $P_S$ is of the form
                \begin{align*} 
                   X_1^k + X_{\ell_2}^k+\cdots +X_{\ell_{m}}^k
                   +\sum_{\substack{2\leq \ell_i<\ell_j\leq n,\\\ell_i,\ell_j\in S}} X_{\ell_i}X_{\ell_j}^{k-1}.
                \end{align*}
            This is the sum of $X_1^k$ and a polynomial $\tilde{P}_S$ of the form \eqref{eqn:PS_form_1}, and so it is nonsingular if and only if $\tilde{P}_S$ is nonsingular. (Note that when $|S|=2$ this is diagonal and hence nonsingular.) Lastly if $1\in S$ and $S\setminus \{1\} \not\subseteq \{r+1,...,n\}$ then $P_S$ is of the form
                \begin{align}\label{eqn:PS_form_3}
                   X_1^k + X_{\ell_2}^k+\cdots +X_{\ell_{m}}^k
                   + \sum_{\substack{2\leq \ell_j \leq r,\\\ell_j\in S}} X_1X_{\ell_j}^{k-1}
                   +\sum_{\substack{2\leq \ell_i<\ell_j\leq n,\\\ell_i,\ell_j \in S}} X_{\ell_i}X_{\ell_j}^{k-1}.
                \end{align}
            In particular, if $|S|=2$ then $P_S$ is of the form
                 \begin{align}\label{eqn:PS_form_4}
                  X_1^k + X_{\ell_2}^k
                   +  X_1X_{\ell_2}^{k-1}.
                \end{align}           
            Consequently, in all cases (after relabelling variables) it suffices to check that for each $m\geq m' \geq 2$ and $\alpha \in \{0,1\}$,  
        \[
                        Q(Y_1,...,Y_m) 
                 := Y_1^k +\cdots +Y_{m}^k 
                 + \sum_{2\leq j\leq m'} Y_{1}Y_{j}^{k-1}
                 +\alpha \sum_{\substack{2\leq i<j\leq m}} Y_{i}Y_{j}^{k-1} 
                   \]
                 is nonsingular. (The case $m'=m,\al=1$ corresponds to \eqref{eqn:PS_form_1} with $Y_i=X_{\ell_i}$, the case $m'<m,\al=1$ corresponds to \eqref{eqn:PS_form_3}, and the case $m'=m=2,\al=0$ corresponds to \eqref{eqn:PS_form_4}.)
                 
             Suppose there exists $a=[a_1: \cdots : a_m]\in \mathbb{P}^{m-1}$ that is a simultaneous solution to the system
                \begin{align*}
                    \frac{\partial Q}{\partial Y_1} &= kY_{1}^{k-1} + \sum_{1<j\leq m'} Y_{j}^{k-1} =0,\\
                    \frac{\partial Q}{\partial Y_{\ell}}&=kY_{\ell}^{k-1} +(k-1)Y_1Y_{\ell}^{k-2}+\al (k-1)\sum_{2\leq i < \ell}Y_iY_{\ell}^{k-2}+\al\sum_{\ell<j\leq m} Y_j^{k-1}=0, \ \ 2 \leq \ell \leq m',\\
                    \frac{\partial Q}{\partial Y_{\ell}}&=kY_{\ell}^{k-1} +\al (k-1)\sum_{2\leq i < \ell}Y_iY_{\ell}^{k-2}+\al\sum_{\ell<j\leq m} Y_j^{k-1}=0, \ \ m'+1\leq \ell \leq m.
                \end{align*}
            Since $k$ is odd, $k-1$ is even, and so the vanishing of $\partial Q / \partial Y_{1}$ forces $a_\ell=0$ for $1\leq \ell \leq m'$. If $\alpha=0$, then the vanishing of the partials $\partial Q / \partial Y_{m'+1},...,\partial Q / \partial Y_m$ forces $a_\ell=0$ for $m'+1\leq \ell \leq m$. Otherwise if $\alpha=1$, then the vanishing of  $\partial Q / \partial Y_{m'}$ forces $a_\ell=0$ for $m'+1\leq \ell \leq m$ (by recalling $a_\ell=0$ for $1 \leq \ell \leq m')$. So $[a_1: \cdots :a_m]$ cannot represent a point in $\mathbb{P}^{m-1}$. Thus $Q$ is nonsingular, as needed.

\subsection{A criterion to check if a form is indecomposable} \label{sec:decomposability}
 We will next prove that the examples given above are indecomposable, and in particular, there is no $\mathrm{GL}_n(\Q)$ change of variables that bring them to the shape $X_1^k + Q_k(X_2,\ldots,X_n)$ (e.g. the shape required in previous work \cite{ACP23,EPV22a}). 
 We refer to the results of Harrison \cite{Har75} and Harrison-Pareigis \cite{HarPar88}, who studied the theory of higher-degree forms using the analogous theory for symmetric spaces. (Being decomposable is  also referred to as being of Sebastiani-Thom type in algebraic geometry literature, for example in the literature on Carlson and Griffiths' result \cite{CarGri80} that the generic polynomial can be reconstructed (up to a constant multiplicative factor) from its Jacobian ideal; see the explicit relation to ST-type in \cite{Wan15}.)

 Let $L$ be a field, in our case  with characteristic zero. A symmetric space of degree $k$ over $L$ is a pair $(V,\theta)$ where $V$ is a vector space over $L$ of dimension $n$ and $\theta:V^k \rightarrow L$ is a symmetric multilinear map. This is equivalent to $\mathrm{Sym}^k(V^*)$, the $k$th symmetric power of $V^*$, which is naturally identified with the space of homogeneous polynomials of $n$ variables and degree $k$ (see standard texts such as \cite{DF03},\cite{Har92}). To describe explicitly the identification between forms and symmetric spaces, write a homogeneous polynomial $F\in L[X_1,...,X_n]$ in the following symmetric form,
   \[
            F(X_1,...,X_n) = \sum_{1\leq i_1,...,i_k \leq n} c_{i_1\cdots i_k}X_{i_1}\cdots X_{i_k},
   \]
    where $c_{i_1\cdots i_k} = c_{\sigma(i_1)\cdots \sigma(i_k)}$ for all $\sigma \in S_k$ and $S_k$ is the symmetric group on $\{1,...,k\}$. 
    Let $V$ be an $n$-dimensional vector space over the field $L$ (we may view $V$ as $L^n$), and let $v_1,...,v_n$ be a basis of $V$. Define $\theta(v_{i_1},...,v_{i_k})= c_{i_1\cdots i_k}.$ Then for all $x_1,...,x_n\in L$, we have the relation
 \[
            F(x_1,...,x_n) = \theta (\sum_{i=1}^n v_i x_i, ..., \sum_{i=1}^n v_i x_i).
  \]
 
  By definition, a symmetric space $(V,\theta)$ is nondegenerate if $\theta(v,v_2,...,v_k)=0$ for all $v_2,...,v_k\in V$ implies $v=0$. Further, a symmetric space is decomposable if there exist nonzero symmetric spaces $(U,\phi)$ and $(W,\psi)$ such that $(V,\theta) = (U,\phi) \oplus (W,\psi)$. 
 Harrison showed that the decomposability of a symmetric space $(V,\theta)$ is characterized by its center $Z(V,\theta)$ which is defined as
\[
        Z(V,\theta) = \{f \in \mathrm{End}_L(V): \theta(f(v_1),v_2,...,v_k) = \theta(v_1,f(v_2),v_3,...,v_k)\}.
   \]
 Note that this also implies $ \theta(v_1,...,f(v_i),...,v_j,...,v_k) = \theta(v_1,...,v_i,...,f(v_j),...,v_k)$ for any $i,j$ since $\theta$ is symmetric.
Precisely, Harrison proves in \cite[Proposition 4.1]{Har75}: let $(V,\theta)$ be a nondegenerate symmetric space of degree $k\geq 3$ over a field $L$ of characteristic zero. Then $(V,\theta)$ is indecomposable if and only if $Z(V,\theta)$ has no nontrivial idempotents.  (An idempotent in a ring is an element $a$ such that $a^2=a$. The trivial idempotent elements are the 0 and 1, respectively the additive identity and the multiplicative identity in the ring.)

To translate this into the language of homogeneous polynomials, we define the center of a form $F$, following \cite{HarPar88}, as
   \[
        Z(F) = \{A\in M_{n}(L) | A^TH_F= H_FA\}
   \]
where $H_F$ denotes the Hessian matrix $(\partial^2 F / \partial X_i \partial X_j)_{1\leq i, j \leq n}$. Then the center, and the decomposability, of a form coincide with those of its associated symmetric space. Precisely, let $F\in L[X_1,...,X_n]$ be a form of degree $k\geq 3$ and let $(V,\theta)$ be the symmetric space associated to $F$. Then it can be shown that  $Z(F)\cong Z(V,\theta)$, and
        $F$ is indecomposable as a form if and only if $(V,\theta)$ is indecomposable as a symmetric space. 
Thus to show a form is indecomposable, it is equivalent to show that its center has no nontrivial idempotents. This is the criterion we will exploit.

It is convenient to note that over a field $L$ (with $\mathrm{char} L \ndiv k$), a form is called central if $Z(F) \simeq L.$ Harrison showed that if a form is central (over $L$), then it is absolutely indecomposable (that is, indecomposable over any field extension of $L$). 
 In our case,   to show $F$ is indecomposable over $\Q$ it suffices to show that $Z(F)\simeq \Q,$ so the form has the even stronger property of being central. We remarked earlier that 
 indecomposable forms are generic. This  is implied for $(n,k) \neq (2,3)$ over $\C$ by \cite[Thm. 3.2]{HLYZ21}  (which shows the set of central forms is open and dense in the moduli space over $\C$; this proof can be adapted to hold over $\Q$). It is also shown directly for $n \geq 3, k \geq 3$ over $\C$ by \cite[Ex. 4.3, 4.4, Cor. 6.1]{Wan15}. The case $(n,k)=(2,3)$ is more complicated, and we defer its study to a different work.

 \subsection{The examples are indecomposable}
 For $k \geq 3$, $n \geq 2$ and $2 \leq r \leq n,$ the example form $P_k$ defined above is indecomposable; we will prove this next by showing $Z(P_k)$ has no nontrivial idempotents.
Precisely, when $(n,k) \neq (2,3)$, we    show that $Z(P_k)\simeq \Q;$ for $(n,k)=(2,3)$, the example $P_3$ is also indecomposable, but with a different center.
(Note that if $r=1$, the form is decomposable since $X_1$ only appears in a diagonal term. Thus we need only consider  $2 \leq r \leq n$.)
 
We present the full proof for all odd $k \geq 5$, $n \geq 2$ and $2 \leq r \leq n$; the proof for $k=3$ and for all even $k \geq 4$ is fundamentally analogous. 
   Fix $P=P_k$ to be the example form  defined above.
    Let $H_P$ denote the Hessian of $P$. Then $H_P/(k-1)$ is the $n\times n$ matrix
        \begin{align*}
        \begin{pmatrix}
        kX_1^{k-2}  &X_2^{k-2} &\cdots & X_r^{k-2}&0&\cdots&0\\
        X_2^{k-2}   &kX_2^{k-2}+Q_2& \cdots& X_r^{k-2}&X_{r+1}^{k-2}&\cdots &X_n^{k-2}\\
        \vdots  &&\ddots&&&&\vdots\\
        X_r^{k-2}&X_r^{k-2}&\cdots &kX_r^{k-2}+Q_r&X_{r+1}^{k-2}&\cdots&X_n^{k-2}\\
        0&X_{r+1}^{k-2}&\cdots&X_{r+1}^{k-2}&kX_{r+1}^{k-2}+Q_{r+1}&\cdots&X_{n}^{k-2}\\
        \vdots&&&&&\ddots&\vdots\\
        0& X_n^{k-2}&\cdots&X_n^{k-2}&X_n^{k-2}&\cdots&kX_n^{k-2}+Q_n
        \end{pmatrix}
        \end{align*} 
        where $Q_\ell = (k-2)\sum_{1\leq i < \ell}X_iX_\ell^{k-3}$ for $\ell\leq r$ and $Q_\ell = (k-2)\sum_{2\leq i < \ell}X_iX_\ell^{k-3}$ for $\ell> r$. Note that each $Q_\ell \not\con 0$. For $k\geq 5$, each $Q_\ell$ consists of monomials that differ from each other and from the entries of $H_P/(k-1)$ that are off the diagonal, and so the vanishing of any linear combination $c_{i_1}Q_{i_1}+\cdots + c_{i_m}Q_{i_m}\con 0$ would imply $c_{i_j}=0$ for all $1\leq j \leq m$.

Let $A=(a_{ij})\in M_{n}(\Q)$ and write $H_P/(k-1)=(h_{ij})$. Let $B_A$ denote $(A^TH_P-H_PA)/(k-1)$. Then $A\in Z(P)$ if and only if $B_A=0$. Note that a priori we have $\{cI_n:c\in\Q\}\subseteq Z(P)$. The assumption that all entries of $B_A$ are 0 implies constraints on the entries of $A$ that show the reverse inclusion $Z(P)\subseteq \{cI_n:c\in\Q\}$, from which we deduce the equality $Z(P) = \{cI_n:c\in \Q\} \simeq \Q$.
 
Write $B_A=(b_{ij})$ so that $b_{ij} = \sum_{\ell=1}^n (a_{\ell i}h_{\ell j} - h_{i\ell}a_{\ell j})$. Note that since $H_P$ is symmetric, $b_{ii}=0$ and $B_A^T=-B_A$. Hence it suffices to consider the $(n^2-n)/2$ entries above the diagonal i.e. $b_{ij}$ with $i<j$. Each entry $b_{ij}$ is a polynomial, so it is $\con 0$ if and only if the coefficient of each term (after regrouping) is zero. We split the  $(n^2-n)/2$ entries of $b_{ij}$ with $i<j$ into the following five cases of $(i,j)$, based on the shapes of the rows and columns:
    \begin{enumerate}
        \item $(1,j)$ with $2\leq j \leq r$, 
        \item $(i,j)$ with $2\leq i<j\leq r$,
        \item $(1,j)$ with $r+1\leq j \leq n$,
        \item $(i,j)$ with $r+1\leq j \leq n$ and $2\leq i \leq r$, and
        \item $(i,j)$ with $r+1\leq i< j \leq n$.
    \end{enumerate}
It suffices to learn from the assumption that $b_{ij} \con 0$ in the cases (1),(3),(4).  

For case (1) with $(1,j)$ with $2\leq j \leq r$, we compute
        \begin{multline*}
            b_{1j}
            = -a_{1j}kX_1^{k-2}
            -\sum_{\substack{\ell=2}}^{j-1} a_{\ell j}X_\ell^{k-2}+(\sum_{\ell=1}^{j-1} a_{\ell 1} + a_{j1}k- a_{jj})X_j^{k-2}\\
            + \sum_{\ell=j+1}^r( a_{\ell 1}-a_{\ell j})X_\ell^{k-2}  +  \sum_{\ell=r+1}^n a_{\ell 1}X_\ell^{k-2}+a_{j1}Q_j.
        \end{multline*}
 The  assumption $b_{1j} \con 0$  for all $2\leq j \leq r$  can be seen to imply that
        \begin{align}
        a_{ii}&=a_{11}, \qquad 2 \leq i \leq r,\label{eqn:cond_1_1}\\
        a_{i1}&=0, \qquad  2 \leq i \leq n \label{eqn:cond_1_3},\\
        a_{ij}&=0, \qquad 1 \leq i \neq j\leq r.\nonumber
        \end{align}
  These give the desired conditions on the top left $r\times r$ block and the first column of $A$. If $r=n$, the proof is complete; otherwise for $r<n$ we continue, as cases (3), (4) are non-vacuous.
    
For case (3), $(1,j)$ with $r+1\leq j \leq n$, we compute
\[
    b_{1j} 
         =- a_{1j}kX_1^{k-2}
            -\sum_{\substack{\ell=2}}^{r} a_{\ell j}X_\ell^{k-2}
            +(\sum_{\ell=2}^{j-1} a_{\ell 1} + a_{j1}k)X_j^{k-2}
            +\sum_{\ell=j+1}^n a_{\ell 1}X_\ell^{k-2}
            +a_{j1}Q_j.
   \]
Thus the assumption $b_{1j}\con 0$ gives in particular the new condition
   \[
        a_{i j}=0, \qquad 1 \leq i \leq r, \quad r+1 \leq j \leq n. 
\]
This is the desired result for  the top right $r\times (n-r)$ block of $A$.

For case (4), $(i,j)$ with $r+1\leq j \leq n$ and $2\leq i \leq r$, we compute
    \begin{multline*}
        b_{ij}  
            =-(\sum_{\ell=1}^{i-1}a_{\ell j} + a_{ij}k)X_i^{k-2}
            -\sum_{\ell=i+1}^{j-1} a_{\ell j}X_\ell^{k-2}
            +(\sum_{\ell=2}^{j-1} a_{\ell i}+a_{ji}k-a_{jj})X_j^{k-2}
            \\
            +\sum_{\ell=j+1}^n (a_{ \ell i}-a_{\ell j})X_\ell^{k-2}
            -a_{ij}Q_i + a_{ji}Q_j.
    \end{multline*}
The assumption $b_{ij} \con 0$ for all $r+1\leq j \leq n$ and $2\leq i \leq r$
  implies that
    \begin{align*}
        a_{ii} &=a_{22}, \qquad 2 \leq i \leq n,\\
        a_{ij} &= 0 , \qquad r+1\leq i \leq n, \quad 2 \leq j \leq r\\
        a_{ij} &= 0 , \qquad 2 \leq i \leq n, \quad r+1\leq j \leq n, \quad i\neq j.
    \end{align*}    
The second condition combined with (\ref{eqn:cond_1_3}) confirms that the entries in the lower-left $(n-r)\times r$ block of $A$ are zeroes, while the third condition finalizes that the off-diagonal entries of the lower-right $(n-r)\times (n-r)$ block of $A$ are zeroes. Finally, combined with (\ref{eqn:cond_1_1}) the first condition shows that all diagonal entries in the lower-right block are also equal to $a_{11}$. Thus $A= a_{11}I_n$, and this completes the proof that $Z(P)\subseteq \{cI_n:c\in \Q\}$.

The above computations focused on the case of $k \geq 5$ odd. If $k\geq 4$ is even, the argument follows exactly the same structure; the assumption that $b_{ij} \con 0$ for indices in case (1) prove the top left $r \times r$ block has the desired structure, and (if $r<n$) indices from cases (3) and (4) complete the information about the remaining matrix. If $k=3$, the proof is  more complicated, because the polynomials $Q_\ell$ defined above now must be grouped with various other terms (the monomials appearing are no longer all distinct). Nevertheless, if $n \geq 3$, the assumption $b_{ij} \con 0$ in the cases (1)--(4) shows that $Z(P_3)\simeq \Q$. If $(n,k)=(2,3)$, we need only consider rank $r=2$, and only the index case (1) is non-vacuous. From $b_{ij}\con 0$ in case (1) we conclude
 $3a_{12}-a_{21}=0$ and $a_{11}-a_{22}+3a_{21}=0$ so that  
 \[
        Z(P_3) = \left\{\begin{pmatrix}
        \al - 9\be  &\be\\
        3\be &\al
        \end{pmatrix}: \al,\be\in \Q
        \right\}.
\]
Any  idempotent $A \in Z(P_3)$ must satisfy $A^2=A,$ by definition of being an idempotent. If $\be=0$, this forces $\al=0$ or $1$, corresponding to the $A$ being either of the trivial idempotents (the zero matrix or the identity matrix). If $\be \neq 0$, the identity $A^2=A$ produces three independent equations in $\al,\be$, and in particular, inspection of these equations shows that $\al$ must satisfy a quadratic equation with no rational roots. Thus $Z(P_3)$ contains no nontrivial idempotents over $\Q$. In conclusion, $P_3$ is not central over $\Q$ but it is indecomposable over $\Q$. (It is incidentally decomposable over $\overline{\Q}$, since its center contains nontrivial idempotents over $\overline{\Q}$.)

\subsection{Codimension of the class of forms}\label{sec_codim}
Let $\Mcal$ denote the moduli space of   forms $P\in \Q[X_1,...,X_n]$ of degree $k \geq 2$.  Then $\dim \Mcal= \binom{n+k-1}{n-1}$. To see this by a ``stars and bars'' argument, note that $\dim \Mcal$ is the number of monomials of degree $k$ in   $n$ variables. Each such monomial can be represented as a configuration of $k$ stars and $n-1$ bars (e.g. $X_1X_2^2$ when $k=3$ and $n=4$ would be represented by the configuration $*|**||.$) The number of such configurations is equal to the number of ways to choose the location of the $n-1$ bars (among $k+n-1$ possible places), and this is the binomial coefficient $\binom{n+k-1}{n-1}$.  

Let $\Dscr\subseteq \Mcal$ denote the set of Dwork-regular forms, and let $\Pscr_r$ denote the set of forms of intertwining rank $\leq r$, for $1\leq r \leq n$, so that
        $\Pscr_1\subseteq \Pscr_2 \subseteq \cdots \subseteq \Pscr_n = \Mcal.$
(We remark that the set of forms that have  intertwining rank precisely $r$, that is, the set $\tilde{\Pscr}_r=\Pscr_r\setminus\Pscr_{r-1}$, has dimension equal to that of $\Pscr_r$, since it can be shown that $\dim \Pscr_{r-1}<\dim \Pscr_r$.)
For each $1 \leq r <n,$ Theorem \ref{thm:main2} proves a nontrivial result for all real symbols $P$ with leading form $P_k \in \Dscr \intersect \Pscr_r.$  
 Here we compare the codimension of  $\Dscr \intersect \Pscr_{n-1}$ in $\mathcal{M}$ (the largest class to which our theorem applies) to the codimension of $\Dscr \intersect \Pscr_{1}$ in $\mathcal{M}$ (equivalent to the largest class to which the previous works \cite{ACP23,EPV22a} applied); we focus on a brief summary, since a more complete study of such forms (and their $\mathrm{GL}_n(\Q)$-orbits) will be given in other work.

Fix $1 \leq r \leq n$. First, $\Dscr$ is open in $\mathcal{M}$, and it can be shown that $\Pscr_r$ is a finite union of (irreducible) affine varieties  in $\mathcal{M}$, say $P_{r,i}$  for $i=1,\ldots, N$, so that $\Pscr_r=\cup_{i=1}^N P_{r,i}$. Then $\Dscr \intersect \Pscr_r$ is a quasi-affine variety, and consequently by \cite[Prop 1.10]{Har77}, $\dim \Dscr \intersect \Pscr_r = \dim \overline{\Dscr \intersect \Pscr_r}$. 
In general, if  $U\subseteq \A^m$ is an open set and   $X=\cup_{i=1}^N X_i\subseteq \A^m$ is a union of irreducible affine varieties $X_i$, then as long as $U\cap X_i\neq \emptyset$ for every $i$, it follows that $\overline{U\cap X} = X$.  For each $i$, it can be shown that $\Dscr \intersect P_{r,i}$ is nonempty, by precisely the examples stated above (up to re-ordering coordinates). Thus we conclude that $\overline{\Dscr \intersect \Pscr_r} = \Pscr_r$, so that it suffices to compute $\dim \Pscr_r$ in $\mathcal{M}$.

We focus on the cases of $\Pscr_1$ and $\Pscr_{n-1};$ it is easier to count the codimension.
By symmetry considerations, the dimension of $\Pscr_{n-1}$ is the dimension of the class of forms for which $X_1$ does not intertwine with $X_n$, or equivalently all those forms that do not contain monomials with the factor $X_1X_n.$
Equivalently,   all coefficients of terms of the form $X_1X_nQ_{k-2}(X_1,\ldots,X_n)$ with $Q_{k-2}$ of degree $k-2$, must be zero. This constrains the coefficients of $\binom{n+(k-2)-1}{n-1}$ monomials, so that 
\[
        \codim (\mathscr{P}_{n-1}) = \binom{n+k-3}{n-1}.
\]
On the other hand, by symmetry considerations, the dimension of $\Pscr_1$ is the dimension of the class of forms $cX_1^k+Q_k(X_2,...,X_n)$, where $Q_k$ has degree $k$; the dimension of such polynomials is $\binom{n+k-2}{n-2}+1$. Equivalently,
\[
        \codim (\mathscr{P}_{1}) = \binom{n+k-1}{n-1} - \binom{n+k-2}{n-2}-1 = \codim (\mathscr{P}_{n-1}) + \binom{n+k-3}{n-2}-1.
\]
Thus $\Pscr_{n-1}$ is a larger class of forms, with $\dim (\mathscr{P}_{n-1}) - \dim (\mathscr{P}_{1})  =\binom{n+k-3}{n-2}-1$ behaving asymptotically like $\sim c_nk^{n-2}$ if $n$ is fixed and $k \maps \infty$, or like $\sim c_kn^{k-1}$ if $k$ is fixed and $n \maps \infty.$

 \section*{Acknowledgements}
 The first author  has been partially supported during portions of this research by NSF DMS-2200470. The second author   has been partially supported during portions of this research by NSF DMS-2200470, NSF
CAREER grant DMS-1652173 and a Joan and Joseph Birman Fellowship, and thanks the Hausdorff Center for Mathematics for productive research visits as a Bonn Research Chair. We thank the anonymous referees for  helpful questions and comments, and the second author thanks Matilde Lal\'{i}n for helpful conversations.

\bibliographystyle{alpha}

\bibliography{NoThBibliography}

\begin{thebibliography}{DGLZ18}

\bibitem[ACP23]{ACP23}
C.~An, R.~Chu, and L.~B. Pierce.
\newblock Counterexamples for high-degree generalizations of the {S}chrödinger
  maximal operator.
\newblock {\em IMRN}, 2023(10):8371--8418, 2023.

\bibitem[AH95]{AleHir95}
J.~Alexander and A.~Hirschowitz.
\newblock Polynomial interpolation in several variables.
\newblock {\em J. Algebraic Geom.}, 4(2):201--222, 1995.

\bibitem[BAD91]{BenDev91}
M.~Ben-Artzi and A.~Devinatz.
\newblock Local smoothing and convergence properties of {S}chr\"{o}dinger type
  equations.
\newblock {\em J. Funct. Anal.}, 101(2):231--254, 1991.

\bibitem[BCLP22]{BCLP22}
A.~Bucur, A.~C. Cojocaru, M.~N. Lal\'{i}n, and L.~B. Pierce.
\newblock Geometric generalizations of the square sieve, with an application to
  cyclic covers.
\newblock {\em Mathematika}, 69(1):106--154, 2022.

\bibitem[Bou95]{Bou95}
J.~Bourgain.
\newblock Some new estimates on oscillatory integrals.
\newblock In {\em Essays on {F}ourier analysis in honor of {E}lias {M}. {S}tein
  ({P}rinceton, {NJ}, 1991)}, volume~42 of {\em Princeton Math. Ser.}, pages
  83--112. Princeton Univ. Press, Princeton, NJ, 1995.

\bibitem[Bou13]{Bou13}
J.~Bourgain.
\newblock On the {S}chr\"{o}dinger maximal function in higher dimensions.
\newblock {\em Tr. Mat. Inst. Steklova}, 280:53--66, 2013.

\bibitem[Bou16]{Bou16}
J.~Bourgain.
\newblock A note on the {S}chr\"{o}dinger maximal function.
\newblock {\em J. Anal. Math.}, 130:393--396, 2016.

\bibitem[Car80]{Car80}
L.~Carleson.
\newblock Some analytic problems related to statistical mechanics.
\newblock In {\em Euclidean Harmonic Analysis ({P}roc. {S}em., {U}niv.
  {M}aryland, {C}ollege {P}ark, {M}d., 1979)}, volume 779 of {\em Lecture Notes
  in Math.}, pages 5--45. Springer, Berlin, 1980.

\bibitem[Car85]{Car85}
A.~Carbery.
\newblock Radial {F}ourier multipliers and associated maximal functions.
\newblock In {\em Recent progress in {F}ourier analysis ({E}l {E}scorial,
  1983)}, volume 111 of {\em North-Holland Math. Stud.}, pages 49--56.
  North-Holland, Amsterdam, 1985.

\bibitem[CCG12]{CCG12}
E.~Carlini, M.~V. Catalisano, and A.~V. Geramita.
\newblock The solution to the {W}aring problem for monomials and the sum of
  coprime monomials.
\newblock {\em J. Algebra}, 370:5--14, 2012.

\bibitem[CG80]{CarGri80}
J.~A. Carlson and P.~A. Griffiths.
\newblock Infinitesimal variations of {H}odge structure and the global
  {T}orelli problem.
\newblock In {\em Journ\'{e}es de {G}\'{e}ometrie {A}lg\'{e}brique d'{A}ngers,
  {J}uillet 1979/{A}lgebraic {G}eometry, {A}ngers, 1979}, pages 51--76.
  Sijthoff \& Noordhoff, Alphen aan den Rijn---Germantown, Md., 1980.

\bibitem[Cow83]{Cow83}
M.~G. Cowling.
\newblock Pointwise behavior of solutions to {S}chr\"{o}dinger equations.
\newblock In {\em Harmonic analysis ({C}ortona, 1982)}, volume 992 of {\em
  Lecture Notes in Math.}, pages 83--90. Springer, Berlin, 1983.

\bibitem[Del74]{Del74}
P.~Deligne.
\newblock La conjecture de {W}eil: I.
\newblock {\em Publications Math\'ematiques de l'IH\'ES}, 43:273--307, 1974.

\bibitem[DF03]{DF03}
D.~S. Dummit and R.~M. Foote.
\newblock {\em Abstract Algebra}.
\newblock Wiley, 2003.

\bibitem[DGL17]{DGL17}
X.~Du, L.~Guth, and X.~Li.
\newblock A sharp {S}chr\"odinger maximal estimate in ${\R^2}$.
\newblock {\em Ann. of Math.}, 186:607--640, 2017.

\bibitem[DGLZ18]{DGLZ18}
X.~Du, L.~Guth, X.~Li, and R.~Zhang.
\newblock Pointwise convergence of {S}chr\"{o}dinger solutions and multilinear
  refined {S}trichartz estimates.
\newblock {\em Forum Math. Sigma}, 6:e14, 18, 2018.

\bibitem[DK82]{DahKen82}
B.~E.~J. Dahlberg and C.~E. Kenig.
\newblock A note on the almost everywhere behavior of solutions to the
  {S}chr\"{o}dinger equation.
\newblock In {\em Harmonic analysis ({M}inneapolis, {M}inn., 1981)}, volume 908
  of {\em Lecture Notes in Math.}, pages 205--209. Springer, Berlin-New York,
  1982.

\bibitem[Dwo62]{Dwo62}
B.~Dwork.
\newblock On the zeta function of a hypersurface.
\newblock {\em Inst. Hautes \'Etudes Sci. Publ. Math.}, (12):5--68, 1962.

\bibitem[DZ19]{DuZha19}
X.~Du and R.~Zhang.
\newblock Sharp {$L^2$} estimates of the {S}chr\"{o}dinger maximal function in
  higher dimensions.
\newblock {\em Ann. of Math. (2)}, 189(3):837--861, 2019.

\bibitem[EPV22]{EPV22a}
D.~Eceizabarrena and F.~Ponce-Vanegas.
\newblock Pointwise convergence over fractals for dispersive equations with
  homogeneous symbol.
\newblock {\em J. Math. Anal. Appl.}, 515(1):Paper No. 126385, 57, 2022.

\bibitem[Har75]{Har75}
D.~K. Harrison.
\newblock A {G}rothendieck ring of higher degree forms.
\newblock {\em Journal of Algebra}, 35:123--138, 1975.

\bibitem[Har77]{Har77}
R.~Hartshorne.
\newblock {\em Algebraic Geometry}.
\newblock Graduate Texts in Mathematics. Springer, 1977.

\bibitem[Har92]{Har92}
J.~Harris.
\newblock {\em Algebraic Geometry: A First Course}.
\newblock Graduate Texts in Mathematics. Springer, 1992.

\bibitem[HLYZ21]{HLYZ21}
H.-L. Huang, H.~Lu, Y.~Ye, and C.~Zhang.
\newblock On centres and direct sum decompositions of higher degree forms.
\newblock {\em Linear and Multilinear Algebra}, 0(0):1--17, 2021.

\bibitem[HM22]{HanMoo22}
K.~Han and H.~Moon.
\newblock A new bound for the {W}aring rank of monomials.
\newblock {\em SIAM J. Appl. Algebra Geom.}, 6(3):407--431, 2022.

\bibitem[HP88]{HarPar88}
D.~K. Harrison and B.~Pareigis.
\newblock Witt rings of higher degree forms.
\newblock {\em Communications in Algebra}, 16(6):1275--1313, 1988.

\bibitem[IK04]{IwaKow04}
H.~Iwaniec and E.~Kowalski.
\newblock {\em Analytic {N}umber {T}heory}, volume~53.
\newblock Amer. Math. Soc. Colloquium Publications, Providence RI, 2004.

\bibitem[Kat08]{Kat08}
N.~M. Katz.
\newblock Estimates for mixed character sums.
\newblock {\em Geom. funct. anal.}, 18:1251--1269, 2008.

\bibitem[Kat09]{Kat09}
N.~M. Katz.
\newblock On a question of {B}rowning and {H}eath-{B}rown.
\newblock In {\em Analytic number theory}, pages 267--288. Cambridge Univ.
  Press, Cambridge, 2009.

\bibitem[KPV91]{KPV91}
C.~E. Kenig, G.~Ponce, and L.~Vega.
\newblock Oscillatory integrals and regularity of dispersive equations.
\newblock {\em Indiana Univ. Math. J.}, 40(1):33--69, 1991.

\bibitem[KR83]{KenRui83}
C.~E. Kenig and A.~Ruiz.
\newblock A strong type {$(2,\,2)$} estimate for a maximal operator associated
  to the {S}chr\"{o}dinger equation.
\newblock {\em Trans. Amer. Math. Soc.}, 280(1):239--246, 1983.

\bibitem[Lan02]{Lan02}
S.~Lang.
\newblock {\em Algebra}.
\newblock Springer, New York, NY, 2002.

\bibitem[Lee06]{Lee06}
S.~Lee.
\newblock On pointwise convergence of the solutions to {S}chr\"{o}dinger
  equations in {$\mathbb{R}^2$}.
\newblock {\em Int. Math. Res. Not.}, pages Art. ID 32597, 21, 2006.

\bibitem[LR17]{LucRog17}
R.~Luc\`a and K.~M. Rogers.
\newblock Coherence on fractals versus pointwise convergence for the
  {S}chr\"{o}dinger equation.
\newblock {\em Comm. Math. Phys.}, 351(1):341--359, 2017.

\bibitem[LR19]{LucRog19}
R.~Luc\`a and K.~M. Rogers.
\newblock A note on pointwise convergence for the {S}chr\"{o}dinger equation.
\newblock {\em Math. Proc. Cambridge Philos. Soc.}, 166(2):209--218, 2019.

\bibitem[LZ21]{LamZie21x}
A.~Lampert and T.~Ziegler.
\newblock Relative rank and regularization.
\newblock {\em arXiv:2106.03933}, 2021.

\bibitem[LZ22]{LamZie22x}
A.~Lampert and T.~Ziegler.
\newblock Schmidt rank and algebraic closure.
\newblock {\em arXiv:2205.05329}, 2022.

\bibitem[MVV96]{MVV96}
A.~Moyua, A.~Vargas, and L.~Vega.
\newblock Schr\"{o}dinger maximal function and restriction properties of the
  {F}ourier transform.
\newblock {\em Internat. Math. Res. Notices}, (16):793--815, 1996.

\bibitem[OS03]{ORySha03}
M.~O'Ryan and D.~B. Shapiro.
\newblock Centers of higher degree forms.
\newblock {\em Linear Algebra Appl.}, 371:301--314, 2003.

\bibitem[Pal97]{Pal97}
R.~S. Palais.
\newblock The symmetries of solitons.
\newblock {\em Bull. Amer. Math. Soc. (N.S.)}, 34(4):339--403, 1997.

\bibitem[Pie20]{Pie20}
L.~B. Pierce.
\newblock On {B}ourgain's counterexample for the {S}chr\"odinger maximal
  function.
\newblock {\em Quart. J. Math.}, 71:1309--1344, 2020.

\bibitem[Pum06]{Pum06}
S.~Pumpl\"{u}n.
\newblock Indecomposable forms of higher degree.
\newblock {\em Math. Z.}, 253(2):347--360, 2006.

\bibitem[Rez13]{Rez13}
B.~Reznick.
\newblock On the length of binary forms.
\newblock In {\em Quadratic and higher degree forms}, volume~31 of {\em Dev.
  Math.}, pages 207--232. Springer, New York, 2013.

\bibitem[RS00]{RanSch00}
K.~Ranestad and F.-O. Schreyer.
\newblock Varieties of sums of powers.
\newblock {\em J. Reine Angew. Math.}, 525:147--181, 2000.

\bibitem[RVV06]{RVV06}
K.~M. Rogers, A.~Vargas, and L.~Vega.
\newblock Pointwise convergence of solutions to the nonelliptic
  {S}chr\"{o}dinger equation.
\newblock {\em Indiana Univ. Math. J.}, 55(6):1893--1906, 2006.

\bibitem[Sj{\"{o}}87]{Sjo87}
P.~Sj{\"{o}}lin.
\newblock Regularity of solutions to the {S}chr\"{o}dinger equation.
\newblock {\em Duke Math. J.}, 55(3):699--715, 1987.

\bibitem[Sj{\"{o}}98]{Sjo98}
P.~Sj{\"{o}}lin.
\newblock A counter-example concerning maximal estimates for solutions to
  equations of {S}chr\"{o}dinger type.
\newblock {\em Indiana Univ. Math. J.}, 47(2):593--599, 1998.

\bibitem[Tao06]{Tao06}
T.~Tao.
\newblock {\em Nonlinear dispersive equations}, volume 106 of {\em CBMS
  Regional Conference Series in Mathematics}.
\newblock Published for the Conference Board of the Mathematical Sciences,
  Washington, DC; by the American Mathematical Society, Providence, RI, 2006.

\bibitem[TV00]{TaoVar00}
T.~Tao and A.~Vargas.
\newblock A bilinear approach to cone multipliers. {I}. {R}estriction
  estimates.
\newblock {\em Geom. Funct. Anal.}, 10(1):185--215, 2000.

\bibitem[Veg88]{Veg88b}
L.~Vega.
\newblock {\em {El multiplicador de Schr\"odinger. La funcion maximal y los
  operadores de restricci\'on}}.
\newblock Universidad Aut\'onoma de Madrid, 1988.

\bibitem[Wan15]{Wan15}
Z.~Wang.
\newblock On homogeneous polynomials determined by their {J}acobian ideal.
\newblock {\em Manuscripta Math.}, 146(3-4):559--574, 2015.

\end{thebibliography}

\end{document}